\documentclass[11pt]{article}
\usepackage{graphicx} 
\usepackage{bbm}
\usepackage{bm}
\usepackage[margin=1in]{geometry}
\usepackage{color}
\usepackage{amsfonts}
\usepackage{algorithm}
\usepackage{algorithmic}
\usepackage{latexsym, amssymb, amsmath, amscd, amsthm, amsxtra}
\usepackage{mathtools}
\usepackage{enumerate}
\usepackage[all]{xy}
\usepackage{mathrsfs}
\usepackage{fancyhdr}
\usepackage{listings}
\usepackage{hyperref}
\usepackage{enumitem}
\usepackage{multirow}
\usepackage{tikz}

\def\11{\mathbbm{1}}

\def\ER{Erd\H{o}s-R\'enyi\ }

\newcommand{\Pb}{\mathbb P}
\newcommand{\Qb}{\mathbb Q}

\newtheorem{thm}{Theorem}[section]

\newtheorem{proposition}[thm]{Proposition}

\newtheorem{lemma}[thm]{Lemma}
\newtheorem{cor}[thm]{Corollary}

\newtheorem{remark}[thm]{Remark}
\newtheorem{DEF}[thm]{Definition}

\numberwithin{equation}{section}

\makeatletter
\newenvironment{breakablealgorithm}
{
		\begin{center}
			\refstepcounter{algorithm}
			\hrule height.8pt depth0pt \kern2pt
			\renewcommand{\caption}[2][\relax]{
				{\raggedright\textbf{\ALG@name~\thealgorithm} ##2\par}%
				\ifx\relax##1\relax 
				\addcontentsline{loa}{algorithm}{\protect\numberline{\thealgorithm}##2}%
				\else 
				\addcontentsline{loa}{algorithm}{\protect\numberline{\thealgorithm}##1}%
				\fi
				\kern2pt\hrule\kern2pt
			}
		}{
		\kern2pt\hrule\relax
	\end{center}
}
\makeatother

\hypersetup{colorlinks=true,linkcolor=blue}

\title{A Smooth Computational Transition in Tensor PCA}

\usepackage{authblk}
\author[1]{Zhangsong Li\thanks{Email: \textit{ramblerlzs@pku.edu.cn}. Partially supported by the National Key R$\&$D program of China (Project No. 2023YFA1010103) and the NSFC Key Program (No. 12231002).}}

\affil[1]{School of Mathematical Sciences, Peking University}

\date{\today}
\begin{document}
\maketitle

\begin{abstract}
    We propose an efficient algorithm for tensor PCA based on counting a specific family of weighted hypergraphs. For the order-$p$ tensor PCA problem where $p \geq 3$ is a fixed integer, we show that when the signal-to-noise ratio is $\lambda n^{-\frac{p}{4}}$ where $\lambda=\Omega(1)$, our algorithm succeeds and runs in time $n^{C+o(1)}$ where $C=C(\lambda)$ is a constant depending on $\lambda$. This algorithm improves a poly-logarithmic factor compared to previous algorithms based on the Sum-of-Squares hierarchy \cite{HSS15} or based on the Kikuchi hierarchy in statistical physics \cite{WAM19}. Furthermore, our result shows a smooth tradeoff between the signal-to-noise ratio and the computational cost in this problem, thereby confirming a conjecture posed in \cite{KWB22}. 
\end{abstract}

\tableofcontents

\section{Introduction}{\label{sec:intro}}

Principal component analysis (PCA), which identifies a direction of maximum variance from a matrix of pairwise correlations, is among the most basic tools for data analysis in a wide range of disciplines. In recent years, variants of PCA have been proposed including restricting directions to the nonnegative orthant (nonnegative matrix factorization) or to directions that are sparse linear combinations of a fixed basis (sparse PCA). A further extension arises when data provides access to higher-order correlations beyond pairwise information. In this case, an analog of PCA is to find a direction with largest possible third moment or other higher-order moment (higher-order PCA or tensor PCA).

In this work, we consider the order-$p$ tensor PCA (also known as the spiked tensor) problem. This model was introduced by \cite{RM14} and has received significant attention recently.\footnote{The special case $p=2$ of the spiked tensor model is the so-called spiked Wigner matrix model which has been widely studied in random matrix theory, statistics, information theory, and statistical physics; see \cite{Mio18} for a survey.} In this model, we observe a $p$-fold $n \times n \times \ldots \times n$ tensor
\begin{equation*}
    \bm Y = \kappa x_*^{\otimes p} + \bm G \,,
\end{equation*}
where the parameter $\kappa \geq 0$ is the signal-to-noise ratio (SNR), $x_*$ is a planted signal drawn from the uniform distribution over the hypercube $\{ -1,+1 \}^n$, and $\bm G$ is a symmetric noise tensor with standard normal entries (see \eqref{eq-def-symmetric-tensor} for the precise definition of $\bm G$). 

Two basic problems regarding this model are as follows: (1) the detection problem, i.e., deciding whether $\bm Y$ is sampled from the law of a spiked tensor or is sampled from the law of a pure symmetric noise tensor; (2) the recovery problem, i.e., recovering the planted signal $x_*$ from the observation $\bm Y$. Information-theoretically, it was shown in \cite{RM14, LML+17} that for any fixed integer $p\geq 2$, detection and recovery is possible if $\kappa \gg n^{ \frac{1-p}{2} }$. On the computational side, the Sum-of-Squares hierarchy provides a polynomial-time detection and recovery algorithm when $\lambda \gg n^{-\frac{p}{4}} \cdot \operatorname{polylog}(n)$ \cite{HSS15}. In contrast, the approximate message passing (AMP) algorithm and other ``local searching'' algorithm only succeeds when $\lambda \gg n^{-\frac{1}{2}}$ \cite{RM14, BGJ20}. As a ``redemption'' of the statistical physics approach, it was shown the Kikuchi spectral algorithm based on higher-order Kikuchi free energies also succeeds when $\lambda \gg n^{-\frac{p}{4}} \cdot \operatorname{polylog}(n)$ \cite{WAM19}. Recently, using deep results in random matrix theory, \cite{BCSvH24+} provides a sharper analysis of the Kikuchi algorithm for even $p$ and shows that it succeeds when $\lambda>c_* n^{-\frac{p}{4}}$ where $c_*$ is an explicit constant. In addition, although it is challenging to prove hardness for a typical instance of the problem even under the assumption of P$\neq$NP, evidences from the Sum-of-Squares hierarchy \cite{HSS15, HKP+17}, the low-degree polynomial framework \cite{KWB22, KMW24} and reduction from a variant of the planted clique conjecture \cite{BB20} suggest that detection and recovery is impossible by efficient algorithms when $\kappa \ll n^{-\frac{p}{4}}$. Thus, when $p \geq 3$ this model exhibits a large {\em statistical-computational gap}, and the (conjectured) computational threshold is around $\kappa \approx n^{-\frac{p}{4}}$. In this work, we will focus on the critical regime $\kappa = \Theta(n^{-\frac{p}{4}})$. Our result can be summarized as follows:
\begin{thm}[Informal]{\label{MAIN-THM}}
    Suppose $p \geq 3$ is a fixed integer and suppose $\kappa = \lambda \cdot n^{-\frac{p}{4}}$ where $\lambda=\Omega(1)$ is a non-vanishing constant. Then there exists two constants $C=C(\lambda)$ and $C'=C'(\lambda)$ and two algorithms $\mathcal A,\mathcal A'$ with running time $n^{C+o(1)}$ and $n^{C'+o(1)}$ respectively, such that $\mathcal A$ (respectively, $\mathcal A'$) takes $\bm Y$ as input and achieves strong detection (respectively, weak recovery).
\end{thm}
Our algorithm thus improves a poly-logarithmic factor compared to existing algorithms \cite{HSS15, WAM19} for general $p$. In addition, our results suggest a very interesting phenomenon that there is no sharp threshold for polynomial time algorithms in tensor PCA in the following sense: by increasing $C,C'$ (while remaining $\Theta(1)$, corresponding to increasing the computational cost of the algorithm while keeping it polynomial time), one would be able to decrease the critical signal-to-noise ratio in the sense of $\lambda$ getting arbitrarily close to zero. This confirms the conjecture of a smooth computational transition in \cite{KWB22}\footnote{The authors in \cite{KWB22} made this conjecture from the calculation of low-degree advantage; see Section~\ref{subsec:low-degree-framework} for a detailed discussion on this side.} and \cite{BKMR25+}\footnote{The authors in \cite{BKMR25+} conjectured this smooth computational transition with respect to the Kikuchi algorithm proposed in \cite{WAM19}. Here we prove the smooth computational transition phenomenon by considering a different algorithm.}.

\subsection{Detection and recovery in tensor PCA}{\label{subsec:problem-set-up}}

We first rigorously specify the spiked tensor model that we will consider. For an integer $p \geq 2$ let $\widetilde{\bm G} \in (\mathbb R^n)^{\otimes p}$ be an asymmetric tensor with i.i.d.\ standard normal entries. We then symmetrize this tensor and obtain a symmetric tensor $\bm G$,
\begin{align}{\label{eq-def-symmetric-tensor}}
    \bm G = \frac{1}{\sqrt{p!}} \sum_{ \pi\in\mathfrak S_p } \widetilde{\bm G}^{\pi} \,,
\end{align}
where $\mathfrak S_p$ is the symmetric group of permutations of $[p]$ and $\widetilde{\bm G}^{\pi}_{i_1,\ldots,i_p} = \widetilde{\bm G}_{\pi(i_1),\ldots,\pi(i_p)}$. Note that if $i_1,\ldots,i_p$ are distinct then $\bm G_{i_1,\ldots,i_p}$ is standard normal. Then, we draw a ``signal vector'' or a ``spike'' $x_* \in \mathbb R^n$ from a prior distribution $\mu$ supported on the sphere $\{ x \in \mathbb R^n: \|x\|=\sqrt{n} \}$. We will focus on the case of Rademacher prior such that $\mu$ is the uniform distribution over the hypercube $\{ -1,+1 \}^n$. Then we let $\bm Y$ be the tensor 
\begin{equation}{\label{eq-def-tensor-PCA}}
    \bm Y = \lambda n^{-\frac{p}{4}} x_*^{\otimes p} + \bm G \,,
\end{equation}
We denote $\Pb=\Pb_{n,p,\lambda}$ be the law of $\bm Y$ under \eqref{eq-def-tensor-PCA}. In addition, let $\mathbb Q=\Qb_{n,p}$ be the law of a symmetric tensor with standard normal entries given by \eqref{eq-def-symmetric-tensor}. We now rigorously state the definition of an algorithm achieves detection/recovery.

\begin{DEF}{\label{def-strong-detection}}
    We say an algorithm $\mathcal A: (\mathbb R^n)^{\otimes p} \to \{ 0,1 \}$ achieves {\em strong detection} between $\Pb$ and $\Qb$, if
    \begin{align}{\label{eq-def-strong-detection}}
        \Pb\big( \mathcal A( \bm{Y} )=0 \big) + \Qb( \mathcal A(\bm Y)=1 ) \to 0 \mbox{ as } n \to \infty \,.
    \end{align}
\end{DEF}
\begin{DEF}{\label{def-recovery}}
    We say an algorithm $\mathcal A: (\mathbb R^n)^{\otimes p} \to \widehat{x} \in \{ -1,+1 \}^n$ achieves {\em weak recovery} with probability $1-\epsilon$, if
    \begin{equation}{\label{eq-def-weak-recovery}}
        \Pb\Big(  \tfrac{ |\langle \widehat x, x_* \rangle| }{ \| \widehat x \| \| x_* \| }= \Omega(1) \Big) \geq 1-\epsilon  \,.
    \end{equation}
    In addition, we say an algorithm $\mathcal A: (\mathbb R^n)^{\otimes p} \to \widehat{x} \in \{ -1,+1 \}^n$ achieves {\em strong recovery}, if
    \begin{equation}{\label{eq-def-strong-recovery}}
        \Pb\Big(  \tfrac{ |\langle \widehat x, x_* \rangle| }{ \| \widehat x \| \| x_* \| }= 1-o(1) \Big) \to 1 \mbox{ as } n \to \infty \,.
    \end{equation}
\end{DEF}

\subsection{Subgraph and subhypergraph counts}{\label{subsec:subgraph-counts}}

Given the symmetry of the noise tensor $\bm G$ and the planted signal vector $x_*$, any detection or recovery statistic must rely on on hypergraph invariants --- hypergraph properties that are invariant under hypergraph isomorphisms, such as subhypergraph counts or spectral properties. This paper adopts a strategy based on subhypergraph counts, a method widely used for network analysis in both theory \cite{MNS15, BDER16} and practice \cite{MSOI+02, ADH+08, RPS+21}. For example, in the context of community detection, counting cycles of logarithmic length has been shown to achieve the optimal detection threshold for distinguishing the stochastic block model (SBM) with symmetric communities from the \ER graph model in the sparse regime \cite{MNS15, MNS24}. Similarly, in the dense regime, counting signed cycles turns out to achieve the optimal asymptotic power \cite{Ban18, BM17}. In addition, algorithms based counting non-backtracking or self-avoiding walks \cite{Mas14, MNS18, BLM15, AS15, AS18} have been proposed to achieve sharp threshold with respect to the related community recovery problem.

In the tensor PCA problem, our observation $\bm{Y}$ is a $p$-fold tensor ($p \geq 3$) with continuous random entries. Consequently, instead of counting standard graphs, our approach must count \emph{weighted hypergraphs}. The heuristic for our algorithm is that counting \emph{2-regular hypergraphs} serves as the natural analog to counting cycles in graphs. Here, a hypergraph is 2-regular if its every vertex belongs to exactly 2 hyperedges. Let $\mathcal H_K$ be the isomorphism class of connected, 2-regular connected hypergraphs on $K$ vertices. In addition, let $\mathcal J_K$ denote the isomorphism class of connected hypergraphs on $K$ vertices where exactly two vertices are leaves (i.e., belong to only one hyperedge) and the remaining vertices are 2-regular. A natural algorithmic approach can then be summarized as follows:
\begin{itemize}
    \item For the detection problem, we decide between the null and alternative hypothesis by counting all the hypergraphs in $\mathcal H_K$ for a carefully chosen $K$;
    \item For the recovery problem, we estimates $x_{*}(i)x_*(j)$ by counting all the hypergraphs in $\mathcal J_K$ whose leaves are exactly $\{ i,j \}$, again for a carefully chosen $K$.
\end{itemize}
This approach provides a high-level explanation for why the smooth computational transition phenomenon occurs only for $p \geq 3$. The reason is that 2-regular connected hypergraphs grows significantly faster than 2-regular connected standard graphs. In fact, there is only one connected 2-regular unlabeled graph on $K$ vertices (which is the $K$-cycle) but the number of connected 2-regular unlabeled hypergraphs on $K$ vertices grows super-exponentially in $K$ (see Lemma~\ref{lem-control-Aut-component-hypergraph} for a stronger assertion). 

However, implementing this approach presents several challenges. Firstly, the subhypergraph counts between different isomorphism class are correlated with each other. Thus, a key challenge is how to better utilize the subhypergraph counts so that the correlation between the counts of non-isomorphic hypergraphs does not overwhelm the signal (i.e., the expectation of counts of isomorphic hypergraphs). Furthermore, direct enumeration of hypergraphs of size $K$ takes $n^{O(K)}$ time which is super-polynomial if $K$ grows with $n$. In our paper, we take the approach following \cite{HS17, MWXY24, MWXY23} by showing that with high probability our test statistic involving counting hypergraphs can be efficiently approximated using the randomized algorithm of \emph{color coding} \cite{AYZ95, AR02, HS17, MWXY24, MWXY23}. This technique enables efficient counting of a subgraph provided it has {\em bounded tree-width}; this was later generalized to hypergraphs with {\em bounded fractional tree-width} \cite{GLS02}.

In conclusion, to obtain a smooth computational transition described in Theorem~\ref{MAIN-THM}, the main innovation in our work is to construct a carefully chosen family of hypergraphs in $\mathcal H_K(\lambda)$ and $\mathcal J_K(\lambda)$ satisfying the following conditions (see Definitions~\ref{def-mathcal-H} and \ref{def-mathcal-J} for the formal description): 
\begin{itemize}
    \item The size of this family grows with rate $e^{ K\cdot \mathsf{gro}(\lambda) }$ where $\mathsf{gro}(\lambda)$ tends to infinity when $\lambda$ tends to $0$ (i.e., this family becomes richer when $\lambda$ decreases);
    \item The correlation between counting different hypergraphs in this family decays ``fast enough'' (i.e., the fluctuation of subgraph counts does not overwhelm the signal);
    \item All the graphs in this family has fractional tree-width bounded by $\mathsf{wid}(\lambda)$, where $\mathsf{wid}(\lambda)$ tends to infinity as $\lambda$ tends to $0$ (i.e., this family can counted in polynomial time, with the power of the polynomial increases as $\lambda$ decreases).
\end{itemize}

\subsection{Connections to low-degree polynomial framework}{\label{subsec:low-degree-framework}}

Our work is motivated by \cite{KWB22}, which predicts the smooth computational transition in tensor PCA phenomenon using the framework of \emph{low-degree polynomials}. Inspired by the sum-of-squares hierarchy, the low-degree polynomial method offers a promising framework for establishing computational lower bounds in high-dimensional inference problems \cite{BHK+19, HS17, HKP+17, Hop18}. Roughly speaking, to study the computational efficiency of the hypothesis testing problem between two probability measures $\Pb$ and $\Qb$, the idea of this framework is to find a low-degree polynomial $f$ (usually with degree $O(\log n)$ where $n$ is the size of data) that separates $\Pb$ and $\Qb$ in a certain sense. The key quantity is the low-degree advantage
\begin{equation}{\label{eq-def-low-deg-adv}}
    \mathsf{Adv}(f) := \frac{ \mathbb E_{\Pb}[f] }{ \sqrt{\mathbb E_{\Qb}[f^2]} } \,.
\end{equation}
The low-degree polynomial method suggests that, if the low-degree advantage of $f$ is uniformly bounded for all polynomials $f$ with degree $D$, then no algorithm with running time $e^{\widetilde{\Theta}(D)}$ can strongly distinguish $\Pb$ and $\Qb$ (here $\widetilde{\Theta}(\cdot)$ means having the same order up to poly-logarithmic factors). The appeal of this framework lies in its ability to yield tight computational lower bounds for a wide range of problems, including detection and recovery problems such as planted clique, planted dense subgraph, community detection, sparse-PCA and graph alignment (see \cite{HS17, HKP+17, Hop18, KWB22, SW22, DMW23+, MW25, DKW22, BEH+22, DDL23+, CDGL24+}), optimization problems such as maximal independent sets in sparse random graphs \cite{GJW24, Wein22}, refutation problems such as certification of RIP and lift-monotone properties of random regular graphs \cite{WBP16, BKW20, KY24}, and constraint satisfaction problems such as random $k$-SAT \cite{BH22}.

In \cite{KWB22}, the authors consider the computational transition in tensor PCA with respect to the low-degree polynomials. For the model in \eqref{eq-def-tensor-PCA}, they show that there exist two constants $A=A(p)$ and $B=B(p)$ such that:
\begin{itemize}
    \item If $D \leq A(p) \cdot \lambda^{\frac{-4}{p-2}}$, then $\mathsf{Adv}(f)=O(1)$ for all polynomials $f$ with degree bounded by $D$;
    \item If $D \geq B(p) \cdot \lambda^{ \frac{-4}{p-2} }$ and $D=\omega(1)$, then there exists a polynomial with degree $D$ such that $\mathsf{Adv}(f)=\omega(1)$.\footnote{This existence result is proven indirectly by relating the optimal advantage to certain ``overlap structures'', rather than by an explicit construction.}
\end{itemize}
Assuming low-degree conjecture, this constitutes evidence for a smooth computational transition, as there exists a smooth tradeoff between the degree of the polynomial (roughly correspond to the computational cost of the algorithm) and its statistical power. However, (on the algorithmic side) they did not rigorously verify this computational transition scheme for two primary reasons. First, even when a polynomial $f$ with $\mathsf{Adv}(f)=\omega(1)$ exists, it is not proven that it can be leveraged into an distinguishing algorithm. This is because the predominant contribution to the low-degree advantage $\mathsf{Adv}(f)$ might originate from certain rare events (see \cite{BEH+22, DMW23+, DDL23+, CDGL24+} for example). Secondly, it is not sure whether such a polynomial can be computed efficiently, as the degree of the polynomial is order $\omega(1)$. 

In this work, we overcome these obstacles by explicitly constructing a low-degree polynomial $f$ that serves as a basis for an efficient distinguishing algorithm. Through a delicate second-moment analysis, we prove that $f$ achieves large values with high probability under $\Pb$ while remaining small with high probability under $\Qb$, thereby strongly distinguishing the two measures. In addition, the degree of $f$ equals $D(p) \cdot \lambda^{ \frac{-4}{p-2} } \cdot \ell$, where $D(p)$ is a constant that only depends on $p$ and $\ell$ is an arbitrary $\omega(1)$ quantity (see \eqref{eq-def-f-mathcal-H} and \eqref{eq-condition-strong-detection}). Thus, the degree of this polynomial nearly matches the low-degree lower bound established in  \cite{KWB22}. On the computational side, the polynomial we construct corresponds to counting specific subhypergraphs with bounded fractional tree-width, and thus we can efficiently compute this polynomial using hypergraph color coding (see the discussions in Section~\ref{subsec:subgraph-counts}).

\subsection{Notation and paper organization}{\label{subsec:notation}}

A hypergraph is a pair $H=(V(H),E(H))$, consisting of a set $V(H)$ of vertices and a set $E(H)$ of nonempty subsets of $V(H)$, the hyperedges of $H$. We say a hypergraph $H$ is $p$-uniform, if $|e|=p$ for all $e \in E(H)$, and in particular we say $H$ is a graph if it is $2$-uniform. Recall \eqref{eq-def-tensor-PCA}, for a $p$-uniform hypergraph $H$ and $e = \{ i_1,\ldots,i_{p} \} \in E(H)$ we will write $\bm Y_e = \bm Y_{i_1,\ldots,i_p}$ and $\bm G_e = \bm G_{i_1,\ldots,i_p}$ (note that $\bm Y$ and $\bm G$ are symmetric so the order of $i_1,\ldots,i_p$ does not matter). We say $H$ is a subhypergraph of $S$, denoted by $H \subset S$, if $V(H) \subset V(S)$ and $E(H) \subset E(S)$. Denote $H \subset \mathsf K_n$ if $V(H) \subset [n]$. For $H,S\subset \mathsf K_n$, denote $H \cap S$ to be the hypergraph with vertices $V(H) \cap V(S)$ and hyperedges $E(H) \cap E(S)$, and denote $H \cup S$ to be the hypergraph with vertices $V(H) \cup V(S)$ and hyperedges $E(H) \cup E(S)$. For all $v \in V(H)$, denote $\mathsf{deg}_H(v)=\#\{ e \in E(H): v \in e \}$ be the degree of $v$ in $H$. We say $v$ is an isolated vertex of $H$, if $\mathsf{deg}_H(v)=0$. Denote $\mathsf I(H)$ be the set of isolated vertices of $H$. We say $v$ is a leaf of $H$, if $\mathsf{deg}_H(v)=1$. Denote $\mathsf{L}(H)$ be the set of leaves in $H$. 

For a hypergraph $H$ and a set $X \subset V(H)$, denote the subhypergraph of $H$ induced by $X$ to be the hypergraph
\begin{align*}
    H|_X = \Big( X, \big\{ e \cap X: e \in E(H) \mbox{ with } e \cap X \neq \emptyset \big\} \Big) \,.
\end{align*}
Also, denote $H|_{\setminus X}$ be the subhypergraph of $H$ induced by $V(H)\setminus X$. The projected graph of a hypergraph is the graph
\begin{align*}
    \underline{H}= \Big( V(H), \big\{ \{ i,j \}: i,j \in e \mbox{ for some } e \in E(H) \big\} \Big) \,.
\end{align*}
A hypergraph $H$ is connected if $\underline{H}$ is connected. A set $X \subset V(H)$ is connected if $H_X$ is connected, and a connected component of $H$ is a maximal connected subset of $V(H)$. We say two hypergraphs $H$ and $H'$ are isomorphic, if there exists a bijection $\pi:V(H) \to V(H')$ such that $\pi(e) \in E(H')$ if and only if $e \in E(H)$. Denote by $[H]$ the isomorphism class of $H$; it is customary to refer to these isomorphic classes as unlabeled hypergraphs. Let $\mathsf{Aut}(H)$ be the set of automorphisms of $H$ (hypergraph isomorphisms to itself). For each $v \in V(H)$, let $\mathsf{Aut}_v(H)=\{ \pi\in\mathsf{Aut}(H):\pi(v)=v \}$. We say $u,v \in V(H)$ belongs to the same equivalence class under $\mathsf{Aut}(H)$, if there exists $\pi\in\mathsf{Aut}(H)$ such that $\pi(u)=v$. Denote $V_{\diamond}(H)$ be the equivalence class of $V(H)$ and $\mathsf{L}_{\diamond}(H)$ be the equivalence class of $\mathsf{L}(H)$ under $\mathsf{Aut}(H)$.

For two real numbers $a$ and $b$, we let $a \vee b = \max \{ a,b \}$ and $a \wedge b = \min \{ a,b \}$. For two sets $A$ and $B$, we define $A\sqcup B$ to be the disjoint union of $A$ and $B$ (so the notation $\sqcup$ only applies when $A, B$ are disjoint). We use standard asymptotic notations: for two sequences $a_n$ and $b_n$ of positive numbers, we write $a_n = O(b_n)$, if $a_n<Cb_n$ for an absolute constant $C$ and for all $n$ (similarly we use the notation $O_h$ if the constant $C$ is not absolute but depends only on $h$); we write $a_n = \Omega(b_n)$, if $b_n = O(a_n)$; we write $a_n = \Theta(b_n)$, if $a_n =O(b_n)$ and $a_n = \Omega(b_n)$; we write $a_n = o(b_n)$ or $b_n = \omega(a_n)$, if $a_n/b_n \to 0$ as $n \to \infty$. Denote $\mu$ to be the uniform distribution over $\{ -1,+1 \}^n$. For a set $\mathsf A$, we will use both $\# \mathsf A$ and $|\mathsf A|$ to denote its cardinality. 

The rest of the paper is organized as follows. In Section~\ref{sec:main-results}, we first present the construction of our detection and recovery statistics and present their theoretical guarantees. Section~\ref{sec:stat-analysis} provides the statistical analysis of our subhypergraph counting statistics. In Section~\ref{sec:hypergraph-color-coding}, we present efficient algorithms to approximately compute our subhypergraph counting statistics based on hypergraph color coding. Several auxiliary results are postponed to the appendix to ensure a smooth flow of presentation.

\section{Main results and discussions}{\label{sec:main-results}}

\subsection{The detection statistics and theoretical guarantees}{\label{subsec:detect-stat}}

We start with some preliminary definitions before specializing to the class of hypergraphs that we will consider. For any $p$-uniform hypergraph $S \subset \mathsf K_n$, define
\begin{equation}{\label{eq-def-f-S}}
    f_S(\bm Y) = \prod_{e \in E(S)} \bm Y_e \,.
\end{equation}
In addition, given a family $\mathcal H_{\star}$ of unlabeled $p$-uniform hypergraphs on $mp\ell$ vertices (here $m,\ell$ are parameters that will be decided later), define
\begin{equation}{\label{eq-def-f-mathcal-H}}
    f_{\mathcal H_{\star}} = \frac{1}{ \sqrt{ n^{mp\ell}\beta_{\mathcal H_{\star}} }} \sum_{ [H] \in \mathcal H_{\star} } \sum_{S \subset \mathsf{K}_n, S \cong H} f_{S}(\bm Y) \,, 
\end{equation}
where 
\begin{equation}{\label{eq-def-beta-mathcal-H}}
    \beta_{\mathcal H_{\star}} = \sum_{[H] \in \mathcal H_{\star}} \frac{ 1 }{ |\mathsf{Aut}(H)| } \,.
\end{equation}
Algorithm~\ref{alg:detection-meta} below describes our proposed method for tensor PCA based on subhypergraph counts.
\begin{breakablealgorithm}{\label{alg:detection-meta}}
\caption{Detection in tensor PCA by subhypergraph counts}
    \begin{algorithmic}[1]
    \STATE {\bf Input:} Adjacency tensor $\bm Y$, a family $\mathcal H_{\star}$ of non-isomorphic $p$-uniform hypergraphs, and a threshold $\tau\geq 0$.
    \STATE Compute $f_{\mathcal H_{\star}}(\bm Y)$ according to \eqref{eq-def-f-mathcal-H}. 
    \STATE Let $\mathtt q=1$ if $f_{\mathcal H_{\star}}(\bm Y) \geq \tau$ and $\mathtt q=0$ if $f_{\mathcal H_{\star}}(\bm Y) < \tau$.
    \STATE {\bf Output:} $\mathtt q$.
    \end{algorithmic}
\end{breakablealgorithm}
At this point Algorithm~\ref{alg:detection-meta} is a ``meta algorithm'' and the key to its application is to carefully choose this collection of hypergraphs $\mathcal H_{\star}$. We begin by considering $p$-uniform hypergraphs satisfying specific connecting properties, which is the staring point of our construction.

\begin{DEF}{\label{def-component-hypergraph}}
    For any integer $m$, define $\mathcal U=\mathcal U(m,p)$ to be the set of unlabeled $p$-uniform hypergraphs $[U]$ satisfying the following conditions:
    \begin{enumerate}
        \item[(1)] $|V(U)|=mp+1$, $|E(U)|=2m$ and $\mathsf{deg}_U(v)\in \{ 1,2 \}$ for all $v \in V(U)$;
        \item[(2)] $U$ is connected;
        \item[(3)] for all $v \in V(U)$, $U|_{\setminus \{v \}}$ is connected.
    \end{enumerate}
\end{DEF}
Note that since $\sum_{v \in V(H)} \mathsf{deg}_H(v)=p|E(H)|$, Item~(1) implies that $H$ has exactly two leaves and the degree of other non-leaf vertices are $2$. The next lemma provides an lower bound on the cardinality of $\mathcal U$, showing that it contains ``sufficiently many'' unlabeled hypergraphs.
\begin{lemma}{\label{lem-control-Aut-component-hypergraph}}
    For $p \geq 3$, there exists an absolute constant $R>0$ such that
    \begin{equation*}
        \beta_{\mathcal U} = \sum_{ [U] \in \mathcal U } \frac{ 1 }{ |\mathsf{Aut}(U)| } \geq R^{-1} \Big( 1-\tfrac{10^{10}R^3}{m} \Big) \frac{(2pm)!}{2^{mp}(p!)^{2m}(mp+1)!(2m)!} \,.
    \end{equation*}
\end{lemma} 

The proof of Lemma~\ref{lem-control-Aut-component-hypergraph} is postponed to Section~\ref{subsec:proof-lem-2.2}. To this end, we construct a special family of hypergraphs $\mathcal H$ as below (see Figure~\ref{fig:necklace-hypergraph} for an illustration).

\begin{DEF}{\label{def-mathcal-H}}
    For any integer $m,\ell$, define $\mathcal H=\mathcal H(m,p,\ell)$ to be the set of unlabeled $p$-uniform hypergraphs $[H]$ such that there exists $U_1,\ldots,U_{\ell}$ with the following holds:  
    \begin{enumerate}
        \item[(1)] $[U_i] \in \mathcal U(m,p)$ for all $1 \leq i \leq \ell$;
        \item[(2)] $\mathsf{L}(U_i)=\{ v_i,v_{i+1} \}$ for all $1 \leq i \leq \ell$, and $V(U_i)\cap V(U_{i+1}) = \{ v_{i+1} \}$ for all $1 \leq i \leq \ell$ (we let $v_{\ell+1}=v_1$ and $U_{\ell+1}=U_1$);
        \item[(3)] $V(H) = \cup_{1 \leq i \leq \ell} V(U_i)$ and $E(H)=\cup_{1 \leq i \leq \ell} E(U_i)$.
    \end{enumerate}
\end{DEF}

\begin{figure}[!ht]
    \centering
    \vspace{0cm}
    \includegraphics[height=6.05cm,width=12.64cm]{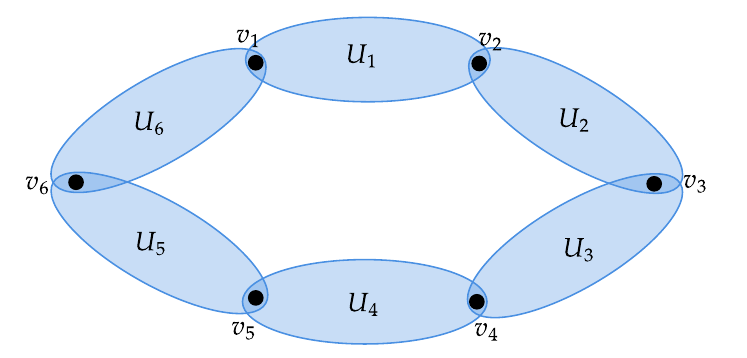}
    \caption{\noindent An unlabeled hypergraph in $\mathcal H(m,p,\ell)$ with $\ell=6$}
    \label{fig:necklace-hypergraph}
\end{figure}

Note that Definitions~\ref{def-component-hypergraph} and \ref{def-mathcal-H} yields that for all $[H] \in \mathcal H$, we have $|V(H)|=mp\ell$, $|E(H)|=2m\ell$ and $\mathsf{deg}_H(v)=2$ for all $v \in H$. We will also provide a similar lower bound on the cardinality of $\mathcal H$, as incorporated in the next lemma.
\begin{lemma}{\label{lem-control-Aut-necklace-hypergraph}}
    For $p \geq 3$, there exists an absolute constant $R>0$ such that
    \begin{equation*}
        \beta_{\mathcal H} = \sum_{ [H] \in \mathcal H } \frac{ 1 }{ |\mathsf{Aut}(H)| } \geq \frac{1}{(2\ell)^2} \cdot \Big( R^{-1} \big( 1-\tfrac{10^{10}R^3}{m} \big) \frac{(2pm)!}{2^{mp}(p!)^{m}(mp)!(2m)!} \Big)^{\ell} \,.
    \end{equation*}
\end{lemma}
The proof of Lemma~\ref{lem-control-Aut-necklace-hypergraph} is incorporated in Section~\ref{subsec:proof-lem-2.4}. Our main result can be summarized as follows:
\begin{proposition}{\label{main-prop-detection}}
    Suppose  
    \begin{align}{\label{eq-condition-strong-detection}}
        p \geq 3, \ m \geq 10^{11} R^3, \ \frac{\lambda^{4m}(2pm)!}{2^{pm}(p!)^{m}(mp)!(2m)!} > 10^{pm} R \mbox{ and }  \omega(1)=\ell=o(\tfrac{\log n}{\log\log n})  \,.
    \end{align}
    Then we have
    \begin{align}{\label{eq-moment-control}}
        \mathbb E_{\Pb}\big[ f_{\mathcal H} \big]=\omega(1), \ \mathbb E_{\Qb}\big[ f_{\mathcal H}^2 \big]=1+o(1) \mbox{, and } \operatorname{Var}_{\Pb}\big[ f_{\mathcal H} \big]=o(1) \cdot \mathbb E_{\Pb}\big[ f_{\mathcal H} \big]^2 \,.
    \end{align}
\end{proposition}
Note that by standard estimates we have
\begin{align*}
    \frac{\lambda^{4m}(2pm)!}{(p!)^{m}(mp)!(2m)!} < \Big( \frac{ \lambda^{4}(2pm)^{2p} }{ p^{p}(pm)^{p}(2m) } \Big)^{2m} < \Big( \frac{ \lambda^{4} m^{p-2} }{ p^{2p} } \Big)^{2m} \,,
\end{align*}
thus choosing 
\begin{equation}{\label{eq-suffice-bound-m}}
    m > 100R (2p)^{\frac{2p}{p-2}} \cdot \max\{ \lambda^{-\frac{4}{p-2}}, 10^{11} R^3 \}
\end{equation}
suffices to satisfy \eqref{eq-condition-strong-detection}. Combining these
variance bounds with Chebyshev’s inequality, we arrive at the following sufficient condition for the statistic $f_{\mathcal H}$ to achieve strong detection.
\begin{thm}{\label{MAIN-THM-detection}}
    Suppose \eqref{eq-condition-strong-detection} is satisfied. Then the testing error satisfies
    \begin{equation}{\label{eq-testing-error}}
        \Pb\big( f_{\mathcal H}(\bm Y)\leq\tau \big) + \Qb\big( f_{\mathcal H}(\bm Y)\geq\tau \big) = o(1) \,,
    \end{equation}
    where the threshold is chosen as
    \begin{equation*}
        \tau = C \mathbb E_{\Pb}\big[ f_{\mathcal H}(\bm Y) \big]
    \end{equation*}
    for any fixed constant $0<C<1$. In particular, Algorithm~\ref{alg:detection-meta} described above achieves strong detection between $\Pb$ and $\Qb$.
\end{thm}

From a computational perspective, evaluating each
\begin{align*}
    \sum_{ S \subset \mathsf K_n: S \cong H } f_S(\bm Y)
\end{align*}
in \eqref{eq-def-f-mathcal-H} by exhaustive search takes $n^{O(mp\ell)}$ time which is super-polynomial when $\ell=\omega(1)$. To resolve this computational issue, in Section~\ref{subsec:approx-detection} we design an polynomial-time algorithm (see Algorithm~\ref{alg:cal-widetilde-f}) to compute an approximation $\widetilde{f}_{\mathcal H}(\bm Y)$ (see \eqref{eq-def-widetilde-f-H}) for $f_{\mathcal H}(\bm Y)$ using the strategy of color coding \cite{AYZ95, AR02, HS17, MWXY24, MWXY23}. The following result shows that the statistic $\widetilde f_{\mathcal H}$ achieves strong detection under the same condition as in Theorem~\ref{MAIN-THM-detection}.
\begin{thm}{\label{MAIN-THM-detection-algorithmic}}
    Suppose \eqref{eq-condition-strong-detection} holds. Then \eqref{eq-testing-error} holds with $\widetilde f_{\mathcal H}$ in place of $f_{\mathcal H}$, namely
    \begin{equation}{\label{eq-testing-error-algorithmic}}
        \Pb\big( \widetilde f_{\mathcal H}(\bm Y)\leq\tau \big) + \Qb\big( \widetilde f_{\mathcal H}(\bm Y)\geq\tau \big) = o(1) \,,
    \end{equation}
    Moreover, $\widetilde f_{\mathcal H}$ can be computed in time $n^{C+o(1)}$ for some constant $C=C(m,p)$.
\end{thm}

\subsection{The recovery statistics and theoretical guarantees}{\label{subsec:recovery-stat}}

Again, we start by introducing our general methodologies before specializing to the class of hypergraphs that we will consider. Recall \eqref{eq-def-f-S}. Given a family $\mathcal J_{\star}$ of unlabeled $p$-uniform hypergraphs such that $|V(J)|=mp\ell+1$ and $|\mathsf L(J)|=2$ for all $[J] \in \mathcal J$ (here again $m,\ell$ are parameters that will be decided later), define 
\begin{equation}{\label{eq-def-Phi-i,j-mathcal-J}}
    \Phi_{i,j}^{\mathcal J_{\star}} := \frac{1}{ \lambda^{2m\ell} n^{\frac{pm\ell}{2}-1} \beta_{\mathcal J_{\star}} } \sum_{[J] \in \mathcal J_{\star}} \sum_{ \substack{ S \subset \mathsf K_n: S \cong J \\ \mathsf L(S)=\{ i,j \} } } f_{S}(\bm Y) \,. 
\end{equation}
for each $i,j \in [n]$. Here $\beta_{\mathcal J_{\star}}$ is defined in \eqref{eq-def-beta-mathcal-H}. Our proposed method of recovery in tensor PCA is as follows.

\begin{breakablealgorithm}{\label{alg:recovery-meta}}
\caption{Recovery in tensor PCA by subhypergraph counts}
    \begin{algorithmic}[1]
    \STATE {\bf Input:} Adjacency tensor $\bm Y$, a family $\mathcal J_{\star}$ of non-isomorphic $p$-uniform hypergraphs.
    \STATE For each pair $i,j \in [n]$, compute $\Phi_{i,j}^{\mathcal J_{\star}}$ as in \eqref{eq-def-Phi-i,j-mathcal-J}. 
    \STATE Arbitrarily choose $i \in [n]$. Let $\widehat{x}(i)=1$ and $\widehat{x}(j) = \mathbf 1( \Phi_{i,j}^{\mathcal J_{\star}}>0 ) - \mathbf 1( \Phi_{i,j}^{\mathcal J_{\star}}\leq 0 )$ for $j \neq i$.
    \STATE {\bf Output:} $\widehat x$.
    \end{algorithmic}
\end{breakablealgorithm}
Again, at this point Algorithm~\ref{alg:recovery-meta} is still a ``meta algorithm'' and the key to its application is to carefully choose this collection of hypergraphs $\mathcal J_{\star}$. Ideally, we would like $\Phi_{i,j}^{\mathcal J_{\star}}$ to be positively correlated with $x_*(i)x_*(j)$. To this end, we construct a special family of hypergraphs $\mathcal J$ as below (see Figure~\ref{fig:necklace-chain} for an illustration). 

\begin{DEF}{\label{def-mathcal-J}}
    Recall Definition~\ref{def-component-hypergraph}. For any integer $m,\ell$, define $\mathcal J=\mathcal J(m,p,\ell)$ to be the set of unlabeled $p$-uniform hypergraphs $[J]$ such that there exists $U_1,\ldots,U_{\ell}$ with the following holds:  
    \begin{enumerate}
        \item[(1)] $[U_i] \in \mathcal U$ for all $1 \leq i \leq \ell$;
        \item[(2)] $\mathsf{L}(U_i)=\{ v_i,v_{i+1} \}$ for all $1 \leq i \leq \ell$, and $V(U_i)\cap V(U_{i+1}) = \{ v_{i+1} \}$ for all $1 \leq i \leq \ell$;
        \item[(3)] $V(J) = \cup_{1 \leq i \leq \ell} V(U_i)$ and $E(J)=\cup_{1 \leq i \leq \ell} E(U_i)$.
        \item[(4)] For all $\pi\in\mathsf{Aut}(J)$ we have $\pi(v_1)=v_1$ and $\pi(v_{\ell+1})=v_{\ell+1}$.
    \end{enumerate}    
\end{DEF}

\begin{figure}[!ht]
    \centering
    \vspace{0cm}
    \includegraphics[height=4.70cm,width=13.42cm]{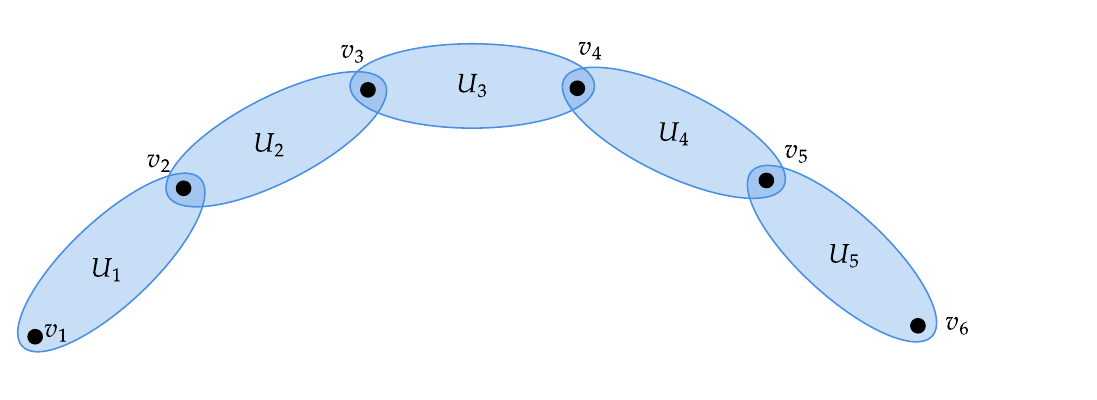}
    \caption{\noindent An unlabeled hypergraph in $\mathcal J(m,p,\ell)$ with $\ell=5$}
    \label{fig:necklace-chain}
\end{figure}

Note that Definitions~\ref{def-component-hypergraph} and \ref{def-mathcal-J} yields that for all $[J] \in \mathcal J$, we have $|V(J)|=mp\ell+1$, $|E(H)|=2m\ell$ and $\mathsf{L}(H)=\{ v_1,v_{\ell+1} \}$. We will provide a similar but more delicate lower bound on the cardinality of $\mathcal J$, as incorporated in the next lemma.
\begin{lemma}{\label{lem-control-Aut-necklace-chain}}
    For all $p,m,\ell>0$, define
    \begin{equation}{\label{eq-def-beta-diamond-U}}
        \beta^{\diamond}_{\mathcal U}= \sum_{ [U] \in \mathcal U } \sum_{ [v] \in \mathsf L_{\diamond}(U) } \frac{ 1 }{ |\mathsf{Aut}_v(U)| } \,.
    \end{equation}
    Then we have for $p \geq 3$
    \begin{align*}
        \big( \beta^{\diamond}_{\mathcal U} \big)^{\ell} \geq \beta_{\mathcal J} := \sum_{ [J] \in \mathcal J } \frac{ 1 }{ |\mathsf{Aut}(J)| } \geq \frac{1}{2} \big( \beta^{\diamond}_{\mathcal U} \big)^{\ell-1} \big( \beta^{\diamond}_{\mathcal U}-2 \big) \,.
    \end{align*}
    In particular, for $p \geq 3$ there exists an absolute constant $R>0$ such that
    \begin{equation*}
        \beta_{\mathcal J} \geq \frac{1}{2} \Big( \frac{ R^{-1}(1-\tfrac{10^{10}R^3}{m}) (2pm)!}{2^{mp}(p!)^{m}(mp)!(2m)!} \Big)^{\ell-1} \Big( \frac{ R^{-1}(1-\tfrac{10^{10}R^3}{m} ) (2pm)!}{2^{mp}(p!)^{m}(mp)!(2m)!} - 2 \Big)  \,.
    \end{equation*}
\end{lemma}

The proof of Lemma~\ref{lem-control-Aut-necklace-chain} is postponed to Section~\ref{subsec:proof-lem-2.9}. The key of our argument is the following result, which shows that $\Phi_{i,j}^{\mathcal J}$ is indeed positively correlated with $x_*(i)x_*(j)$. 
\begin{proposition}{\label{main-prop-recovery}}
    Suppose for some sufficiently small constant $0<\delta<0.1$ 
    \begin{align}{\label{eq-condition-weak-recovery}}
        p \geq 3, \ m>10^{11} R^3, \  \frac{ R^{-1}\lambda^{4m}(2pm)!}{2^{pm}(p!)^{m}(mp)!(2m)!} > (0.1\delta)^{-pm} \mbox{ and }  \ell=\log n  \,.
    \end{align}
    Then we have
    \begin{align}
        \mathbb E_{\Pb}\Big[ \big( \Phi^{\mathcal J}_{i,j} - x_*(i)x_*(j) \big)^2 \Big] \leq \delta^2 \,.  \label{eq-L2-estimation-error}
    \end{align}
\end{proposition}
Combining these variance bounds with Markov's inequality, we arrive at the following sufficient condition for Algorithm~\ref{alg:recovery-meta} to achieve weak recovery.
\begin{thm}{\label{MAIN-THM-recovery}}
    Assume \eqref{eq-condition-weak-recovery} holds. Then we have (below we write $\widehat{x}$ to be the output of Algorithm~\ref{alg:recovery-meta})
    \begin{align*}
        \Pb\Big( \frac{ \langle \widehat x, x_* \rangle }{ \| \widehat x \| \| x_* \| } \geq 1-\delta \Big) \geq 1-\delta \,.
    \end{align*}
\end{thm}
Again, to resolve the computational issue of calculating $\Phi_{i,j}^{\mathcal J}$, in Section~\ref{subsec:approx-recovery}, we give a polynomial-time algorithm (see Algorithm~\ref{alg:cal-widetilde-Phi}) that computes an approximation $\widetilde{\Phi}_{i,j}^{\mathcal J}$ for $\Phi_{i,j}^{\mathcal J}$ using the strategy of color coding. The following result shows that the approximated similarity score $\widetilde{\Phi}_{i,j}^{\mathcal J}$ enjoys the same statistical guarantee under the same condition \eqref{eq-condition-weak-recovery} as Theorem~\ref{MAIN-THM-recovery}.
\begin{thm}{\label{MAIN-THM-recovery-algorithmic}}
    Proposition~\ref{main-prop-recovery} and Theorem~\ref{MAIN-THM-recovery} continues to hold with $\widetilde{\Phi}_{i,j}^{\mathcal J}$ in place of $\Phi_{i,j}^{\mathcal J}$. In addition, $\{ \widetilde{\Phi}_{i,j}^{\mathcal J}: i,j \in [n] \}$ can be computed in time $n^{C'+o(1)}$ for some constant $C'=C(m,p,\delta)$.
\end{thm}
\begin{remark}
    Our result in Theorem~\ref{MAIN-THM-recovery} also shows a smooth computational transition in another sense: by increasing $m$ (while remaining $\Theta(1)$, corresponding to increasing the computational cost of the algorithm while keeping it polynomial) so that \eqref{eq-condition-weak-recovery} holds for smaller constant $\delta$, we can found a polynomial time weak recovery algorithm whose failure probability and recovery error is an arbitrary small constant. In particular, when $\lambda=\omega(1)$ we can choose a suitable $\delta=o(1)$ so that our algorithm achieves strong recovery. 
\end{remark}

\section{Statistical analysis of subhypergraph counts}{\label{sec:stat-analysis}}

In this section, we establish the statistical guarantee of our subhypergraph counting statistic $f_{\mathcal H}$ and $\{ \Phi_{i,j}^{\mathcal J} \}$, as stated in Propositions~\ref{main-prop-detection} and \ref{main-prop-recovery}. Then, we prove of Theorems~\ref{MAIN-THM-detection} and \ref{MAIN-THM-recovery} provided with Propositions~\ref{main-prop-detection} and \ref{main-prop-recovery}.

\subsection{Proof of Proposition~\ref{main-prop-detection}}{\label{proof-prop-2.5}}

\begin{lemma}{\label{lem-mean-var-f-H-part-1}}
    Suppose that \eqref{eq-condition-strong-detection} holds, then
    \begin{align}
        &\mathbb E_{\Qb}[ f_{\mathcal H} ] =0 \,, \label{eq-mean-Qb-f-H} \\
        &\mathbb E_{\Pb}[ f_{\mathcal H} ] =[1+o(1)] \lambda^{2m\ell} \sqrt{\beta_{\mathcal H}} \,, \label{eq-mean-pb-f-H} \\
        &\mathbb E_{\Qb}[ f_{\mathcal H}^2 ] = 1+o(1) \,. \label{eq-var-Qb-f-H}
    \end{align}
\end{lemma}
\begin{proof}
    Recall \eqref{eq-def-f-mathcal-H}. Since $\mathbb E_{\Qb}[ f_S(\bm Y) ]=0$ for all $S \subset \mathsf K_n, E(S) \neq \emptyset$, we have $\mathbb E_{\Qb}[f_{\mathcal H}]=0$ by linearity. In addition, recall \eqref{eq-def-tensor-PCA}, by writing $x^{e} = x_{i_1} \ldots x_{i_p}$ for $e=\{ i_1,\ldots,i_p \}$ we have
    \begin{align*}
        \mathbb E_{\Pb}[f_S(\bm Y)] &= \mathbb E_{x\sim\mu} \mathbb E_{\Pb}\Big[ \prod_{e \in E(S)} \big( \lambda n^{-\frac{p}{4}} x^{e} + \bm G_e \big) \mid x \Big] \\
        &= \mathbb E_{x \sim \mu}\Big[ \prod_{ e \in E(S) } \big( \lambda n^{-\frac{p}{4}} x^{e} \big) \Big] = \lambda^{|E(S)|} n^{-\frac{p|E(S)|}{4}} \mathbb E_{x \sim \mu}\Big[ \prod_{ e \in E(S) } x^{e} \Big] \\
        &= \lambda^{2m\ell} n^{ -\frac{pm\ell}{2} } \,,
    \end{align*}
    where the last equality follows from $|E(S)|=2m\ell$ and 
    \begin{align*}
        \prod_{ e \in E(S) } x^{e} = \prod_{ e\in E(S) } \prod_{ v \in V(S), v \in e } x_v = \prod_{ v \in V(S) } x_v^{ \mathsf{deg}_S(v) } = \prod_{ v \in V(S) } x_v^{2} = 1
    \end{align*}
    for all $[S] \in \mathcal H$. Thus, we have
    \begin{align*}
        \mathbb E_{\Pb}[f_{\mathcal H}] &\overset{\eqref{eq-def-f-mathcal-H}}{=} \frac{1}{\sqrt{ n^{mp\ell}\beta_{\mathcal H} }}  \sum_{ [H] \in \mathcal H } \sum_{S \subset \mathsf{K}_n, S \cong H} \lambda^{2m\ell} n^{ -\frac{pm\ell}{2} } \\
        &= \frac{ \lambda^{2m\ell} }{ n^{pm\ell}\sqrt{\beta_{\mathcal H}} }  \sum_{[H] \in \mathcal H} \#\big\{ S \cong \mathsf K_n: S \cong H \big\} \\
        & \overset{\text{Lemma~\ref{lem-prelim-graphs},(2)}}{=} [1+o(1)] \cdot \frac{ \lambda^{2m\ell} }{ n^{pm\ell}\sqrt{\beta_{\mathcal H}} }  \sum_{[H] \in \mathcal H} \frac{n^{pm\ell}}{ |\mathsf{Aut}(H)| } \overset{\eqref{eq-def-beta-mathcal-H}}{=} \lambda^{2m\ell} \sqrt{\beta_{\mathcal H}} \,.
    \end{align*}
    Also, from Lemma~\ref{lem-control-Aut-necklace-hypergraph} and \eqref{eq-condition-strong-detection} we see that $\lambda^{2m\ell} \sqrt{\beta_{\mathcal H}}=\omega(1)$. Finally, note that
    \begin{align}{\label{eq-standard-orthogonal}}
        \mathbb E_{\Qb}\big[ f_{S}(\bm Y) f_{K}(\bm Y) \big] = \mathbf 1_{ \{ S=K \} } \,,
    \end{align}
    we have
    \begin{align*}
        \mathbb E_{\Qb}[f_{\mathcal H}^2] &\overset{\eqref{eq-def-f-mathcal-H}}{=} \frac{ 1 }{ n^{pm\ell} \beta_{\mathcal H} } \sum_{ [H],[I] \in \mathcal H } \sum_{ \substack{ S,K \subset \mathsf K_n \\ S \cong H, K \cong I } } \mathbb E_{\Qb}\big[ f_{S}(\bm Y) f_{K}(\bm Y) \big] \\
        &= \frac{ 1 }{ n^{pm\ell} \beta_{\mathcal H} } \sum_{ [H],[I] \in \mathcal H } \sum_{ \substack{ S,K \subset \mathsf K_n \\ S \cong H, K \cong I } } \mathbf 1_{ \{ S=K \} } \\
        &= \frac{ 1 }{ n^{pm\ell} \beta_{\mathcal H} } \sum_{ [H] \in \mathcal H } \#\big\{ S \cong \mathsf K_n: S \cong H \big\}  \\
        &\overset{\text{Lemma~\ref{lem-prelim-graphs},(2)}}{=} [1+o(1)] \cdot \frac{ 1 }{ n^{pm\ell} \beta_{\mathcal H} } \sum_{ [H] \in \mathcal H } \frac{ n^{pm\ell} }{ |\mathsf{Aut}(H)| } \overset{\eqref{eq-def-beta-mathcal-H}}{=} 1+o(1) \,. \qedhere
    \end{align*}
\end{proof}

Now we bound the variance of $f_{\mathcal H}$ under the alternative hypothesis $\Pb$. Note that
\begin{align}{\label{eq-var-Pb-f-H-relax-1}}
    \operatorname{Var}_{\Pb}[f_{\mathcal H}] \overset{\eqref{eq-def-f-mathcal-H}}{=} \frac{ 1 }{ n^{mp\ell} \beta_{\mathcal H} } \sum_{[H],[I] \in \mathcal H} \sum_{ \substack{ S,K \subset \mathsf K_n \\ S \cong H, K \cong I } } \operatorname{Cov}_{\Pb}( f_S(\bm Y), f_K(\bm Y) ) \,.
\end{align}

\begin{lemma}{\label{lem-est-cov-f-S-f-K}}
    We have
    \begin{align*}
        \operatorname{Cov}_{\Pb}( f_S(\bm Y), f_K(\bm Y) ) \leq \mathbf 1_{ \{ E(S) \cap E(K) \neq \emptyset \} } \cdot \lambda^{4m\ell} n^{-pm\ell} \cdot \big( \lambda^{-1} n^{\frac{p}{4}} \big)^{2|E(S) \cap E(K)|}
    \end{align*}
\end{lemma}
\begin{proof}
    Note that
    \begin{align}
        \mathbb E_{\Pb}[ f_S(\bm Y) f_{K}(\bm Y) ] &\overset{\eqref{eq-def-tensor-PCA},\eqref{eq-def-f-mathcal-H}}{=} \mathbb E_{\Pb}\Big[ \prod_{ e \in E(S) } \big( \lambda n^{-\frac{p}{4}} x^{e}+\bm G_{e} \big) \prod_{ e \in E(K) } \big( \lambda n^{-\frac{p}{4}} x^{e}+\bm G_{e} \big) \Big] \nonumber \\
        &= \mathbb E_{x \sim\mu} \mathbb E_{\Pb}\Big[ \prod_{ e \in E(S) \cap E(K) } \big( \lambda n^{-\frac{p}{4}} x^{e}+\bm G_{e} \big)^2 \prod_{ e \in E(S) \triangle E(K) } \big( \lambda n^{-\frac{p}{4}} x^{e}+\bm G_{e} \big) \mid x \Big] \nonumber \\
        &= \mathbb E_{x \sim\mu}\Big[ \prod_{ e \in E(S) \cap E(K) } \big( 1+\lambda^2 n^{-\frac{p}{2}} \big) \prod_{ e \in E(S) \triangle E(K) } \big( \lambda n^{-\frac{p}{4}} x^{e} \big)  \Big] \,. \label{eq-cov-f-S-f-K-simplify-1}
    \end{align}
    In addition, from Definition~\ref{def-mathcal-H} we see that $|E(S)|=2m\ell$ and $\mathsf{deg}_S(v)=2$ for all $v \in V(S)$ if $[S] \in\mathcal H$. Thus, for $[S],[K] \in \mathcal H$ we have
    \begin{align*}
        \prod_{ e \in E(S) \triangle E(K) } x^{e} &= \prod_{ e \in E(S) \triangle E(K) } \prod_{ v \in e } x_v = \prod_{v \in V(S) \cup V(K)} x_v^{\#\{ e \in E(S) \triangle E(K): v \in e \}} \\
        &= \prod_{v \in V(S) \cup V(K)} x_v^{ \#\{ e \in E(S): v \in e \} + \#\{ e \in E(K): v \in e \} - 2 \#\{ e \in E(S) \cap E(K): v \in e \} } \\
        &= \prod_{v \in V(S) \cup V(K)} x_v^{ 2+2-2\#\{ e \in E(S) \cap E(K): v \in e \} } =1 \,.
    \end{align*}
    Thus, we see that 
    \begin{align*}
        \eqref{eq-cov-f-S-f-K-simplify-1} = \big( \lambda n^{-\frac{p}{4}} \big)^{|E(S) \triangle E(K)|} \big( 1+\lambda^2 n^{-\frac{p}{2}} \big)^{ |E(S) \cap E(K)| } 
    \end{align*}
    Recall that in Lemma~\ref{lem-mean-var-f-H-part-1} we have shown that 
    \begin{align*}
        \mathbb E_{\Pb}[f_S(\bm Y)]=\mathbb E_{\Pb}[f_K(\bm Y)]= \lambda^{2m\ell} n^{-\frac{pm\ell}{2}} \,,
    \end{align*}
    thus we have $\operatorname{Cov}_{\Pb}( f_S(\bm Y),f_K(\bm Y) )=0$ if $E(S) \cap E(K)=\emptyset$. In addition, when $E(S) \cap E(K) \neq \emptyset$ we have
    \begin{align*}
        \operatorname{Cov}_{\Pb}( f_S(\bm Y),f_K(\bm Y) ) &\leq \mathbb E_{\Pb}[ f_S(\bm Y) f_{K}(\bm Y) ] \leq (1+o(1)) \big( \lambda n^{-\frac{p}{4}} \big)^{|E(S) \triangle E(K)|} \\
        &= [1+o(1)] \lambda^{4m\ell} n^{-pm\ell} \cdot \big( \lambda^{-1} n^{\frac{p}{4}} \big)^{2|E(S) \cap E(K)|} \,. \qedhere
    \end{align*}
\end{proof}

\begin{lemma}{\label{lem-mean-var-f-H-part-2}}
    Assume \eqref{eq-condition-strong-detection} holds. Then $\frac{ \operatorname{Var}_{\Pb}[f_{\mathcal H}] }{ \mathbb E_{\Pb}[f_{\mathcal H}]^2 }=o(1)$.
\end{lemma}
\begin{proof}
    Plugging Lemma~\ref{lem-est-cov-f-S-f-K} into \eqref{eq-var-Pb-f-H-relax-1} we get that
    \begin{align}
        \operatorname{Var}_{\Pb}[f_{\mathcal H}] &\leq \frac{1}{n^{pm\ell} \beta_{\mathcal H}} \sum_{[H],[I] \in \mathcal H} \sum_{ \substack{ S,K \subset \mathsf K_n: S \cong H, K \cong I \\ E(S) \cap E(K) \neq \emptyset } } \lambda^{4m\ell} n^{-pm\ell} \cdot \big( \lambda^{-1} n^{\frac{p}{4}} \big)^{2|E(S) \cap E(K)|} \nonumber \\
        &= \frac{ \lambda^{4m\ell} }{ n^{2pm\ell}\beta_{\mathcal H} } \sum_{ \substack{ S \subset \mathsf{K}_n \\ [S] \in \mathcal H } } \sum_{ \substack{ K \subset \mathsf K_n: [K] \in \mathcal H \\ E(S) \cap E(K) \neq \emptyset } } \big( \lambda^{-1} n^{\frac{p}{4}} \big)^{2|E(S) \cap E(K)|} \,. \nonumber
    \end{align}
    In addition, from Lemma~\ref{lem-mean-var-f-H-part-1} we have $\mathbb E_{\Pb}[f_{\mathcal H}]=[1+o(1)] \lambda^{2m\ell} \sqrt{\beta_{\mathcal H}}$. Thus,
    \begin{align}
        \frac{ \operatorname{Var}_{\Pb}[f_{\mathcal H}] }{ \mathbb E_{\Pb}[f_{\mathcal H}]^2 } &\leq \frac{ 1+o(1) }{ n^{2pm\ell} \beta_{\mathcal H}^2 } \sum_{ \substack{ S \subset \mathsf{K}_n \\ [S] \in \mathcal H } } \sum_{ \substack{ K \subset \mathsf K_n: [K] \in \mathcal H \\ E(S) \cap E(K) \neq \emptyset } } \big( \lambda^{-1} n^{\frac{p}{4}} \big)^{2|E(S) \cap E(K)|} \nonumber \\
        &= \frac{ 1+o(1) }{ n^{2pm\ell} \beta_{\mathcal H}^2 } \sum_{ \substack{ S \subset \mathsf{K}_n \\ [S] \in \mathcal H } } \sum_{ 1 \leq k \leq 2m\ell } \big( \lambda^{-1} n^{\frac{p}{4}} \big)^{2k} \#\big\{ K \subset \mathsf K_n: [K] \in \mathcal H, |E(K) \cap E(S)|=k \big\} \,. \label{eq-var-Pb-f-H-relax-2}
    \end{align}
    Note that
    \begin{align*}
        \frac{ 1+o(1) }{ n^{2pm\ell} \beta_{\mathcal H}^2 } \sum_{ \substack{ S \subset \mathsf{K}_n, [S] \in \mathcal H } } \big( \lambda^{-1} n^{\frac{p}{4}} \big)^{2\cdot 2m\ell} =\ & \frac{1+o(1)}{ \lambda^{4m\ell} n^{pm\ell} \beta_{\mathcal H}^2 } \#\big\{ S: S \subset \mathsf K_n, [S] \in \mathcal H \big\} \\
        =\ & \frac{1+o(1)}{ \lambda^{4m\ell} n^{pm\ell} \beta_{\mathcal H}^2 } \sum_{[H] \in \mathcal H} \#\big\{ S: S \subset \mathsf K_n, S \cong H \big\} \\
        \overset{\text{Lemma~\ref{lem-prelim-graphs},(1)}}{=} & \frac{1+o(1)}{ \lambda^{4m\ell} n^{pm\ell} \beta_{\mathcal H}^2 } \sum_{[H] \in \mathcal H} \frac{ n^{pm\ell} }{ |\mathsf{Aut}(H)| } \overset{\eqref{eq-def-beta-mathcal-H}}{=} \frac{1+o(1)}{ \lambda^{4m\ell} \beta_{\mathcal H} } \overset{\eqref{eq-condition-strong-detection}}{=} o(1) \,.
    \end{align*}
    Thus, 
    \begin{align}
        \eqref{eq-var-Pb-f-H-relax-2} &= o(1) + \frac{ 1+o(1) }{ n^{2pm\ell} \beta_{\mathcal H}^2 } \sum_{ \substack{ S \subset \mathsf{K}_n \\ [S] \in \mathcal H } } \sum_{ 1 \leq k \leq 2m\ell-1 } \big( \lambda^{-1} n^{\frac{p}{4}} \big)^{2k} \#\big\{ K \subset \mathsf K_n: [K] \in \mathcal H, |E(K) \cap E(S)|=k \big\} \nonumber \\
        & \overset{\text{Lemma~\ref{lem-prelim-graphs},(3)}}{=} o(1) + \frac{ 1+o(1) }{ n^{2pm\ell} \beta_{\mathcal H}^2 } \sum_{ \substack{ S \subset \mathsf{K}_n \\ [S] \in \mathcal H } } \sum_{ 1 \leq k \leq 2m\ell-1 } \sum_{ 2h>pk } \big( \lambda^{-1} n^{\frac{p}{4}} \big)^{2k} \mathrm{ENUM}(k,h) \,,  \label{eq-var-Pb-f-H-relax-3}
    \end{align}
    where 
    \begin{align}{\label{eq-def-ENUM}}
        \mathrm{ENUM}(k,h) = \#\big\{ K \subset \mathsf K_n: [K] \in \mathcal H, |E(K) \cap E(S)|=k, |V(S) \cap V(K)|=h \big\} \,.
    \end{align}
    From Item~(4) in Lemma~\ref{lem-prelim-graphs}, we have
    \begin{align*}
        \mathrm{ENUM}(k,h) \leq \binom{pm\ell}{h} \sum_{[I] \in \mathcal H} \frac{ n^{pm\ell-h} (pm\ell)! }{ |\mathsf{Aut}(I)| } \overset{\eqref{eq-def-beta-mathcal-H}}{=} \beta_{\mathcal H} n^{pm\ell-h} (pm\ell)! \binom{pm\ell}{h} \leq \beta_{\mathcal H} n^{pm\ell-h} (pm\ell)^{pm\ell} \,.
    \end{align*}
    Plugging this estimation into \eqref{eq-var-Pb-f-H-relax-3}, we get that
    \begin{align*}
        \frac{ \operatorname{Var}_{\Pb}[f_{\mathcal H}] }{ \mathbb E_{\Pb}[f_{\mathcal H}]^2 } &\leq o(1) + \frac{ 1+o(1) }{ n^{2pm\ell} \beta_{\mathcal H}^2 } \sum_{ \substack{ S \subset \mathsf{K}_n \\ [S] \in \mathcal H } } \sum_{ 1 \leq k \leq 2m\ell-1 } \sum_{ 2h>pk } \beta_{\mathcal H} n^{pm\ell-h} (pm\ell)^{pm\ell}  \\
        &= o(1) + \frac{ 1+o(1) }{ n^{pm\ell} \beta_{\mathcal H} } \sum_{ \substack{ S \subset \mathsf{K}_n \\ [S] \in \mathcal H } } \sum_{ 1 \leq k \leq 2m\ell-1 } \sum_{ 2h>pk } \big( \lambda^{-1} n^{\frac{p}{4}} \big)^{2k} n^{-h} (pm\ell)^{pm\ell}  \\
        &\overset{\eqref{eq-condition-strong-detection}}{=} o(1) + \frac{ 1+o(1) }{ n^{pm\ell} \beta_{\mathcal H} } \sum_{ \substack{ S \subset \mathsf{K}_n \\ [S] \in \mathcal H } } \sum_{ 1 \leq k \leq 2m\ell-1 } \sum_{ 2h>pk } n^{o(1)} \cdot n^{-h+\frac{k}{2}}  \\
        &= o(1) + n^{-\frac{1}{2}+o(1)} \cdot \frac{ 1 }{ n^{pm\ell} \beta_{\mathcal H} } \#\big\{ S \subset \mathsf K_n: [S] \in \mathcal H \big\} \overset{\eqref{eq-def-beta-mathcal-H}}{=} o(1) \,. \qedhere
    \end{align*}
\end{proof}

\subsection{Proof of Theorem~\ref{MAIN-THM-detection}}{\label{subsec:proof-thm-2.6}}

With Proposition~\ref{main-prop-detection}, we are ready to prove Theorem~\ref{MAIN-THM-detection}.
\begin{proof}[Proof of Theorem~\ref{MAIN-THM-detection}]
    Under condition \eqref{eq-condition-strong-detection}, from Proposition~\ref{main-prop-detection} we get that
    \begin{align*}
        \mathbb E_{\Qb}[f_{\mathcal H}]=0, \ \mathbb E_{\Qb}[f_{\mathcal H}^2] = 1+o(1), \ \mathbb E_{\Pb}[f_{\mathcal H}] = \omega(1), \ \operatorname{Var}_{\Pb}[f_{\mathcal H}] = o(1) \cdot \mathbb E_{\Pb}[f_{\mathcal H}]^2 \,.
    \end{align*}
    Thus, for any constant $0<C<1$ we obtain
    \begin{align*}
        &\Pb\Big( f_{\mathcal H}(\bm Y) \leq C\mathbb E_{\Pb}[f_{\mathcal H}] \Big) \leq \frac{ \operatorname{Var}_{\Pb}[f_{\mathcal H}] }{ (1-C)^2 \mathbb E_{\Pb}[f_{\mathcal H}]^2 } = o(1) \,, \\
        &\Qb\Big( f_{\mathcal H}(\bm Y) \geq C\mathbb E_{\Pb}[f_{\mathcal H}] \Big) \leq \frac{ \operatorname{Var}_{\Qb}[f_{\mathcal H}] }{ C^2 \mathbb E_{\Pb}[f_{\mathcal H}]^2 } = o(1) \,. \qedhere
    \end{align*}
\end{proof}

\subsection{Proof of Proposition~\ref{main-prop-recovery}}{\label{subsec:proof-prop-2.10}}

\begin{lemma}{\label{lem-conditional-exp-given-endpoints}}
    We have 
    \begin{align*}
        \mathbb E[ f_S(\bm Y) \mid x_*(i),x_*(j) ] = \big( \lambda n^{-\frac{p}{4}} \big)^{2m\ell} x_*(i) x_*(j)
    \end{align*}
    for all $[S] \in \mathcal J$ and $\mathsf L(S)=\{ i,j \}$.
\end{lemma}
\begin{proof}
    Denote $\mu_{ \setminus\{ i,j \} }$ be the law of $\{ x_*(k): k \neq i,j \}$. We then have
    \begin{align*}
        \mathbb E[ f_S(\bm Y) \mid x_*(i),x_*(j) ] = \mathbb E_{ \{ x_*(k): k \neq i,j \} \sim \mu_{ \setminus\{ i,j \} } }\Big[ \mathbb E\big[ f_S(\bm Y) \mid x_* \big] \Big] \,.
    \end{align*}
    Recall \eqref{eq-def-f-S}, we have
    \begin{align*}
        \mathbb E\big[ f_S(\bm Y) \mid x_* \big] = \mathbb E\Big[ \prod_{e \in E(S)} \big( \lambda n^{-\frac{p}{4}} x_*^{e} + \bm G_e \big) \mid x_* \Big] = \prod_{e \in E(S)} \big( \lambda n^{-\frac{p}{4}} x_*^{e} \big) \,.
    \end{align*}
    Note that for all $[S] \in \mathcal J$ with $\mathsf L(S)=\{ i,j \}$, we have $|E(S)|=2m\ell$ and $\mathsf{deg}_S(v)=2$ for all $v \in V(S) \setminus \{ i,j \}$. Thus,
    \begin{align*}
        \prod_{e \in E(S)} \big( \lambda n^{-\frac{p}{4}} x_*^{e} \big) &= \big( \lambda n^{-\frac{p}{4}} \big)^{|E(S)|} \prod_{e \in E(S)} \prod_{v \in e} x_*(v) \\
        &= \big( \lambda n^{-\frac{p}{4}} \big)^{2m\ell} \prod_{v \in V(S)} x_*(v)^{ \mathsf{deg}_S(v) } = \big( \lambda n^{-\frac{p}{4}} \big)^{2m\ell} x_*(i) x_*(j) \,,
    \end{align*}
    leading to the desired result.
\end{proof}
Clearly, Lemma~\ref{lem-conditional-exp-given-endpoints} implies that 
\begin{align*}
    \mathbb E_{\Pb}[f_S(\bm Y)]=0 \mbox{ for all } [S] \in \mathcal J, \mathsf L(S)=\{ i,j \} \,.
\end{align*}
In addition, we see that
\begin{align*}
    \mathbb E_{\Pb}\big[ \Phi^{\mathcal J}_{i,j} \mid x_*(i), x_*(j) \big] &= \frac{ 1 }{ \lambda^{2m\ell} n^{\frac{pm\ell}{2}-1} \beta_{\mathcal J}  } \sum_{[J] \in \mathcal J} \sum_{ \substack{ S \subset \mathsf K_n: S \cong J \\ \mathsf L(S)=\{ i,j \} } } \big( \lambda n^{-\frac{p}{4}} \big)^{2m\ell} x_*(i) x_*(j) \\
    &= \frac{ x_*(i)x_*(j) }{ n^{pm\ell-1} \beta_{\mathcal J} } \sum_{[J] \in \mathcal J} \#\big\{ S \subset \mathsf K_n: S \cong J, \mathsf L(S)=\{ i,j \} \big\} \\
    & \overset{\text{Lemma~\ref{lem-prelim-mathcal-J},(1)}}{=} \frac{ 1+o(1) }{ n^{pm\ell-1} \beta_{\mathcal J} } \sum_{[J] \in \mathcal J} \frac{ n^{pm\ell-1} }{ |\mathsf{Aut}(J)| } x_*(i)x_*(j) \overset{\eqref{eq-def-beta-mathcal-H}}{=} (1+o(1)) x_*(i)x_*(j) \,.
\end{align*}
Thus, we have
\begin{align}
    \mathbb E_{\Pb}\big[ (\Phi_{i,j}^{\mathcal J}-x_*(i)x_*(j))^2 \big] &= \operatorname{Var}_{\Pb}\big[ \Phi_{i,j}^{\mathcal J}-x_*(i)x_*(j) \big] \nonumber \\
    &= \mathbb E\Big[ \operatorname{Var}_{\Pb}\big[ \Phi_{i,j}^{\mathcal J} \mid x_*(i),x_*(j) \big] \Big] + \operatorname{Var}\Big[ \mathbb E_{\Pb}\big[ \Phi_{i,j}^{\mathcal J} \mid x_*(i),x_*(j) \big] \Big] \nonumber  \\
    &= o(1) + \mathbb E\Big[ \operatorname{Var}_{\Pb}\big[ \Phi_{i,j}^{\mathcal J} \mid x_*(i),x_*(j) \big] \Big]  \,. \label{eq-var-Phi-J-i,j-relax-1}
\end{align}
So it suffices to show that
\begin{equation}{\label{eq-var-Phi-i,j-relax-2}}
    \mathbb E\Big[ \operatorname{Var}_{\Pb}\big[ \Phi_{i,j}^{\mathcal J} \mid x_*(i),x_*(j) \big] \Big] \leq \delta^2  \,.
\end{equation}
Recall \eqref{eq-def-Phi-i,j-mathcal-J}, we have
\begin{align}
    & \operatorname{Var}_{\Pb}\big[ \Phi_{i,j}^{\mathcal J} \mid x_*(i),x_*(j) \big]  \nonumber   \\
    =\ & \frac{ 1 }{ \lambda^{4m\ell} n^{pm\ell-2} \beta_{\mathcal J}^2 } \sum_{ [J],[I] \in \mathcal J } \sum_{ \substack{ S \cong J, K \cong I \\ \mathsf L(S)=\mathsf L(K)=\{ i,j \} } } \operatorname{Cov}_{\Pb}\big( f_S(\bm Y), f_{K}(\bm Y) \mid x_*(i),x_*(j) \big)  \nonumber  \\
    =\ & \frac{ 1 }{ \lambda^{4m\ell} n^{pm\ell-2} \beta_{\mathcal J}^2 } \sum_{ [J],[I] \in \mathcal J } \sum_{ \substack{ S \cong J, K \cong I \\ \mathsf L(S)=\mathsf L(K)=\{ i,j \} } } \Big( \mathbb E_{\Pb}\big[ f_S(\bm Y) f_{K}(\bm Y) \mid x_*(i),x_*(j) \big] - \lambda^{4m\ell} n^{-pm\ell} \Big) \,,  \label{eq-var-Phi-i,j-relax-3}
\end{align}
where the last equality follows from Lemma~\ref{lem-conditional-exp-given-endpoints}.
\begin{lemma}{\label{lem-conditional-exp-f-S-f-K}}
    For $[S],[K] \in \mathcal J$ with $\mathsf L(S)=\mathsf L(K)=\{ i,j \}$, we have
    \begin{align*}
        \mathbb E_{\Pb}\big[ f_S(\bm Y) f_{K}(\bm Y) \mid x_*(i),x_*(j) \big] = \big( \lambda^2 n^{-\frac{p}{2}} + 1 \big)^{|E(S)\cap E(K)|} \big( \lambda n^{-\frac{p}{4}} \big)^{ |E(S) \triangle E(K)| } \,.
    \end{align*}
\end{lemma}
\begin{proof}
    Denote $\mu_{i,j}=\mu(\cdot \mid x(i)=x_*(i),x_j=x_*(j))$, then
    \begin{align}
        & \mathbb E_{\Pb}\big[ f_S(\bm Y) f_{K}(\bm Y) \mid x_*(i),x_*(j) \big] = \mathbb E_{x \sim \mu_{i,j}} \mathbb E_{\Pb}\Big[ f_S(\bm Y) f_K(\bm Y) \mid x_*=x \Big] \nonumber  \\
        \overset{\eqref{eq-def-f-S}}{=}\ & \mathbb E_{x \sim \mu_{i,j}} \mathbb E_{\Pb}\Big[ \prod_{ e \in E(S) \triangle E(K) } \big( \lambda n^{-\frac{p}{4}} x^{e} + \bm G_e \big) \prod_{ e \in E(S) \cap E(K) } \big( \lambda n^{-\frac{p}{4}} x^{e} + \bm G_e \big)^2 \mid x \Big] \nonumber  \\
        =\ & \mathbb E_{x \sim \mu_{i,j}}\Big[ \prod_{ e \in E(S) \triangle E(K) } \big( \lambda n^{-\frac{p}{4}} x^{e} \big) \prod_{ e \in E(S) \cap E(K) } \big( \lambda^2 n^{-\frac{p}{2}} + 1 \big) \Big] \,.  \label{eq-conditional-exp-f-S-f-K-relax-1}
    \end{align}
    Using Definition~\ref{def-mathcal-J}, we see that $\mathsf{deg}_S(v), \mathsf{deg}_K(v) \in \{ 0,2 \}$ for all $v \in [n] \setminus \{ i,j \}$. Thus, 
    \begin{align*}
        \prod_{e \in E(S) \triangle E(K)} x^{e} &= \prod_{e \in E(S) \triangle E(K)} \prod_{v \in e} x_v = \prod_{v \in V(S) \cup V(K)} x_v^{ \#\{ e \in \in E(S) \triangle E(K): v \in e \} } \\
        &= \prod_{v \in V(S) \cup V(K)} x_v^{ \mathsf{deg}_S(v) + \mathsf{deg}_K(v) - 2\#\{ e \in \in E(S) \cap E(K): v \in e \} } =1 \,.
    \end{align*}
    Plugging this equality into \eqref{eq-conditional-exp-f-S-f-K-relax-1}, we get that
    \begin{equation*}
        \eqref{eq-conditional-exp-f-S-f-K-relax-1} = \big( \lambda^2 n^{-\frac{p}{2}} + 1 \big)^{ |E(S) \cap E(K)| } \big( \lambda n^{-\frac{p}{4}} \big)^{ |E(S) \triangle E(K)| } \,.  \qedhere
    \end{equation*}
\end{proof}
Using Lemma~\ref{lem-conditional-exp-f-S-f-K}, we see that when $E(S) \cap E(K)=\emptyset$ 
\begin{align*}
    \mathbb E_{\Pb}\big[ f_S(\bm Y) f_{K}(\bm Y) \mid x_*(i),x_*(j) \big] - \lambda^{4m\ell} n^{-pm\ell} = 0 \,.
\end{align*}
In addition, for $E(S) \cap E(K) \neq \emptyset$
\begin{align*}
    \mathbb E_{\Pb}\big[ f_S(\bm Y) f_{K}(\bm Y) \mid x_*(i),x_*(j) \big] - \lambda^{4m\ell} n^{-pm\ell} &\leq \big( \lambda^2 n^{-\frac{p}{2}} + 1 \big)^{ |E(S) \cap E(K)| } \big( \lambda n^{-\frac{p}{4}} \big)^{ |E(S) \triangle E(K)| } \\
    &\leq [1+o(1)] \big( \lambda n^{-\frac{p}{4}} \big)^{ 4m\ell-2|E(S) \cap E(K)| }
\end{align*}
So we have
\begin{align}
    \eqref{eq-var-Phi-i,j-relax-3} &\leq \frac{ 1+o(1) }{ \lambda^{4m\ell} n^{pm\ell-2} \beta_{\mathcal J}^2 } \sum_{ \substack{ [S], [K] \in \mathcal J: E(K) \cap E(K) \neq \emptyset \\ \mathsf L(S)=\mathsf L(K)=\{ i,j \} } } \big( \lambda n^{-\frac{p}{4}} \big)^{ 4m\ell-2|E(S) \cap E(K)| } \nonumber  \\
    &= \frac{ 1+o(1) }{ n^{2pm\ell-2} \beta_{\mathcal J}^2 } \sum_{ \substack{ [S], [K] \in \mathcal J: E(K) \cap E(K) \neq \emptyset \\ \mathsf L(S)=\mathsf L(K)=\{ i,j \} } } \big( \lambda^{-2} n^{\frac{p}{2}} \big)^{ |E(S) \cap E(K)| }  \,. \label{eq-var-Phi-i,j-relax-4}
\end{align}
Thus, it suffices to show the following lemma.
\begin{lemma}{\label{lem-very-technical-sec-3}}
    Suppose \eqref{eq-condition-weak-recovery} holds. Then
    \begin{align*}
        \frac{ 1 }{ n^{2pm\ell-2} \beta_{\mathcal J}^2 } \sum_{ \substack{ [S], [K] \in \mathcal J: E(S) \cap E(K) \neq \emptyset \\ \mathsf L(S)=\mathsf L(K)=\{ i,j \} } } \big( \lambda^{-2} n^{\frac{p}{2}} \big)^{ |E(S) \cap E(K)| } \leq \delta^2 \,.
    \end{align*}
\end{lemma}

To prove Lemma~\ref{lem-very-technical-sec-3}, we first introduce some notations. For $S,K \subset \mathsf K_n$ with $[S],[K] \in \mathcal J$, denote 
\begin{equation}{\label{eq-decomposition-S}}
    S = U_1 \cup \ldots \cup U_{\ell} \mbox{ with } \mathsf L(U_{\mathtt k}) = \{ v_{\mathtt k},v_{\mathtt k+1} \}, V(U_{\mathtt k}) \cap V(U_{\mathtt k+1}) = \{ v_{\mathtt k+1} \}, v_{1}=i, v_{\ell+1}=j
\end{equation}
and 
\begin{equation}{\label{eq-decomposition-K}}
    K = U_1' \cup \ldots \cup U_{\ell}' \mbox{ with } \mathsf L(U_{\mathtt k}') = \{ v_{\mathtt k}',v_{\mathtt k+1}' \}, V(U_{\mathtt k}') \cap V(U_{\mathtt k+1}') = \{ v_{\mathtt k+1}' \}, v_{1}'=i, v_{\ell+1}'=j
\end{equation}
are decomposition as in Definition~\ref{def-mathcal-J} (using Item~(4) in Definition~\ref{def-mathcal-J} and Item~(2) in Lemma~\ref{lem-prelim-mathcal-J}, it is clear that such decompositions are unique). Define
\begin{equation}{\label{eq-def-mathtt-W}}
    \mathsf{IND}(S,K):= \big\{ 1 \leq \mathtt k \leq \ell: U_{\mathtt k} \subset K \big\} \,.
\end{equation}
Using Item~(2) in Lemma~\ref{lem-prelim-mathcal-J}, we see that for all $\mathtt k \in \mathsf{IND}(S,K)$ there exists some $1 \leq \mathtt m \leq \ell$ such that $U_{\mathtt k}=U'_{\mathtt m}$. Thus, there exists a function $\phi:\mathsf{IND}(S,K) \to[\ell]$ such that $U_{\mathtt k}= U'_{\phi(\mathtt k)}$ for all $\mathtt k \in \mathsf{IND}(S,K)$. In addition, since $V(U_{\mathtt k}) \cap V(U_{\mathtt m}) \neq \emptyset$ if and only if $|\mathtt m-\mathtt k|=1$, we see that (we in addition define $\phi(0)=0$ and $\phi(\ell+1)=\ell+1$) 
\begin{align}
    |\phi(\mathtt k)-\phi(\mathtt m)|=1 \mbox{ if and only if } |\mathtt k-\mathtt m|=1, \mathtt k,\mathtt m \in \mathsf{IND}(S,K) \cup \{ 0,\ell \} \,. \label{eq-property-phi}
\end{align}
In addition, given $\mathsf{IND}(S,K)$ we can uniquely determine a sequence $\mathsf{SEQ}(S,K)=(\mathtt s_1,\mathtt q_1,\ldots,\mathtt s_{\mathtt T}, \mathtt q_{\mathtt T})$ such that
\begin{equation}{\label{eq-def-SEQ(S,K)}}
    \begin{aligned}
        & 0 = \mathtt s_1 < \mathtt q_1 \leq \mathtt s_2 < \mathtt q_2 \ldots \leq \mathtt s_{\mathtt T} < \mathtt q_{\mathtt T} = \ell+1 \,, \\
        & \{ 1,\ldots,\ell \} \setminus \mathsf{IND}(S,K) = [\mathtt q_1,\mathtt s_2] \cup \ldots \cup [\mathtt q_{\mathtt T-1},\mathtt s_{\mathtt T}] \,.
    \end{aligned}
\end{equation}
(if $\mathsf{IND}(S,K)=\{ 1,\ldots,\ell \}$ we simply let $\mathtt T=1$). Then, it is straightforward to check that there exists a unique sequence $\overline{\mathsf{SEQ}}(S,K)=(\mathtt s_1',\mathtt q_1',\ldots,\mathtt s_{\mathtt T}', \mathtt q_{\mathtt T}')$ such that
\begin{equation}{\label{eq-def-overline-SEQ(S,K)}}
    \begin{aligned}
        & \mathtt q_1'= \mathtt q_1, \mathtt s_{\mathtt T}'=\mathtt s_{\mathtt T} \mbox{ and } \{ \mathtt q_2' - \mathtt s_2', \ldots, \mathtt q_{\mathtt T-1}'-\mathtt s_{\mathtt T-2}' \} = \{ \mathtt q_2 - \mathtt s_2, \ldots, \mathtt q_{\mathtt T-1}-s_{\mathtt T-2} \} \,,   \\
        & \{ 1,\ldots,\ell \} \setminus \phi(\mathsf{IND}(S,K)) = [\mathtt q_1',\mathtt s_2'] \cup \ldots \cup [\mathtt q_{\mathtt T-1}',\mathtt s_{\mathtt T}'] \,.
    \end{aligned}
\end{equation}
Finally, define $U_{[\mathtt q_{\mathtt t},\mathtt s_{\mathtt t+1}]} = \cup_{ \mathtt q_{\mathtt t} \leq \mathtt k \leq \mathtt s_{\mathtt t+1} } U_{\mathtt k}$, define
\begin{equation}{\label{eq-def-gamma-t}}
    \gamma_{\mathtt t}= \#\Gamma_{\mathtt t} := \#\{ 1 \leq \mathtt k \leq \ell: E(U_{\mathtt k}') \cap E(U_{[\mathtt q_{\mathtt t},\mathtt s_{\mathtt t+1}]}) \neq \emptyset \} \mbox{ for all } 1 \leq \mathtt t \leq \mathtt T-1 \,.
\end{equation}
We now divide the summation \eqref{eq-var-Phi-i,j-relax-4} into several different parts.
\begin{DEF}{\label{def-different-parts}}
    We write
    \begin{enumerate}
        \item[(I)] $(S,K) \in \mathsf{Part}_{I}$, if $\mathtt T=1$ in \eqref{eq-def-SEQ(S,K)};
        \item[(II)] $(S,K) \in \mathsf{Part}_{II}$, if $\mathtt T=2$ in \eqref{eq-def-SEQ(S,K)} and $\gamma_1=0$ in \eqref{eq-def-gamma-t};
        \item[(III)] $(S,K) \in \mathsf{Part}_{III}$, if $\mathtt T=2$ in \eqref{eq-def-SEQ(S,K)} and $\gamma_1 \geq 1$ in \eqref{eq-def-gamma-t};
        \item[(IV)] $(S,K) \in \mathsf{Part}_{IV}$, if $\mathtt T\geq 3$ in \eqref{eq-def-SEQ(S,K)}.
    \end{enumerate}
\end{DEF}
We will control the contribution in \eqref{eq-var-Phi-i,j-relax-4} from four parts separately.

\begin{lemma}{\label{lem-very-technical-sec-3-transfer}}
    Suppose \eqref{eq-condition-weak-recovery} holds. We have the following estimates: for any $[K] \in \mathcal J$ with $\mathsf{L}(K)=\{ i,j \}$
    \begin{align}
        &\sum_{ S:(S,K) \in \mathsf{Part}_I } \big( \lambda^{-2} n^{\frac{p}{2}} \big)^{ |E(S) \cap E(K)| } = \lambda^{-4m\ell} n^{pm\ell} \,; \label{eq-technical-part-I} \\
        &\sum_{ S:(S,K) \in \mathsf{Part}_{II} } \big( \lambda^{-2} n^{\frac{p}{2}} \big)^{ |E(S) \cap E(K)| } \leq \delta^2 \beta_{\mathcal J} n^{pm\ell-1} \,;  \label{eq-technical-part-II} \\
        &\sum_{ S:(S,K) \in \mathsf{Part}_{III} } \big( \lambda^{-2} n^{\frac{p}{2}} \big)^{ |E(S) \cap E(K)| } = n^{-\frac{1}{2}+o(1)} \cdot \beta_{\mathcal J} n^{pm\ell-1} \,; \label{eq-technical-part-III} \\
        &\sum_{ S:(S,K) \in \mathsf{Part}_{VI} } \big( \lambda^{-2} n^{\frac{p}{2}} \big)^{ |E(S) \cap E(K)| } = n^{-1+o(1)} \cdot \beta_{\mathcal J} n^{pm\ell-1} \,.  \label{eq-technical-part-IV}
    \end{align}
\end{lemma}

The proof of Lemma~\ref{lem-very-technical-sec-3-transfer} is the most technical part in our paper and we postpone it to Section~\ref{sec:supp-proofs-sec-3}. Provided with Lemma~\ref{lem-very-technical-sec-3-transfer}, we can now provide the proof of Lemma~\ref{lem-very-technical-sec-3}.
\begin{proof}[Proof of Lemma~\ref{lem-very-technical-sec-3} assuming Lemma~\ref{lem-very-technical-sec-3-transfer}]
    Combining \eqref{eq-technical-part-I}--\eqref{eq-technical-part-IV}, we see that for each fixed $[K] \in \mathcal J$ with $\mathsf{L}(K)=\{ i,j \}$,
    \begin{align*}
        \sum_{ \substack{ S \subset \mathsf K_n: E(S) \cap E(K) \neq \emptyset \\ [S] \in \mathcal J, \mathsf{L}(S)=\{ i,j \} } } \big( \lambda^{-2} n^{\frac{p}{2}} \big)^{ |E(S) \cap E(K)| } &\leq \beta_{\mathcal J} n^{pm\ell-1} \cdot \Big( n \lambda^{-4m\ell} \beta_{\mathcal J}^{-1} + \delta^2 + n^{-\frac{1}{2}+o(1)} + n^{-1+o(1)} \Big) \\
        & \overset{\eqref{eq-condition-weak-recovery}}{=} (\delta^2+o(1)) \beta_{\mathcal J} n^{pm\ell-1} \,.
    \end{align*}
    Thus, we have
    \begin{align*}
        \frac{ 1 }{ n^{2pm\ell-2} \beta_{\mathcal J}^2 } \sum_{ \substack{ [S], [K] \in \mathcal J: E(S) \cap E(K) \neq \emptyset \\ \mathsf L(S)=\mathsf L(K)=\{ i,j \} } } &\leq \frac{ \delta^2+o(1) }{ n^{2pm\ell-2} \beta_{\mathcal J}^2 } \sum_{[J] \in \mathcal J} \#\big\{ K \subset\mathsf K_n: K \cong J, \mathsf{L}(K)=\{ i,j \} \big\} \\
        &\overset{\text{Lemma~\ref{lem-prelim-mathcal-J},(1)}}{=} \frac{ \delta^2+o(1) }{ n^{2pm\ell-2} \beta_{\mathcal J}^2 } \frac{ n^{mp\ell-1} }{ |\mathsf{Aut}(J)| } \overset{\eqref{eq-def-beta-mathcal-H}}{=} \delta^2 + o(1) \,. \qedhere
    \end{align*}
\end{proof}

\subsection{Proof of Theorem~\ref{MAIN-THM-recovery}}{\label{subsec:proof-thm-2.11}}

With Proposition~\ref{main-prop-recovery}, we are ready to prove Theorem~\ref{MAIN-THM-recovery}.
\begin{proof}[Proof of Theorem~\ref{MAIN-THM-recovery}]
    For all $j \neq i$, note that 
    \begin{align*}
        \big( \widehat{x}(j) - x_*(i)x_*(j) \big)^2 \leq \big( \Phi_{i,j}^{\mathcal J} - x_*(i)x_*(j) \big)^2 \,.
    \end{align*}
    Thus, suppose \eqref{eq-condition-weak-recovery} holds, using Proposition~\ref{main-prop-recovery} we see that
    \begin{align*}
        \mathbb E_{\Pb}\Big[ \big( \widehat{x}(j) - x_*(i)x_*(j) \big)^2 \Big] \leq \mathbb E_{\Pb}\Big[ \big( \Phi_{i,j}^{\mathcal J} - x_*(i)x_*(j) \big)^2 \Big] \leq \delta^2 \,.
    \end{align*}
    Thus, we have
    \begin{align*}
        \mathbb E_{\Pb}\Big[ \big\| \widehat{x} - x_*(i) \cdot x_*  \big\|^2 \Big] \leq \delta^2 n \,.
    \end{align*}
    By Markov inequality, we then have
    \begin{align*}
        \Pb\Big( \big\| \widehat{x} - x_*(i) \cdot x_*  \big\|^2 \leq \delta n \Big) \geq 1-\delta \,.
    \end{align*}
    Combined with the fact that $\| \widehat x \|=\| x_* \|=\sqrt{n}$, we see that
    \begin{equation*}
        \Pb\Big( \frac{ \langle \widehat x,x_* \rangle }{ \| \widehat x \| \| x_* \| } \geq 1-\delta \Big) = \Pb\Big( \big\| \widehat{x} - x_*(i) \cdot x_*  \big\|^2 \leq 2\delta n \Big) \geq 1- \delta \,. \qedhere
    \end{equation*}
\end{proof}

\section{Approximate statistics by hypergraph color coding}{\label{sec:hypergraph-color-coding}}

In this section, we provide an efficient algorithm to approximately compute the detection statistic $f_{\mathcal H}$ given in \eqref{eq-def-f-mathcal-H} and the recovery statistic $\{ \Phi^{\mathcal J}_{i,j} \}$ given in \eqref{eq-def-Phi-i,j-mathcal-J}, using the idea of hypergraph color coding.

\subsection{Approximate the detection statistics}{\label{subsec:approx-detection}}

Recall Definition~\ref{def-mathcal-H}. By applying the hypergraph color coding method, we first approximately count the signed subgraphs that are isomorphic to a query hypergraph in $\mathcal H$ with $2m\ell$ edges and $mp\ell$ vertices. Specifically, recall \eqref{eq-def-tensor-PCA}, we generate a random coloring $\tau:[n] \to [mp\ell]$ that assigns a color to each vertex of $[n]$ from the color set $mp\ell$ independently and uniformly at random. Given any $V \subset [n]$, let $\chi_{\tau}(V)$ indicate that $\tau(V)$ is colorful, i.e., $\tau(x) \neq \tau(y)$ for any distinct $x,y \in V$. In particular, if $|V|=mp\ell$, then $\chi_{\tau}(V)=1$ with probability
\begin{equation}{\label{eq-def-r}}
    r := \frac{ (mp\ell)! }{ (mp\ell)^{mp\ell} } = e^{-\Theta(mp\ell)} \,.
\end{equation}
For any $p$-uniform unlabeled hypergraph $[H]$ with $mp\ell$ vertices, we define
\begin{equation}{\label{eq-def-g-S}}
    g_{H}(\bm Y,\tau) := \sum_{ S \subset \mathsf K_n, S \cong H } \chi_{\tau}(V(S)) \prod_{e \in E(S)} \bm Y_e \,.
\end{equation}
Then $\mathbb E[ g_{H}(\bm Y,\tau) \mid \bm Y ] = r\sum_{S \cong H} f_S(\bm Y)$ where $f_S(\bm Y)$ was defined in \eqref{eq-def-f-mathcal-H}. To further obtain an accurate approximation of $f_S(\bm Y)$, we average over multiple copies of $g_S(\bm Y,\tau)$ by generating $t$ independent random colorings, where
\begin{align*}
    t := \lceil 1/r \rceil \,.
\end{align*}
Next, we plug in the averaged subgraph count to approximately compute $f_{\mathcal H}(\bm Y)$. Specifically, we generate $t$ random colorings $ \{ \tau_{\mathtt k} \}_{\mathtt k=1}^{t}$ which are independent copies of $\tau$ that map $[n]$ to $[mp\ell]$. Then, we define
\begin{equation}{\label{eq-def-widetilde-f-H}}
    \widetilde{f}_{\mathcal H} := \frac{1}{\sqrt{ n^{mp\ell}\beta_{\mathcal H} }}  \sum_{ [H] \in \mathcal H } \Big( \frac{1}{tr} \sum_{\mathtt k=1}^{t} g_H(\bm Y,\tau_{\mathtt k}) \Big) \,.
\end{equation}
The following result shows that $\widetilde{f}_{\mathcal H}(\bm Y)$ well approximates $f_{\mathcal H}(\bm Y)$ in a relative sense.
\begin{proposition}{\label{prop-approximate-detection-statistics}}
    Suppose \eqref{eq-condition-strong-detection} holds. Then under both $\Pb$ and $\Qb$
    \begin{align}{\label{eq-approximate-detection-statistics}}
        \frac{ \widetilde{f}_{\mathcal H} - f_{\mathcal H} }{ \mathbb E_{\Pb}[f_{\mathcal H}] } \overset{L_2}{\longrightarrow} 0  \,.
    \end{align}
\end{proposition}
Since the convergence in $L_2$ implies the convergence in probability, as an immediate corollary of Theorem~\ref{MAIN-THM-detection} and Proposition~\ref{prop-approximate-detection-statistics}, we conclude that the approximate test statistic $\widetilde{f}_{\mathcal H}$ succeeds under the same condition as the original test statistic $f_{\mathcal H}$. Now we show the approximate hypergraph counts $g_{H}(\bm Y,\tau)$ can be computed efficiently via the following algorithm.
\begin{breakablealgorithm}{\label{alg:dynamic-programming}}
    \caption{Computation of $g_{H}(\bm Y,\tau)$}
    \begin{algorithmic}[1]
    \STATE \textbf{Input}: A weighted adjacency tensor $\bm Y$ on $[n]$ with its vertices colored by $\tau$, and an element $[H] \in \mathcal H$ with $H = U_1 \cup \ldots U_{\ell}$ be the decomposition satisfying Definition~\ref{def-mathcal-H}. 
    \STATE Compute $\mathfrak a(H) = \mathfrak a_+(H)+\mathfrak a_{-}(H)$, where $\mathfrak a_{+}(H)=\#\{ 1 \leq k \leq \ell: \exists \pi \in \mathsf{Aut}(H), \pi(v_i)=v_{i+k} \}$ and $\mathfrak a_{-}(H)=\#\{ 1 \leq k \leq \ell: \exists \pi \in \mathsf{Aut}(H), \pi(v_i)=v_{\ell+i-k} \}$.
    \STATE Denote $\widehat{H}=U_2 \cup \ldots \cup U_{\ell}$. Choose the endpoints of $\widehat{H}$ as $\{ v_2,v_\ell \}$.
    \STATE For every subset $C \in [mp\ell]$ with $|C|=mp+1$, every distinct $c_1,c_2 \in C$, every possible $U_{\ell}$ and every distinct $x,y \in [n]$, compute $Y_{\ell}(x,y;c_1,c_2;C,U_{\ell})= \Lambda(x,y;c_1,c_2;C,[U_{\ell}])$, where for $[U] \in \mathcal U_m$
    \begin{align*}
        \Lambda( x,y;c_1,c_2;C,[U] ) = \sum_{ S \cong U } f_{S}(\bm Y) \cdot \mathbf 1\Big\{ \tau(V(S))=C, \mathsf{L}(S)=\{ x,y \}, \tau(x)=c_1, \tau(y)=c_2 \Big\} \,.
    \end{align*}
    \STATE For every $1 \leq k \leq \ell-2$ and $c_1,c_2 \in C, |C|=mp(k+1)+1$, compute recursively  
    \begin{align*}
       &Y_{\ell-k}(x,y;c_1,c_2,C;U_{\ell-k} \cup \ldots \cup U_{\ell}) \\
       =\ & \sum_{c \in C} \sum_{ \substack{ (C_1,C_2): C_1 \cup C_2 = C \\ c_1 \in C_1, c_2 \in C_2 \\ C_1 \cap C_2 = \{ c \} } } \sum_{ \substack{ z \in [n] \\ z \neq x,y } } \Lambda( x,z;c_1,c;C_1,[U_{\ell-k}] ) Y_{\ell-k+1}(y,z;c,c_2,C_2;U_{\ell-k+1} \cup \ldots \cup U_{\ell}) \,.
    \end{align*}
    \STATE Compute
    \begin{align*}
        g_H(\bm Y;\tau) =\ & \frac{1}{\mathfrak a(H)} \sum_{ \substack{ c_1,c_2 \in [mp\ell] \\ c_1 \neq c_2 } } \sum_{ \substack{ (C_1,C_2): C_1 \cup C_2 = [mp\ell] \\ C_1 \cap C_2 = \{ c_1,c_2 \} } } \sum_{ \substack{ x,y \in [n] \\ x \neq y } } \Lambda( x,y;c_1,c_2;C_1,[U_{1}] ) \\
        & \cdot Y_{2}(y,x;c_2,c_1,C_2;U_{2} \cup \ldots \cup U_{\ell}) \,.
    \end{align*}
    \textbf{Output}: $g_H(\bm Y;\tau)$.
    \end{algorithmic}
\end{breakablealgorithm}

\begin{lemma}{\label{lem-color-coding}}
    For any coloring $\tau:[n] \to [K+1]$ and any hypergraph $[H] \in \mathcal H$, Algorithm~\ref{alg:dynamic-programming} computes $g_H(\bm Y,\tau)$ in time $O(n^{mp+3+o(1)})$.
\end{lemma}

Finally, we show that the approximate test statistic $\widetilde{f}$ can be computed efficiently.
\begin{breakablealgorithm}{\label{alg:cal-widetilde-f}}
    \caption{Computation of $\widetilde{f}_{\mathcal H}$}
    \begin{algorithmic}[1]
    \STATE {\bf Input:} Adjacency tensor $\bm Y$, signal-to-noise ratio parameter $\lambda$, dimension parameter $p$, and integers $m,\ell$.
    \STATE List all non-isomorphic, unrooted unlabeled hypergraphs satisfying Definition~\ref{def-component-hypergraph}. Denote this list as $\mathcal U_m$.
    \STATE List all non-isomorphic, unrooted unlabeled hypergraphs satisfying Definition~\ref{def-mathcal-H} in the following way: list all hypergraphs sequences $(U_1,\ldots,U_{\ell})$ such that $U_i \in \mathcal U_m$, $\mathsf L(U_i)=\{ v_i,v_{i+1} \}$, $V(U_i) \cap V(U_{i+1})=\{ v_{i+1} \}$, and delete the sequences that generates isomorphic $H=U_1 \cup \ldots \cup U_{\ell}$. Denote this list as $\mathcal H$.
    \STATE For each $[H] \in \mathcal H$, compute $\operatorname{Aut}(H)$ via Algorithm~\ref{alg:cal-Aut-H-in-mathcal-H}. 
    \STATE Generate i.i.d.\ random colorings $\{ \tau_{\mathtt k} : 1 \leq \mathtt k \leq t \}$ that map $[n]$ to $[mp\ell]$ uniformly at random. 
    \FOR{each $1 \leq \mathtt k \leq t$}
    \STATE For each $[H] \in \mathcal H$, compute $g_{H}(\bm Y,\tau_{\mathtt k})$ via Algorithm~\ref{alg:dynamic-programming}. 
    \ENDFOR
    \STATE Compute $\widetilde{f}_{\mathcal H}(\bm Y)$ according to \eqref{eq-def-widetilde-f-H}.
    \STATE {\bf Output:} $\widetilde{f}_{\mathcal H}(\bm Y)$.
    \end{algorithmic}
\end{breakablealgorithm}

\begin{proposition}{\label{prop-running-time-detection}}
    Under condition~\ref{eq-condition-strong-detection}, Algorithm~\ref{alg:cal-widetilde-f} computes $\widetilde{f}_{\mathcal H}(\bm Y)$ in time $n^{C+o(1)}$ for some constant $C=C(m,p)$. 
\end{proposition}
\begin{proof}
    Listing all non-isomorphic, unrooted unlabeled hypergraphs in $\mathcal U_m$ takes time $\binom{ \binom{mp+1}{p} }{ 2m }=n^{O(1)}$. Listing all non-isomorphic, unrooted unlabeled hypergraphs satisfying Definition~\ref{def-mathcal-H} in Step~3 takes time
    \begin{align*}
        O\Big( |\mathcal U_m|^{\ell} \Big) = O\Big( \tbinom{ (mp+1)^{p} }{ 2m }^{\ell} \Big) = n^{o(1)} \,.
    \end{align*}
    From Corollary~\ref{cor-bound-Aut-H-in-mathcal-H}, calculating $\mathsf{Aut}(H)$ for $[H]\in \mathcal H$ takes time
    \begin{align*}
        |\mathcal H| \cdot O\Big( 2\ell \cdot ((mp+1)!)^{\ell} \Big) = O\Big( \tbinom{ (mp+1)^{p} }{ 2m }^{\ell} \Big) \cdot O\Big( 2\ell \cdot ((mp+1)!)^{\ell} \Big) = n^{o(1)} \,.
    \end{align*}
    For all $1 \leq i \leq t$, calculating $g_{H}(\bm Y,\tau_i)$ takes time $n^{mp+3+o(1)}$. Thus, the total running time of Algorithm~\ref{alg:cal-widetilde-f} is 
    \begin{equation*}
        O(n^{o(1)}) + O(n^{o(1)}) + 2t \cdot O(n^{mp+3+o(1)}) = O(n^{mp+3+o(1)}) \,. \qedhere
    \end{equation*}
\end{proof}
\begin{remark}
    Recall \eqref{eq-suffice-bound-m}. The running time of Algorithm~\ref{alg:cal-widetilde-f} can be chosen as $n^{C(\lambda)+o(1)}$, where $C(\lambda)=O_p(1) \cdot \max\{1,\lambda\}^{-\frac{4p}{p-2}}$.
\end{remark}

We complete this subsection by pointing out that Theorem~\ref{MAIN-THM-detection-algorithmic} follows directly from Proposition~\ref{prop-approximate-detection-statistics} and Proposition~\ref{prop-running-time-detection}.

\subsection{Approximate the recovery statistics}{\label{subsec:approx-recovery}}

Recall Definition~\ref{def-mathcal-J}, we generate a random coloring $\xi:[n] \to [mp\ell]$ that assigns a color to each vertex of $[n]$ from the color set $mp\ell$ independently and uniformly at random. Given any $V \subset [n]$, let $\chi_{\xi}(V)$ indicate that $\xi(V)$ is colorful, i.e., $\xi(x) \neq \xi(y)$ for any distinct $x,y \in V$. In particular, if $|V|=mp\ell$, then $\chi_{\xi}(V)=1$ with probability
\begin{equation}{\label{eq-def-varsigma}}
    \varkappa := \frac{ (mp\ell+1)! }{ (mp\ell+1)^{mp\ell+1} } \,.
\end{equation}
For any $p$-uniform unlabeled hypergraph $[J] \in \mathcal J$ with $mp\ell$ vertices, we define
\begin{equation}{\label{eq-def-g-J-xi}}
    h_{J}(\bm Y,\xi) := \sum_{ \substack{ S \subset \mathsf K_n: S \cong J \\ \mathsf{L}(S)=\{ i,j \} } } \chi_{\xi}(V(S)) \prod_{e \in E(S)} \bm Y_e \,.
\end{equation}
Then $\mathbb E[ g_{J}(\bm Y,\xi) \mid \bm Y ] = \varkappa \sum_{S \cong J, \mathsf{L}(S)=\{i,j\} } f_S(\bm Y)$ where $f_S(\bm Y)$ was defined in \eqref{eq-def-f-mathcal-H}. Next, we define $t=\lceil \frac{1}{\varkappa} \rceil$ and generate $t$ random colorings $ \{ \xi_{\mathtt k} \}_{\mathtt k=1}^{t}$ which are independent copies of $\xi$ that map $[n]$ to $[mp\ell]$. Then, we define
\begin{equation}{\label{eq-def-widetilde-Phi-H}}
    \widetilde{\Phi}^{\mathcal J}_{i,j} := \frac{1}{ \lambda^{2m\ell} n^{\frac{mp\ell}{2}-1} \beta_{\mathcal J} } \sum_{ [J] \in \mathcal J } \Big( \frac{1}{t\varkappa} \sum_{\mathtt k=1}^{t} g_J(\bm Y,\xi_{\mathtt k}) \Big) \,.
\end{equation}
Moreover, the following result bounds the approximation error under
the same conditions as those in Proposition~\ref{main-prop-recovery} for the second moment calculation.
\begin{proposition}{\label{prop-approximate-recovery-statistics}}
    Suppose \eqref{eq-condition-weak-recovery} holds. Then
    \begin{equation}{\label{eq-approximate-recovery-statistics}}
        \mathbb E_{\Pb}\Big[ \big( \widetilde{\Phi}_{i,j}^{\mathcal J} - \Phi^{\mathcal J}_{i,j} \big)^2 \Big] \leq \delta^2 \,.
    \end{equation}
\end{proposition}
Using triangle inequality, as an immediate corollary of Proposition~\ref{main-prop-recovery} and Proposition~\ref{prop-approximate-recovery-statistics}, we conclude that the approximate test statistic $\{ \widetilde{\Phi}^{\mathcal J}_{i,j} \}$ satisfy \eqref{eq-L2-estimation-error} with $\delta$ replaced by $2\delta$. Now we show the approximate hypergraph counts $h_{J}(\bm Y,\xi)$ can be computed efficiently via the following algorithm.
\begin{breakablealgorithm}{\label{alg:dynamic-programming-recovery}}
    \caption{Computation of $h_{J}(\bm Y,\xi)$}
    \begin{algorithmic}[1]
    \STATE \textbf{Input}: A weighted adjacency tensor $\bm Y$ on $[n]$ with its vertices colored by $\xi$, and an element $[J] \in \mathcal J$. 
    \STATE Let $H = U_1 \cup \ldots U_{\ell}$ be the decomposition in Definition~\ref{def-mathcal-H}. 
    \STATE For every subset $C \in [mp\ell]$ with $|C|=mp+1$, every distinct $c_1,c_2 \in C$, every possible $U_{\ell}$ and every distinct $x,y \in [n]$, compute $Y_{\ell}(x,y;c_1,c_2;C,U_{\ell})= \Lambda(x,y;c_1,c_2;C,[U_{\ell}])$, where for $[U] \in \mathcal U_m$
    \begin{align*}
        \Lambda( x,y;c_1,c_2;C,[U] ) = \sum_{ S \cong U } f_{S}(\bm Y) \cdot \mathbf 1\Big\{ \xi(V(S))=C, \mathsf{L}(S)=\{ x,y \}, \tau(x)=c_1, \tau(y)=c_2 \Big\} \,.
    \end{align*}
    \STATE For every $1 \leq k \leq \ell-1$ and $c_1,c_2 \in C, |C|=mp(k+1)+1$, compute recursively  
    \begin{align*}
       &Y_{\ell-k}(x,y;c_1,c_2,C;U_{\ell-k} \cup \ldots \cup U_{\ell}) \\
       =\ & \sum_{c \in C} \sum_{ \substack{ (C_1,C_2): C_1 \cup C_2 = C \\ c_1 \in C_1, c_2 \in C_2 \\ C_1 \cap C_2 = \{ c \} } } \sum_{ \substack{ z \in [n] \\ z \neq x,y } } \Lambda( x,z;c_1,c;C_1,[U_{\ell-k}] ) Y_{\ell-k+1}(y,z;c,c_2,C_2;U_{\ell-k+1} \cup \ldots \cup U_{\ell}) \,.
    \end{align*}
    \STATE Compute
    \begin{align*}
        h_J(\bm Y,\xi) = \sum_{ c_1\neq c_2 } Y_{1}(i,j;c_1,c_2,C;U_{1} \cup \ldots \cup U_{\ell})
    \end{align*}
    \textbf{Output}: $h_J(\bm Y;\xi)$.
    \end{algorithmic}
\end{breakablealgorithm}

\begin{lemma}{\label{lem-color-coding-recovery}}
    For any coloring $\xi:[n] \to [mp\ell+1]$ and any hypergraph $[J] \in \mathcal J$, Algorithm~\ref{alg:dynamic-programming} computes $g_H(\bm Y,\xi)$ in time $O(n^{2mp+3+o(1)})$.
\end{lemma}

Finally, we show that the approximate test statistic $\widetilde{f}$ can be computed efficiently.
\begin{breakablealgorithm}{\label{alg:cal-widetilde-Phi}}
    \caption{Computation of $\widetilde{\Phi}^{\mathcal J}_{i,j}$}
    \begin{algorithmic}[1]
    \STATE {\bf Input:} Adjacency tensor $\bm Y$, signal-to-noise ratio parameter $\lambda$, dimension parameter $p$, and integers $m,\ell$.
    \STATE List all non-isomorphic, unrooted unlabeled hypergraphs satisfying Definition~\ref{def-component-hypergraph}. Denote this list as $\mathcal U$.
    \STATE List all non-isomorphic, unrooted unlabeled hypergraphs satisfying Definition~\ref{def-mathcal-H} in the following way: list all hypergraphs sequences $(U_1,\ldots,U_{\ell})$ such that $U_i \in \mathcal U$, $\mathsf L(U_i)=\{ v_i,v_{i+1} \}$, $V(U_i) \cap V(U_{i+1})=\{ v_{i+1} \}$, and delete the sequences that generates isomorphic $H=U_1 \cup \ldots \cup U_{\ell}$. Denote this list as $\mathcal J$.
    \STATE For each $[J] \in \mathcal J$, compute $\operatorname{Aut}(H)$ via Algorithm~\ref{alg:cal-Aut-H-in-mathcal-H}. 
    \STATE Generate i.i.d.\ random colorings $\{ \xi_{\mathtt k} : 1 \leq \mathtt k \leq t \}$ that map $[n]$ to $[mp\ell+1]$ uniformly at random. 
    \FOR{each $1 \leq \mathtt k \leq t$}
    \STATE For each $[J] \in \mathcal J$, compute $h_{J}(\bm Y,\xi_{\mathtt k})$ via Algorithm~\ref{alg:dynamic-programming}. 
    \ENDFOR
    \STATE Compute $\widetilde{\Phi}^{\mathcal J}_{i,j}(\bm Y)$ according to \eqref{eq-def-widetilde-Phi-H}.
    \STATE {\bf Output:} $\{ \widetilde{\Phi}^{\mathcal J}_{i,j}(\bm Y): 1 \leq i,j \leq n \}$.
    \end{algorithmic}
\end{breakablealgorithm}

\begin{proposition}{\label{prop-running-time-recovery}}
    Under \eqref{eq-condition-weak-recovery}, Algorithm~\ref{alg:cal-widetilde-Phi} computes $\{ \widetilde{\Phi}^{\mathcal J}_{i,j}(\bm Y): 1 \leq i,j \leq n \}$ in time $n^{C'+o(1)}$ for some $C'=C'(m,p)$. 
\end{proposition}
\begin{proof}
    Listing all non-isomorphic, unrooted unlabeled hypergraphs in $\mathcal U$ takes time $\binom{ \binom{mp+1}{p} }{ 2m }=n^{o(1)}$. Listing all non-isomorphic, unrooted unlabeled hypergraphs satisfying Definition~\ref{def-mathcal-H} in Step~3 takes time
    \begin{align*}
        O\Big( |\mathcal U|^{\ell} \Big) = O\Big( \tbinom{ (mp+1)^{p} }{ 2m }^{\ell} \Big) \overset{\eqref{eq-condition-weak-recovery}}{\leq} n^{ (mp+1)^{2mp} } \,.
    \end{align*}
    From Corollary~\ref{cor-bound-Aut-H-in-mathcal-H}, calculating $\mathsf{Aut}(H)$ for $[H]\in \mathcal H$ takes time
    \begin{align*}
        |\mathcal H| \cdot O\Big( 2\ell \cdot ((mp+1)!)^{\ell} \Big) = O\Big( \tbinom{ (mp+1)^{p} }{ 2m }^{\ell} \Big) \cdot O\Big( 2\ell \cdot ((mp+1)!)^{\ell} \Big) \overset{\eqref{eq-condition-weak-recovery}}{\leq} n^{ (mp+1)! + (mp+1)^{2pm} } \,.
    \end{align*}
    For all $1 \leq \mathtt k \leq t$, calculating $h_{J}(\bm Y,\xi_{\mathtt k})$ takes time $n^{2mp+3+o(1)}$. Thus, the total running time of Algorithm~\ref{alg:cal-widetilde-f} is 
    \begin{equation*}
        O\big( n^{(mp+1)^{2pm}} \big) + O\big( n^{(mp+1)!+(mp+1)^{2pm}} \big) + 2t \cdot O(n^{2mp+3+o(1)}) = O\big( n^{(mp+1)!+(mp+1)^{2pm}} \big) \,. \qedhere
    \end{equation*}
\end{proof}

\appendix

\section{Preliminary facts about graphs}{\label{sec:prelim-graphs}}

\begin{lemma}{\label{lem-prelim-graphs}}
    Let $S,K \subset \mathsf{K}_n$ be hypergraphs and $[H],[I]$ be unlabeled hypergraphs.
    \begin{enumerate}
        \item[(1)] We have $|E(S \cap K)|+|E(S \cup K)|=|E(S)|+|E(K)|$ and $|V(S \cap K)|+|V(S \cup K)|=|V(S)|+|V(K)|$.
        \item[(2)] If $|V(H)|=n^{o(1)}$, then $\#\{ S \subset \mathsf K_n: S \cong H \}=\frac{ (1+o(1)) n^{|V(H)|} }{ |\mathsf{Aut}(H)| }$.  
        \item[(3)] If $S \subset \mathsf K_n$, then $\#\{ K \subset \mathsf K_n: K \cong [I], |V(S) \cap V(K)|=h \} \leq \binom{|V(S)|}{h} \frac{ n^{|V(I)|-h} |V(I)|^{h} }{ |\mathsf{Aut}(I)| } $.
    \end{enumerate}
\end{lemma}
\begin{proof}
    Note that Item~(1) directly follows from the definition of $S \cap K$ and $S \cup K$ in Section~\ref{subsec:notation} and the inclusion-exclusion principle. As for Item~(2), we have
    \begin{align*}
        \#\{ S \subset \mathsf K_n: S \cong H \} &= \binom{n}{|V(H)|} \cdot \frac{|V(H)|!}{|\mathsf{Aut}(H)|} = \frac{ n^{|V(H)|} }{ |\mathsf{Aut}(H)| } \cdot \prod_{1 \leq i \leq |V(H)|} \frac{ n-i+1 }{ n } \\
        & \overset{|V(H)|=n^{o(1)}}{=} [1+o(1)] \frac{ n^{|V(H)|} }{ |\mathsf{Aut}(H)| } \,.
    \end{align*}

    As for Item~(3), note that the choices of $V(S) \cap V(K) \subset V(S)$ is bounded by $\binom{|V(S)|}{h}$ and the choices of $V(K) \setminus V(S)$ is bounded by $\binom{n}{|V(I)|-h}$. In addition, once $V(K)$ is fixed the choices of $K$ is bounded by $\frac{|V(I)|!}{|\mathsf{Aut}(I)|}$. Thus, we have
    \begin{align*}
        \#\{ K \subset \mathsf K_n: K \cong [I], |V(S) \cap V(K)|=h \} &\leq \frac{|V(I)|!}{|\mathsf{Aut}(I)|} \binom{|V(S)|}{h} \binom{n}{|V(I)|-h} \\
        &\leq \binom{|V(S)|}{h} \frac{ n^{|V(I)|-h} |V(I)|^{h} }{ |\mathsf{Aut}(I)| } \,. \qedhere
    \end{align*} 
\end{proof}

We also need the following result on the asymptotic number of hypergraphs with a given degree sequences. The proof of the following result can be found in \cite[Theorem~1.1]{BG16}.
\begin{lemma}{\label{lem-labeled-hypergraph-given-degree-sequence}}
    Suppose $N,M,p \geq 1$ and let $d=(d_1,\ldots,d_N)$ such that $\sum_{1 \leq i \leq N} d_i=pM$. Denote $d_{\operatorname{max}}=\max_{1 \leq i \leq N} d_i$. Suppose that $N,M \to \infty$ and that $d_{\operatorname{max}}^3 = o(M)$. Then
    \begin{align*}
        &\#\big\{ p\mbox{-uniform hypergraph } U: V(U)=[N], \mathsf{deg}_U(i)=d_i \mbox{ for } 1 \leq i \leq N \big\} \\
        =\ & \frac{ (pM)! }{ M! (p!)^{M} \prod_{1 \leq i \leq N} d_i! } \cdot \exp\Big( -\frac{ \sum_{1 \leq i \leq N} d_i(d_i-1) }{ 2pM } + O\big( \frac{ d_{\operatorname{max}}^3 }{ pM } \big) \Big) \,.
    \end{align*}
    In particular, if $d_{\operatorname{max}}=2$ then there exists a universal constant $R>0$ such that
    \begin{align*}
        &\#\big\{ p\mbox{-uniform hypergraph } U: V(U)=[N], \mathsf{deg}_U(i)=d_i \mbox{ for } 1 \leq i \leq N \big\} \\
        & \in \Big( \frac{1}{R} \cdot \frac{ (pM)! }{ M! (p!)^{M} \prod_{1 \leq i \leq N} d_i! } , R \cdot \frac{ (pM)! }{ M! (p!)^{M} \prod_{1 \leq i \leq N} d_i! } \Big) \,.
    \end{align*}
\end{lemma}

Recall Definition~\ref{def-component-hypergraph}. We now provide some preliminary results for $[U] \in \mathcal U$.

\begin{lemma}{\label{lem-prelim-result-mathcal-U}}
    Suppose $[U] \in \mathcal U$ and $S,K \subset \mathsf K_n$.
    \begin{enumerate}
        \item[(1)] We have $|\mathsf{Aut}(U)| \leq (2m)(p!)^{2m} p^{2m}$.
        \item[(2)] For all subgraph $W \subset U$ with $V(W) \neq \emptyset$ and $E(W) \neq E(U)$ we have $|\mathsf L(W) \setminus \mathsf L(U)| \geq 2$.
        \item[(3)] Recall \eqref{eq-def-beta-diamond-U}. For each $i,j \in [n]$ we have 
        \begin{align*}
            & \#\big\{ U \subset \mathsf K_n: i \in \mathsf L(U), [U] \in \mathcal U \big\} = [1+o(1)] n^{mp} \beta^{\diamond}_{\mathcal U} \,; \\
            & \#\big\{ U \subset \mathsf K_n: \mathsf L(U)=\{ i,j \}, [U] \in \mathcal U \big\} = [1+o(1)] n^{mp-1} \beta^{\diamond}_{\mathcal U} \,.
        \end{align*}  
        \item[(4)] For all $i,j \in [n]$ and all $V \subset [n] \setminus \{ i,j \}$ we have
        \begin{align*}
            & \#\big\{ U \subset \mathsf K_n: [U] \in \mathcal U, i \in \mathsf L(U), V \subset V(U) \big\} \leq (mp+1)^{|V|} n^{mp-|V|} \beta^{\diamond}_{\mathcal U} \,; \\
            & \#\big\{ U \subset \mathsf K_n: [U] \in \mathcal U, \mathsf L(U)=\{ i,j \}, V \subset V(U) \big\} \leq (mp+1)^{|V|} n^{mp-1-|V|} \beta^{\diamond}_{\mathcal U} \,.
        \end{align*}
    \end{enumerate}
\end{lemma}
\begin{proof}
    As for Item~(1), for all $\pi\in\mathsf{Aut}(U)$, consider the induced bijection 
    \begin{align*}
        \pi_E: E(U) \overset{\pi}{\longrightarrow} E(U) \,.
    \end{align*}
    We first argue that the possible choices of $\pi_E$ is bounded by $(2m)\cdot p^{m}$. Since $U$ is connected, we can list $E(U)$ in the following order $e_1,\ldots,e_{2m}$ such that
    \begin{align*}
        e_{i+1} \cap e_{j} \neq \emptyset \mbox{ for some } j \leq i \,.
    \end{align*}
    The choices of $\pi(e_1)$ is bounded by $2m$. Having decided $\pi(e_1),\ldots,\pi(e_{i})$, since $e_{i+1} \cap e_j \neq \emptyset$ for some $j \leq i$ we see that we must have $\pi(e_{i+1}) \cap \pi(e_j) \neq \emptyset$; however, since $\mathsf{deg}_U(v) \in \{ 1,2 \}$ for all $v \in V(H)$ we see that $\#\{ e \in E(H): e \cap e_j \neq \emptyset \} \leq p$. Thus, the choices of $e_{i+1}$ is bounded by $p$. Thus the total choices of $\pi_E$ is bounded by $(2m) p^{2m}$. In addition, having fixed $\pi_E$, for each $1 \leq i \leq m$ the choices of $\pi$ restricted on $e_i$ is bounded by $p!$. Thus, the possible choices of $\pi$ is bounded by $(2m)p^{2m}(p!)^{2m}$.

    As for Item~(2), without losing of generality, we may assume that $\mathsf I(W)=\emptyset$. If $\mathsf L(W) \subset \mathsf L(U)$, then  $\mathsf{deg}_W(v)=\mathsf{deg}_U(v)$ for all $v \in V(W)$. Thus, we have either $V(W)=V(U)$ (which implies that $W=U$) or $V(W)$, $V(U) \setminus V(W)$ are disconnected in $U$, which forms a contradiction. In addition, if $\mathsf L(W) \setminus \mathsf L(U)=\{ v \}$ for some $v \in V(U)$, then we have $\mathsf{deg}_W(v)=\mathsf{deg}_U(v)$ for all $v \in V(W) \setminus \{ v \}$. Also, since $\mathsf{deg}_U(v)=2$ there exists $e \not \in E(W)$ and $v \in e$. This implies that 
    \begin{align*}
        \big( e \setminus \{ v \} \big) \cap V(W) = \emptyset
    \end{align*}
    and thus $V(W) \setminus \{ v \}$ and $e \setminus \{ v \}$ are disconnected in $V|_{\setminus \{ v \}}$. This also forms a contradiction and thus we must have $|\mathsf L(W) \setminus \mathsf L(U)| \geq 2$.

    As for Item~(3), for $u,v \in V(U)$ and $u',v' \in V(U')$ write $(v,U) \cong (v',U')$ if there exists an isomorphism $\pi:V(U) \to V(U')$ such that $\pi(U)=U'$ and $\pi(v)=v'$, and write $(u,v,U) \cong (u',v',U')$ if such isomorphism $\pi$ maps $u$ to $u'$ and maps $v$ to $v'$. Now we show the first equality. Note that 
    \begin{align*}
        \#\big\{ U \subset \mathsf K_n: i \in \mathsf L(U), [U] \in \mathcal U \big\} = \sum_{ [U] \in \mathcal U } \sum_{ [v] \in \mathsf L_{\diamond}(U) } \#\big\{ W \subset \mathsf K_n: i \in V(W), (i,W) \cong (v,U) \big\} \,.
    \end{align*}
    Thus, according to \eqref{eq-def-beta-diamond-U} it suffices to show that
    \begin{align*}
        \#\big\{ W \subset \mathsf K_n: i \in V(W), (i,W) \cong (v,U) \big\} = \frac{ [1+o(1)] n^{mp} }{ |\mathsf{Aut}_v(U)| } \,.
    \end{align*}
    Note that $|V(W)|=mp+1$ for $W \cong U$, there are $\binom{n-1}{mp}$ ways to choose $V(W)$ and given $i \in V(W)$, there are $\frac{ (mp)! }{ |\mathsf{Aut}_v(U)| }$ ways to choose $(i,W) \cong (v,U)$. Thus, we have
    \begin{align*}
        \#\big\{ W \subset \mathsf K_n: i \in V(W), (i,W) \cong (v,U) \big\} = \binom{n-1}{mp} \cdot \frac{ (mp)! }{ |\mathsf{Aut}_v(U)| } = \frac{ [1+o(1)] n^{mp} }{ |\mathsf{Aut}_v(U)| } \,.
    \end{align*}
    The second equality can be derived in the similar manner, using the fact that
    \begin{align*}
        \#\big\{ U \subset \mathsf K_n: i \in \mathsf L(U), [U] \in \mathcal U \big\} = \sum_{ [U] \in \mathcal U } \sum_{ [v] \in \mathsf L_{\diamond}(U) } \#\big\{ W \subset \mathsf K_n: i,j \in V(W), (i,j;W) \cong (v,v',U) \big\} 
    \end{align*}
    where $v'$ is the element in $\mathsf{L}(U) \setminus \{ v \}$ (note that the choice of the representative element $v$ of $[v]$ does not affect the quantity $\#\{ W \subset \mathsf K_n: i,j \in V(W), (i,j;W) \cong (v,v',U) \}$, so the right hand side is well-defined). Thus, it suffices to show that 
    \begin{align*}
        \#\big\{ W \subset \mathsf K_n: i,j \in V(W), (i,j;W) \cong (v,v',U) \big\} = \frac{ [1+o(1)] n^{mp-1} }{ |\mathsf{Aut}_v(U)| } \,.
    \end{align*}
    Note that $|V(W)|=mp+1$ for $W \cong U$, there are $\binom{n-1}{mp-1}$ ways to choose $V(W)$ and given $i \in V(W)$, there are $\frac{ (mp-1)! }{ |\mathsf{Aut}_v(U)| }$ ways to choose $(i,j,W) \cong (v,v',U)$. Thus, we have
    \begin{align*}
        \#\big\{ W \subset \mathsf K_n: i,j \in V(W), (i,j;W) \cong (v,v',U) \big\} = \frac{ (mp-1)! }{ |\mathsf{Aut}_v(U)| } \cdot \binom{n-1}{mp-1} = \frac{ [1+o(1)] n^{mp-1} }{ |\mathsf{Aut}_v(U)| } \,.
    \end{align*}

    As for Item~(4), note that
    \begin{align*}
        & \#\big\{ [U] \in \mathcal U, U \subset \mathsf K_n: i \in \mathsf L(U), V \subset V(U) \big\} \\
        =\ & \sum_{ [U] \in \mathcal U } \sum_{ [v] \in \mathsf L_{\diamond}(U) } \#\big\{ W \subset \mathsf K_n: i \in V(W), V \subset V(W), (i,W) \cong (v,U) \big\} \,.
    \end{align*}
    Thus, according to \eqref{eq-def-beta-diamond-U} it suffices to show that
    \begin{align*}
        \#\big\{ W \subset \mathsf K_n: i \in V(W), V \subset V(W), (i,W) \cong (v,U) \big\} \leq \frac{ (mp+1)^{|V|} n^{mp-|V|} }{ |\mathsf{Aut}_v(U)| } \,.
    \end{align*}
    Note that $V(W) \setminus \{ i \}$ has $\binom{n-|V|-1}{mp-|V|}$ choices. In addition, once we decided $V(W)$ there are at most $\frac{ (mp)! }{ |\mathsf{Aut}_v(U)| }$ number of $W$ such that $(i,W) \cong (v,U)$. Thus, we have
    \begin{align*}
        \#\big\{ W \subset \mathsf K_n: i \in V(W), V \subset V(W), (i,W) \cong (v,U) \big\} \leq \frac{ \binom{n-|V|-1}{mp-|V|} (mp)! }{ |\mathsf{Aut}_v(U)| } \leq \frac{ n^{mp-|V|} (mp)^{|V|} }{ |\mathsf{Aut}_v(U)| } \,,
    \end{align*}
    leading to the first equality. The second equality can be derived in the similar manner.
\end{proof}

Recall Definition~\ref{def-mathcal-H}. We now provide some preliminary results for $[H] \in \mathcal H$.
\begin{lemma}{\label{lem-prelim-mathcal-H}}
    Suppose $[H] \in \mathcal H$.
    \begin{enumerate}
        \item[(1)] For all $H' \subset H$ with $|E(H')| \neq \emptyset$ and $E(H') \neq E(H)$ we have $2|V(H')|> p|E(H)|$.
        \item[(2)] Denote $H=U_1\cap\ldots\cap U_{\ell}$ with $U_i \in \mathcal U$ satisfying Definition~\ref{def-mathcal-H}. Then for any $\pi \in \mathsf{Aut}(H)$ there exists $1 \leq k \leq \ell$ such that either $\pi(U_i)=U_{i+k}$ for all $1 \leq i \leq \ell$ or $\pi(U_i)=U_{i-k+\ell}$ for all $1 \leq i \leq \ell$ (we again denote $U_{i+\ell}=U_{i}$). 
    \end{enumerate}
\end{lemma}
\begin{proof}
    As for Item~(1), note that $\mathsf{deg}_H(v)=2$ for all $v \in V(H)$. Thus, since $H' \subset H$ is $p$-uniform
    \begin{align*}
        p|E(H')| = \sum_{ v \in V(H) } \mathsf{deg}_{H'}(v) \leq \sum_{ v \in V(H) } \mathsf{deg}_{H}(v) \leq 2|V(H)| \,.
    \end{align*}
    In addition, if equality holds we must have $\mathsf{deg}_{H'}(v)= \mathsf{deg}_{H}(v)$ for all $v \in V(H')$, which means that for all $v \in V(H')$ the neighbors of $v$ in $H$ also belongs to $H'$. Combines with the fact that $H$ is connected, this implies that either $|E(H')|=\emptyset$ or $H'=H$. Thus Item~(1) holds.

    As for Item~(2), recall that we have $V(U_{i}) \cap V(U_{i+1})=\{ v_{i+1} \}$ for some $v_i \in \mathsf L(U_i), \mathsf L(U_{i+1})$. We first argue that we must have $\pi(v_i) \in \{ v_1,\ldots,v_{\ell} \}$. Assume on the contrary that $\pi(v_i) \in V(U_k) \setminus \{ u_k,u_{k+1} \}$. Note that in the hypergraph $H_{ \setminus \{ v_i,v_j \} }$, there are exactly two connected components, one has $mp|j-i|-1$ vertices and the other has $mp(\ell-|j-i|)-1$ vertices. Then
    \begin{itemize}
        \item If $\pi(v_j) \in V(U_k)$, then from Item~(3) in Definition~\ref{def-component-hypergraph} we see that $V(U_k) \setminus \{ \pi(v_i) \}$ is connected. Thus, we see that either $U_k|_{ \setminus \{ \pi(v_i), \pi(v_j) \} }$ is connected or it contains two components $C_k,C_k'$ with $|V(C_k)|,|V(C_k')|<mp-1$. Thus, we have either $H|_{ \setminus \{ \pi(v_i),\pi(v_j) \} }$ is connected (if $U_k|_{ \setminus \{ \pi(v_i), \pi(v_j) \} }$ is connected or if $C_k \cap \{ v_k,v_{k+1} \}, C_{k}' \cap \{ v_k,v_{k+1} \} \neq \emptyset$) or $H|_{ \setminus \{ \pi(v_i),\pi(v_j) \} }$ contains a component with strictly less than $mp-1$ vertices (if $C_k \cap \{ v_k,v_{k+1} \}=\emptyset$ or $C_k' \cap \{ v_k,v_{k+1} \}=\emptyset$). This forms a contradiction.
        \item If $\pi(v_j) \in V(U_{l}) \setminus V(U_k)$ for some $l \neq k$, then from Item~(3) in Definition~\ref{def-component-hypergraph} we see that $U_k|_{\setminus \{ \pi(v_i) \}}$ and $U_{l}|_{\setminus \{ \pi(v_j) \}}$ is connected, and $\{ v_k,v_{k+1} \} \in V(U_k) \setminus \{ v_i,v_j \}$. Thus, we have $H|_{\setminus \{ v_i,v_j \}}$ is connected, which also leads to contradiction. 
    \end{itemize}
    In conclusion, we have $\pi$ maps $\{ v_1,\ldots,v_{\ell} \}$ to itself. Then clearly there exists $k$ that either $\pi(v_i)=v_{i+k}$ for all $1 \leq i \leq \ell$ or $\pi(v_i)=v_{\ell+i-k}$ for all $1 \leq i \leq \ell$. This shows that there exists $1 \leq k \leq \ell$ such that either $\pi(U_i)=U_{i+k}$ for all $1 \leq i \leq \ell$ or $\pi(U_i)=U_{i-k+\ell}$ for all $1 \leq i \leq \ell$.
\end{proof}

\begin{cor}{\label{cor-bound-Aut-H-in-mathcal-H}}
    For all $H \in \mathcal H$, we have 
    \begin{align*}
        |\mathsf{Aut}(H)| \leq (2\ell) \cdot \Big( (2m)(p!)^{2m} p^{2m} \Big)^{\ell} \,.
    \end{align*}
    In addition, there exists an algorithm with running time $O\big( 2\ell \cdot ((mp+1)!)^{\ell} \big)$ to compute $\mathsf{Aut}(H)$.
\end{cor}
\begin{proof}
    Note that by Item~(2) in Lemma~\ref{lem-prelim-graphs} for all $1 \leq i,j \leq \ell$, the number of isomorphisms between $U_i$ and $U_{j}$ is bounded by $(2m)(p!)^{2m} p^{2m}$. Combined with Item~(2) in Lemma~\ref{lem-prelim-mathcal-H}, we get that 
    \begin{align*}
        |\mathsf{Aut}(H)| \leq (2\ell) \cdot \Big( (2m)(p!)^{2m} p^{2m} \Big)^{\ell} \,.
    \end{align*} 
    In addition, consider the following algorithm:
    \begin{breakablealgorithm}{\label{alg:cal-Aut-H-in-mathcal-H}}
    \caption{Computation of $\mathsf{Aut}(H)$ for $[H] \in \mathcal H$}
        \begin{algorithmic}[1]
        \STATE {\bf Input:} Unlabeled hypergraph $[H] \in\mathcal H$ with decomposition $H = U_1 \cup \ldots \cup U_{\ell}$ satisfying Definition~\ref{def-mathcal-H}.
        \FOR{each $1 \leq k \leq \ell$}
        \STATE Compute $\mathsf{Aut}_+(k,H)_{i} := \big\{ \pi_i: V(U_i) \to V(U_{i+k}): \pi_i(U_i)=U_{i+k}, \pi_i(v_i)=v_{i+k} \big\}$.
        \STATE Compute $\mathsf{Aut}_-(k,H)_{i} := \big\{ \pi_i: V(U_i) \to V(U_{\ell+i-k}): \pi_i(U_i)=U_{\ell+i-k}, \pi_i(v_i)=v_{\ell+i-k} \big\} $.
        \STATE Compute $\mathsf{Aut}_+(k,H) := \big\{ \pi: \forall 1 \leq i \leq \ell, \exists \pi_i \in \mathsf{Aut}_+(k,H)_{i},  \pi=\pi_i \mbox{ restricted on } V(U_i) \big\}$.
        \STATE Compute $\mathsf{Aut}_-(k,H) := \big\{ \pi: \forall 1 \leq i \leq \ell, \exists \pi_i \in \mathsf{Aut}_-(k,H)_{i},  \pi=\pi_i \mbox{ restricted on } V(U_i) \big\}$.
        \ENDFOR
        \STATE Compute $\mathsf{Aut}(H)=\cup_{1 \leq k \leq \ell} (\mathsf{Aut}_+(k,H) \cup \mathsf{Aut}_-(k,H))$. 
        \STATE {\bf Output:} $\mathsf{Aut}(H)$.
        \end{algorithmic}
    \end{breakablealgorithm}
    From Item~(2) in Lemma~\ref{lem-prelim-mathcal-H}, we see that Algorithm~\ref{alg:cal-Aut-H-in-mathcal-H} correctly outputs all the automorphisms in $\mathsf{Aut}(H)$. In addition, it is clear that Algorithm~\ref{alg:cal-Aut-H-in-mathcal-H} runs in time $O\big( 2\ell \cdot ((mp+1)!)^{\ell} \big)$, since the computation of $\mathsf{Aut}_+(k,H)_{i}, \mathsf{Aut}_-(k,H)_{i}$ takes time $O((mp+1)!)$ and the computation of $\mathsf{Aut}_+(k,H), \mathsf{Aut}_-(k,H)$ takes time $O(((mp+1)!)^{\ell})$.
\end{proof}

Recall Definition~\ref{def-mathcal-J}. Now we provide some preliminary results for $[J] \subset \mathcal J$.
\begin{lemma}{\label{lem-prelim-mathcal-J}}
    Suppose $[J] \subset \mathcal J$.
    \begin{enumerate}
        \item[(1)] We have $\#\{ S \subset \mathsf K_n: S \cong J, \mathsf L(S)=\{ i,j \} \} = \frac{ (1+o(1)) n^{|V(J)|-2} }{ |\mathsf{Aut}(J)| }$.  
        \item[(2)] Suppose $J=U_1 \cup \ldots \cup U_{\ell}$ such that $U_1,\ldots,U_{\ell}$ satisfies Items~(1)--(4) in Definition~\ref{def-mathcal-J}. Also suppose that there exists some $U \subset J$ with $[U] \in \mathcal U$. Then $U=U_i$ for some $1 \leq i \leq \ell$.
        \item[(3)] Suppose $J=U_1 \cup \ldots \cup U_{\ell}$ such that $(U_1,\ldots,U_{\ell})$ and $(v_1,\ldots,v_{\ell+1})$ satisfies Items~(1)--(4) in Definition~\ref{def-mathcal-J}. Then $\pi(U_i)=U_i$ and $\pi(v_{j})=v_j$ for all $1 \leq i \leq \ell$ and $1 \leq j \leq \ell+1$.
        \item[(4)] Suppose $J=U_1 \cup \ldots \cup U_{\ell}$ such that $(U_1,\ldots,U_{\ell})$ and $(v_1,\ldots,v_{\ell+1})$ satisfies Items~(1)--(4) in Definition~\ref{def-mathcal-J}. Then $|\mathsf{Aut}(J)|=\prod_{1 \leq i \leq \ell}|\mathsf{Aut}_{v_{i}}(U_i)|$.
    \end{enumerate}
\end{lemma}
\begin{proof}
    As for Item~(1), list $V(J)$ in an arbitrary but fixed order $\{ v_1, \ldots, v_{|V(J)|} \}$ such that $\mathsf L(S)=\{ v_{|V(J)|-1}, v_{|V(J)|} \}$. Denote $\mathfrak W$ be the set of sequences $\mathfrak w=(u_1,\ldots,u_{|V(J)|-2},i,j)$ such that $u_k \in [n] \setminus \{ i,j \}$ and $u_{k} \neq u_{k'}$. Let $\psi(\mathfrak w) \subset \mathsf K_n$ be the hypergraph such that 
    \begin{align*}
        \pi_{\mathfrak w}(v_{|V(H)|-1})=i, \pi_{\mathfrak w}(v_{|V(H)|})=j, \pi_{\mathfrak w}( v_{k} ) = u_k \mbox{ for } k \leq |V(H)|-2
    \end{align*}
    is an isomorphism between $H$ and $\psi(\mathfrak w)$. Note that if $\psi(\mathfrak w')=\psi(\mathfrak w)$, then $\pi_{\mathfrak w} \circ \pi_{\mathfrak w'}^{-1} \in \mathsf{Aut}(J)$. Thus, from Item~(4) in Definition~\ref{def-mathcal-J} we see that
    \begin{align*}
        \#\big\{ \mathfrak w': \psi(\mathfrak w')=\psi(\mathfrak w) \big\} = |\mathsf{Aut}(J)| \,.
    \end{align*}
    Thus, we have
    \begin{align*}
        & \#\Big\{ S \subset \mathsf K_n: S \cong J, \mathsf L(S)=\{ i,j \} \Big\} = \frac{ |\mathfrak W| }{ |\mathsf{Aut}(J)| } \\
        =\ & \binom{n}{|V(J)|-2} \cdot \frac{(|V(J)|-2)!}{|\mathsf{Aut}(J)|} = [1+o(1)] \frac{n^{|V(J)|-2}}{|\mathsf{Aut}(J)|} \,,
    \end{align*}
    where the last equation follows from $|V(J)|=n^{o(1)}$.

    As for Item~(2), suppose on the contrary that $E(U) \cap E(U_i), E(U) \cap E(U_j) \neq \emptyset$ for two different $i,j$. Then since $|E(U)|=|E(U_i)|=|E(U_j)|=2m$ we see that $U \cap U_i$ and $U \cap U_{j}$ are two connected non-empty subgraph of $U_i$ and $U_j$ respectively, with $U \cap U_i \neq U_i$ and $U \cap U_j \neq U_j$. Thus, using Item~(2) in Lemma~\ref{lem-prelim-result-mathcal-U} we see that $|\mathsf L(U \cap U_i) \setminus \mathsf L(U_i)|, |\mathsf L(U \cap U_j) \setminus \mathsf L(U_j)| \geq 2$. As vertices in $V(U_i) \setminus \mathsf L(U_i)$ do not have neighbors in $V(J) \setminus V(U_i)$, we see that
    \begin{align*}
        \mathsf L(U \cap U_i) \setminus \mathsf L(U_i), \mathsf L(U \cap U_j) \setminus \mathsf L(U_j) \subset \mathsf L(U) \,.
    \end{align*}
    This yields $|\mathsf L(U)| \geq 4$ and thus forms a contradiction.

    As for Item~(3), note that using Item~(4) in Definition~\ref{def-mathcal-J}, we see that for all $\pi(v_1)=v_1$. In addition, since $[\pi(U_1)] \in \mathcal U$ and $\pi(U_1) \subset J$, we see from Item~(2) that we must have $\pi(U_1)=U_1$ and thus $\pi(v_2)=v_2$. Repeating this argument we see that for all $\pi\in\mathsf{Aut}(J)$ we have $\pi(U_i)=U_i$ for $1 \leq i \leq \ell$ and $\pi(v_j)=v_j$ for all $1 \leq j \leq \ell+1$. Based on this, for each $1 \leq i \leq \ell$ we have $\pi_i \in \mathsf{Aut}_{v_i}(U_i)$ and 
    \begin{align*}
        \pi \longrightarrow (\pi_1,\ldots,\pi_{\ell})
    \end{align*}
    is a bijection between $\mathsf{Aut}(J)$ and $\mathsf{Aut}_{v_1}(U_1) \otimes \ldots \otimes \mathsf{Aut}_{v_\ell}(U_\ell)$. This leads to Item~(4).
\end{proof}

\begin{cor}{\label{cor-efficient-compute-Aut-J}}
    For any $J \in \mathcal J$, there exists an algorithm with running time $O( ((mp+1)!)^{\ell} )$ that computes $\mathsf{Aut}(J)$.
\end{cor}
\begin{proof}
    Consider the following algorithm:
    \begin{breakablealgorithm}{\label{alg:cal-Aut-J-in-mathcal-J}}
    \caption{Computation of $\mathsf{Aut}(J)$ for $[J] \in \mathcal J$}
        \begin{algorithmic}[1]
        \STATE {\bf Input:} Unlabeled hypergraph $[J] \in\mathcal J$ with decomposition $J=U_1 \cup \ldots \cup U_{\ell}$ satisfying Definition~\ref{def-mathcal-J}.
        \FOR{each $1 \leq k \leq \ell$}
        \STATE Compute $\mathsf{Aut}_{v_{k}}(U_k)$.
        \ENDFOR
        \STATE Compute $\mathsf{Aut}(J) = \{ \pi: \pi(v)=\pi_k(v) \mbox{ for some } \pi_k \in \mathsf{Aut}_{v_{k}}(U_k), \forall v \in V(U_k), 1 \leq k \leq \ell \}$. 
        \STATE {\bf Output:} $\mathsf{Aut}(J)$.
        \end{algorithmic}
    \end{breakablealgorithm}
    Using Item~(3) in Lemma~\ref{lem-prelim-mathcal-J}, it is clear that Algorithm~\ref{alg:cal-Aut-J-in-mathcal-J} correctly computes $\mathsf{Aut}(J)$ and runs in time $O( ((mp+1)!)^{\ell} )$.
\end{proof}

\section{Postponed proof in Section~\ref{sec:main-results}}{\label{sec:supp-proofs-sec-2}}

In this section we use the results we established in Section~\ref{sec:prelim-graphs} to provide the postponed proofs of Lemmas~\ref{lem-control-Aut-component-hypergraph}, \ref{lem-control-Aut-necklace-hypergraph} and \ref{lem-control-Aut-necklace-chain}.

\subsection{Proof of Lemma~\ref{lem-control-Aut-component-hypergraph}}{\label{subsec:proof-lem-2.2}}

Note that $\#\{ H \subset \mathsf K_{mp+1}: H \cong U \}=\tfrac{(mp+1)!}{|\mathsf{Aut}(U)|}$. Thus, we have
\begin{align*}
    \sum_{[U] \in \mathcal U} \frac{ 1 }{ |\mathsf{Aut}(U)| } &= \frac{1}{(mp+1)!} \sum_{[U] \in \mathcal U} \#\big\{ H \subset \mathsf K_{mp+1}: H \cong U \big\} \\
    &= \frac{1}{(mp+1)!} \#\big\{ H \subset \mathsf K_{mp+1}: H \cong U \mbox{ for some } [U] \in \mathcal U \big\} \,.
\end{align*}
Thus, it suffices to show that 
\begin{equation}{\label{eq-lem-2.2-relax}}
    \#\mathfrak U_{m,p} := \#\big\{ H \subset \mathsf K_{mp+1}: [H] \in \mathcal U \big\} \geq R^{-1} \Big( 1-\frac{10^{10}R^3}{m} \Big) \frac{(mp+1)(2pm)!}{2^{mp}(p!)^{2m}(2m)!} 
\end{equation}
for some absolute constant $R>0$. To this end, for all constant $k \geq 0$ define $\mathfrak A_{(2N-k)/p,N}$ to be the set of $p$-uniform hypergraphs $U$ such that the following conditions hold:
\begin{enumerate}
    \item[(1)] $V(U)=[N]$ and $|E(U)|=\frac{2N-k}{p}$.
    \item[(2)] $\mathsf{deg}_U(v) \in \{ 1,2 \}$ for all $v \in [N]$.
\end{enumerate}
Clearly, for all $U \in \mathfrak A_{(pN-k)/2,N}$ we have $|\mathsf L(U)|=k$. In addition, given $v \in [N]$ define
\begin{align*}
    \mathfrak A_{ (pN-k)/2,N,v } := \Big\{ U \in \mathfrak A_{(pN-k)/2,N}: v \in \mathsf L(U) \Big\} \,.
\end{align*}
\begin{lemma}{\label{lem-enu-desired-hypergraphs}}
    there exists a universal constant $R>0$ such that
    \begin{align}
        & \frac{ R^{-1} \binom{N}{k} \cdot (2N-k)! }{ (\frac{2N-k}{p})! (p!)^{(2N-k)/p} 2^{N-k} } \leq \#\mathfrak{A}_{(2N-k)/p,N} \leq \frac{ R \binom{N}{k} \cdot (2N-k)! }{ (\frac{2N-k}{p})! (p!)^{(2N-k)/p} 2^{N-k} } \,; \label{eq-est-mathfrak-A-general} \\
        & \frac{ R^{-1} \binom{N-1}{k-1} \cdot (2N-k)! }{ (\frac{2N-k}{p})! (p!)^{(2N-k)/p} 2^{N-k} } \leq \#\mathfrak{A}_{(2N-k)/p,N} \leq \frac{ R \binom{N-1}{k-1} \cdot (2N-k)! }{ (\frac{2N-k}{p})! (p!)^{(2N-k)/p} 2^{N-k} } \,.  \label{eq-est-mathfrak-A-particular}
    \end{align}
\end{lemma}
\begin{proof}
    For $U \in \mathfrak{A}_{(pN-k)/2,N}$, note that there are at most $\binom{N}{k}$ ways to choose $\mathsf L(U)$. In addition, once we have fixed $\mathsf L(U)$ the degree sequence of $U$ is fixed from the definition of $\mathfrak{A}_{(pN-k)/2,N}$. Thus, using Lemma~\ref{lem-labeled-hypergraph-given-degree-sequence} with $M=\frac{2N-k}{p}$ we obtain the desired bound on $\#\mathfrak{A}_{(2N-k)/p,N}$. The bound on $\#\mathfrak{A}_{(2N-k)/p,N,v}$ can be derived in the same manner, except the choices of $\mathsf L(U)$ is bounded by $\binom{N-1}{k-1}$. 
\end{proof}

Note that for all $U \in \mathfrak A_{2m,mp+1}$ we have $U$ satisfies Item~(1) in Definition~\ref{def-component-hypergraph}. Thus, it suffices to upper bound the number of $U \in\mathfrak A_{2m,mp+1}$ such that $U$ does not satisfy Item~(2) or Item~(3) in Definition~\ref{def-component-hypergraph}. We first deal with the case that $U$ is not connected. In this case, $U$ can be decomposed into two components $U_1$ and $U_2$ such that $V(U_1) \cap V(U_2)=\emptyset$. Then $\mathsf L(U)=\mathsf L(U_1) \sqcup \mathsf L(U_2)$. Without losing of generality, we may assume that $|\mathsf L(U_1)| \leq 1$. Denote $|E(U_1)|=k$ for some $3 \leq k \leq 2m-3$, we then have $|V(U_1)| = \frac{pk+1}{2}$ if $pk$ is odd and $|V(U_1)|=\tfrac{pk}{2}$ if $pk$ is even. Thus, if $pk$ is odd the possible enumeration of $(U_1,U_2)$ in this case is bounded by
\begin{align*}
    & \binom{mp+1}{\frac{kp+1}{2}} \big| \mathfrak A_{k,(kp+1)/2} \big| \cdot \big| \mathfrak A_{(2m-k),((2m-k)p+1)/2} \big|  \\
    \overset{\eqref{eq-est-mathfrak-A-general}}{\leq}\ & R^2 \binom{mp+1}{ \frac{kp+1}{2} } \cdot \frac{ \frac{kp+1}{2} \cdot (kp)! }{ k! (p!)^{k} 2^{\frac{kp-1}{2}} } \cdot \frac{ \frac{((2m-k)p+1)}{2} \cdot ((2m-k)p)! }{ (2m-k)! (p!)^{2m-k} 2^{ \frac{(2m-k)p-1}{2} } } \\
    =\ & R^2 \binom{mp+1}{ \frac{kp+1}{2} } \cdot \frac{ (kp+1)!((2m-k)p+1)! }{ k!(2m-k)! (p!)^{2m} 2^{mp+1} } \\
    \overset{\eqref{eq-est-mathfrak-A-general}}{\leq}\ & R^3 \binom{mp+1}{ \frac{kp+1}{2} } \cdot \big| \mathfrak A_{2m,mp+1} \big| \cdot \frac{ (kp+1)((2m-k)p+1) }{ mp(mp+1) } \cdot \binom{2m}{k} \cdot \binom{2mp}{kp}^{-1} \\
    \leq\ & \big| \mathfrak A_{2m,mp+1} \big| \cdot \frac{ R^3 (kp+1)((2m-k)p+1) \binom{2m}{k} \binom{mp+1}{\frac{kp+1}{2}} }{ mp(mp+1)\binom{2mp}{kp} }  \,.
\end{align*}
Similarly, when $pk$ is even the possible enumeration of $(U_1,U_2)$ in this case is bounded by
\begin{align*}
    & \binom{mp+1}{ \frac{kp}{2} } \cdot \big| \mathfrak A_{k,kp/2} \big| \cdot \big| \mathfrak A_{2m-k,(2m-k)p/2+1} \big|  \\
    \overset{\eqref{eq-est-mathfrak-A-general}}{\leq}\ & R^2 \binom{mp+1}{ \frac{kp}{2} } \cdot \frac{ (kp)! }{ (k!)(p!)^{k} 2^{\frac{kp}{2}} } \cdot \binom{(2m-k)p/2+1}{2} \cdot \frac{ ((2m-k)p)! }{ (2m-k)! (p!)^{2m-k} 2^{\frac{(2m-k)p}{2}-1} } \\
    =\ & R^2 \binom{mp+1}{ \frac{kp}{2} } \cdot \frac{ ((2m-k)p+2)((2m-k)p) \cdot (kp)!((2m-k)p)! }{ k!(2m-k)! (p!)^{2m-k} 2^{mp} } \\
    \overset{\eqref{eq-est-mathfrak-A-general}}{\leq}\ & R^3 \binom{mp+1}{ \frac{kp}{2} } \cdot \big| \mathfrak A_{2m,mp+1} \big| \cdot \frac{ ((2m-k)p+2)((2m-k)p) }{ mp(mp+1) } \cdot \binom{2m}{k} \cdot \binom{2mp}{kp}^{-1}  \\
    \leq\ & \big| \mathfrak A_{2m,mp+1} \big| \cdot \frac{ R^3 ((2m-k)p+2)((2m-k)p) \binom{2m}{k} \binom{mp+1}{\frac{kp}{2}} }{ mp(mp+1) \binom{2mp}{kp} }  \,.
\end{align*}
Thus, we see that
\begin{align}
    &\#\Big\{ U \in \mathfrak A_{2m,mp+1}: U \mbox{ does not satisfy Item~(2) in Definition~\ref{def-component-hypergraph}} \Big\} \nonumber \\
    \leq\ & R^3 \big| \mathfrak A_{2m,mp+1} \big| \cdot \Bigg( \sum_{ \substack{ 3 \leq k \leq 2m-3 \\ kp \text{ is odd} } } \frac{ (kp+1)((2m-k)p+1) \binom{2m}{k}\binom{mp}{\frac{kp+1}{2}} }{ mp(mp+1) \binom{2mp}{kp} } \nonumber \\
    & \quad\quad\quad\quad\quad\quad\quad\quad +  \sum_{ \substack{ 3 \leq k \leq 2m-3 \\ kp \text{ is even} } } \frac{ ((2m-k)p+2)((2m-k)p) \binom{2m}{k} \binom{mp}{\frac{kp}{2}} }{ mp(mp+1) \binom{2mp}{kp} } \Bigg) \nonumber \\
    \leq\ & \frac{10^{9} R^3}{m} \big| \mathfrak A_{2m,mp+1} \big| \,,  \label{eq-crutial-bound-lem-2.2-first-part}
\end{align}
where the last inequality follows from
\begin{align*}
    & \frac{ (kp+1)((2m-k)p+1) \binom{2m}{k}\binom{mp}{\frac{kp+1}{2}} }{ mp(mp+1) \binom{2mp}{kp} } \leq \frac{ 10^{8} }{ m^2 } \mbox{ for } 3 \leq k\leq 2m-3,p\geq 3, kp \mbox{ is odd} \,; \\
    & \frac{ ((2m-k)p+1)((2m-k)p) \binom{2m}{k} \binom{mp}{\frac{kp}{2}} }{ mp(mp+1) \binom{2mp}{kp} } \leq \frac{ 10^{8} }{ m^2 } \mbox{ for } 3 \leq k \leq 2m-3, p \geq 3, kp \mbox{ is even} \,.
\end{align*}
Now we deal with the case that $U|_{\setminus \{ v \}}$ is disconnected for some $v \in [mp+1]$ in the similar manner. In this case, $U$ can be written as $U=U_1 \cup U_2$, where $V(U_1) \cap V(U_2) = \{ v \}$ with $v \in \mathsf L(U_1) \cap \mathsf L(U_2)$. Without losing of generality, we may assume that $1 \leq |\mathsf L(U_1)| \leq 2$. Denote $|E(U_1)|=k$ for some $3 \leq k \leq 2m-3$, we then have $|V(U_1)|=\tfrac{pk+1}{2}$ if $pk$ is odd and $|V(U_1)|=\tfrac{pk}{2}+1$ if $pk$ is odd. Thus, if $pk$ is odd the possible enumeration of $(U_1,U_2)$ in this case is bounded by
\begin{align*}
    & (mp+1) \binom{mp}{\frac{kp-1}{2}} \cdot \big| \mathfrak A_{k,(kp+1)/2;v} \big| \cdot \big| \mathfrak A_{2m-k,((2m-k)p+3)/2;v} \big|  \\
    \overset{\eqref{eq-est-mathfrak-A-particular}}{\leq}\ & R^2 (mp+1) \binom{mp}{ \frac{kp-1}{2} } \cdot \frac{ (kp)! }{ k! (p!)^{k} 2^{\frac{kp-1}{2}} } \cdot \frac{ \binom{((2m-k)p+1)/2}{2} \cdot ((2m-k)p)! }{ (2m-k)! (p!)^{2m-k} 2^{\frac{(2m-k)p-3}{2}} } \\
    =\ & R^2 (mp+1) \binom{mp}{ \frac{kp-1}{2} } \cdot \frac{ ((2m-k)p+1)((2m-k)p-1) \cdot (kp)!((2m-k)p)! }{ k!(2m-k)! (p!)^{2m} 2^{mp+1} } \\
    \overset{\eqref{eq-est-mathfrak-A-general}}{\leq}\ & R^3 (mp+1) \binom{mp}{ \frac{kp-1}{2} } \cdot \big| \mathfrak A_{2m,mp+1} \big| \cdot \frac{ ((2m-k)p+1)((2m-k)p-1) }{ mp(mp+1) } \cdot \binom{2m}{k} \cdot \binom{2pm}{pk}^{-1} \\
    \leq\ & \big| \mathfrak A_{2m,mp+1} \big| \cdot \frac{ R^3 (2m-k)^2 \binom{2m}{k}\binom{mp}{\frac{kp-1}{2}} }{ m \binom{2mp}{kp} }  \,.
\end{align*}
Similarly, when $pk$ is even the possible enumeration of $(U_1,U_2)$ in this case is bounded by
\begin{align*}
        & (mp+1) \binom{mp}{ \frac{kp}{2} } \cdot \big| \mathfrak A_{k,kp/2+1;v} \big| \cdot \big| \mathfrak A_{(2m-k),(2m-k)p/2+1;v} \big|  \\
        \overset{\eqref{eq-est-mathfrak-A-particular}}{\leq}\ & R^2 (mp+1) \binom{mp}{\frac{kp}{2}} \cdot \frac{ \frac{kp}{2} \cdot (kp)!}{ k! (p!)^{k} 2^{\frac{kp}{2}-1}} \cdot \frac{ \frac{(2m-k)p}{2} \cdot ((2m-k)p)! }{ (2m-k)! (p!)^{2m-k} 2^{\frac{(2m-k)p}{2}-1} } \\
        =\ & R^2 (mp+1) \binom{mp}{ \frac{kp}{2} } \cdot \frac{ kp((2m-k)p) \cdot (kp)!((2m-k)p)! }{ k!(2m-k)! 2^{mp} } \\
        \overset{\eqref{eq-est-mathfrak-A-general}}{\leq}\ & R^3 (mp+1) \binom{mp}{ \frac{kp}{2} } \cdot \Big| \mathfrak A_{2m,mp+1} \Big| \cdot \frac{ (kp)((2m-k)p) }{ mp(mp+1) } \cdot \binom{2m}{k}  \cdot \binom{2pm}{pk}^{-1} \\
        \leq\ & \Big| \mathfrak A_{2m,mp+1} \Big| \cdot \frac{ R^3 k(2m-k) \binom{2m}{k} \binom{mp}{\frac{kp}{2}} }{ m\binom{2mp}{kp} }  \,.
\end{align*}
Thus, we get that
\begin{align}
    & \#\Big\{ U \in \mathfrak A_{2m,mp+1}: U \mbox{ does not satisfy Item (3) in Definition~\ref{def-component-hypergraph}} \Big\} \nonumber \\
    \leq\ & R^3 \big| \mathfrak A_{2m,mp+1} \big| \cdot \Bigg( \sum_{ \substack{ 3 \leq k \leq 2m-3 \\ kp \text{ is odd} } } \frac{ (2m-k)^2 \binom{2m}{k}\binom{mp}{\frac{pk-1}{2}} }{ m \binom{2mp}{kp} } + \sum_{ \substack{ 3 \leq k \leq 2m-3 \\ kp \text{ is even} } } \frac{ k(2m-k) \binom{2m}{k} \binom{mp}{\frac{pk}{2}} }{ m\binom{2mp}{kp} } \Bigg) \nonumber \\
    \leq\ & \frac{10^{9}R^3}{m} \big| \mathfrak A_{2m,mp+1} \big| \,, \label{eq-crutial-bound-lem-2.2-second-part}
\end{align}
where the last inequality follows from when $kp$ is odd
\begin{align*}
    & \frac{ (2m-k)^2 \binom{2m}{k}\binom{mp}{\frac{pk-1}{2}} }{ m \binom{2mp}{kp} } \leq \frac{10^{8}}{m} \mbox{ if } k=3,p=3 \,; \\
    & \frac{ (2m-k)^2 \binom{2m}{k}\binom{mp}{\frac{pk-1}{2}} }{ m \binom{2mp}{kp} } \leq \frac{10^{8}}{m} \mbox{ if } 4 \leq k \leq 2m-3,p\geq 3 \mbox{ or } k=3, p \geq 4 
\end{align*}
and when $kp$ is even
\begin{align*}
    &\frac{ k(2m-k) \binom{2m}{k} \binom{mp}{\frac{pk}{2}} }{ m\binom{2mp}{kp} } \leq \frac{10^{8}}{m^2} \mbox{ if } 3 \leq k \leq 2m-3, p \geq 3 \,.
\end{align*}
Combining \eqref{eq-crutial-bound-lem-2.2-first-part} and \eqref{eq-crutial-bound-lem-2.2-second-part}, we get that
\begin{align*}
    & \#\Big\{ U \in \mathfrak A_{2m,mp+1}: U \mbox{ does not satisfy Item (2) or (3) in Definition~\ref{def-component-hypergraph}} \Big\} \nonumber \\
    \leq\ & \tfrac{\cdot 10^{10}R^3}{m} \#\mathfrak A_{2m,mp+1} \,,
\end{align*}
Combined with \eqref{eq-est-mathfrak-A-general}, we get that 
\begin{align*}
    \#\mathfrak U_{m,p} \geq \Big( 1-\frac{10^{10}R^3}{m} \Big) \#\mathfrak A_{2m,mp+1} \overset{\eqref{eq-est-mathfrak-A-general}}{\geq} R^{-1} \Big( 1-\frac{10^{10}R^3}{m} \Big) \frac{ (mp+1)(2mp)! }{ 2^{mp} (p!)^{2m} (2m)! } \,,
\end{align*}
leading to \eqref{eq-lem-2.2-relax}.

\subsection{Proof of Lemma~\ref{lem-control-Aut-necklace-hypergraph}}{\label{subsec:proof-lem-2.4}}

Consider the following mapping
\begin{align*}
    \phi: \mathcal U^{ \otimes \ell } \to \mathcal H,  
\end{align*}
where for each $[U_i] \in \mathcal U (1 \leq i \leq \ell)$ we choose an arbitrary realization $U_i$ with $U_{i} \cap U_{i+1}=\{ v_{i+1} \} \in \mathsf L(U_i)$ and let $\phi([U_1],\ldots,[U_\ell])=[H]$ where $H=U_1 \cup \ldots \cup U_\ell$. By Item~(2) of Lemma~\ref{lem-prelim-mathcal-H}, if $\phi([U_1],\ldots,[U_\ell]) = \phi( [U_1'], \ldots,[U_{\ell}'] )$, there exists $1 \leq k\leq \ell$ such that
\begin{align*}
    \mbox{either } [U_i]=[U_{i+k}] \mbox{ for all } 1 \leq i \leq \ell \mbox{ or } [U_i]=[U_{\ell+i-k}] \mbox{ for all } 1 \leq i \leq \ell \,.
\end{align*}
Thus, we get that
\begin{align*}
    \sum_{[H] \in \mathcal H} \frac{1}{|\mathsf{Aut}(H)|} \geq \frac{1}{2\ell} \cdot \sum_{ [U_1], \ldots, [U_\ell] \in \mathcal U } \frac{1}{ \mathsf{Aut}( \phi([U_1],\ldots,[U_{\ell}]) ) } \,.
\end{align*}
In addition, again using Item~(2) of Lemma~\ref{lem-prelim-mathcal-H}, we see that for all $\pi\in \mathsf{Aut}( \phi([U_1],\ldots,[U_{\ell}]) )$, there exists $1 \leq k \leq i$ such that either $\pi(U_i)=U_{i+k}$ for all $1 \leq i \leq \ell$ or $\pi(U_i)=U_{i-k+\ell}$ for all $1 \leq i \leq \ell$. Since the number of isomorphisms between $U_i$ and $U_j$ is bounded by $\mathsf{Aut}(U_i)$, we see that
\begin{align*}
    \mathsf{Aut}( \phi([U_1],\ldots,[U_{\ell}]) ) \leq 2\ell \prod_{1 \leq i \leq \ell} |\mathsf{Aut}(U_i)| \,.
\end{align*}
Thus, we see that
\begin{align*}
    \sum_{[H] \in \mathcal H} \frac{1}{|\mathsf{Aut}(H)|} &\geq \frac{1}{(2\ell)^2} \cdot \sum_{ [U_1], \ldots, [U_\ell] \in \mathcal U } \prod_{1 \leq i \leq \ell} \frac{1}{|\mathsf{Aut}(U_i)|} \\
    &\overset{\text{Lemma~\ref{lem-control-Aut-component-hypergraph}}}{\geq} \frac{1}{(2\ell)^2} \cdot \Big( R^{-1} \big( 1-\tfrac{10^{10}R^3}{m} \big) \frac{(2pm)!}{2^{mp}(p!)^{m}(mp+1)!(2m)!} \Big)^{\ell} \,. \qedhere
\end{align*}

\subsection{Proof of Lemma~\ref{lem-control-Aut-necklace-chain}}{\label{subsec:proof-lem-2.9}}

We first show 
\begin{equation}{\label{eq-upper-bound-lem-2.9}}
    \big( \beta^{\diamond}_{\mathcal U} \big)^{\ell} \geq \beta_{\mathcal J} \,.
\end{equation}
Define
\begin{equation}{\label{eq-def-widetilde-mathcal-U}}
    \widetilde{U} := \big\{ ([v],[U]): [U] \in \mathcal U, [v] \in \mathsf L_{\diamond}(U) \big\} \,.
\end{equation}
It is clear that for all $[J] \in \mathcal J$, suppose $J=U_1 \cup U_2 \ldots \cup U_{\ell}$ is the desired decomposition in Definition~\ref{def-mathcal-J}, then the following mapping
\begin{align*}
    J \longrightarrow \Big( ([v_1],[U_1]), \ldots, ([v_{\ell}],[U_{\ell}]) \Big)
\end{align*}
is an injection from $\mathcal J$ to $\widetilde{U}^{ \otimes \ell }$. In addition, Item~(2) in Lemma~\ref{lem-prelim-mathcal-J} yields that
\begin{align*}
    |\mathsf{Aut}(J)|= \prod_{1 \leq i \leq \ell} |\mathsf{Aut}_{v_i}(U_i)| \,.
\end{align*}
Thus, 
\begin{align*}
    \beta_{\mathcal J} &\overset{\eqref{eq-def-beta-mathcal-H}}{=} \sum_{ [J] \in \mathcal J } \frac{1}{|\mathsf{Aut}(J)|} \leq \sum_{ ([v_1],[U_1]), \ldots, ([v_{\ell}],[U_{\ell}]) \in \widetilde{U} } \prod_{1 \leq i \leq \ell} \frac{ 1 }{ |\mathsf{Aut}_{v_i}(U_i)| } \\
    &\leq \Big( \sum_{ ([v],[U]) \in \widetilde{U} } \frac{ 1 }{ |\mathsf{Aut}_{v}(U)| } \Big)^{\ell} \overset{\eqref{eq-def-widetilde-mathcal-U},\eqref{eq-def-beta-diamond-U}}{=} \big( \beta^{\diamond}_{\mathcal U} \big)^{\ell} \,.
\end{align*}
Now we show that 
\begin{equation}{\label{eq-lower-bound-lem-2.9}}
    \beta_{\mathcal J} \geq \big( \beta^{\diamond}_{\mathcal U} \big)^{\ell-1} \big( \beta^{\diamond}_{\mathcal U}-2 \big) \,.
\end{equation}
Consider the mapping 
\begin{align*}
    \Big\{ \big( ([v_1],[U_1]), \ldots, ([v_{\ell}],[U_{\ell}]) \big) \in \widetilde{\mathcal U}^{\otimes \ell}: U_1 \not \cong U_{\ell} \Big\} \longrightarrow \mathcal J 
\end{align*}
such that $\big( ([v_1],[U_1]), \ldots, ([v_{\ell}],[U_{\ell}])$ is mapped to $J=U_1 \cup \ldots \cup U_{\ell}$ according to Definition~\ref{def-mathcal-J}. Note that for all $\pi \in \mathsf{Aut}(J)$, we have $\pi( \{ v_1,v_{\ell+1} \} ) =\{ v_1,v_{\ell+1} \}$, and similarly as in Item~(2) of Lemma~\ref{lem-prelim-mathcal-H} we can show that $\pi(v_i) \in \{ v_2,\ldots,v_\ell \}$ for all $2 \leq i \leq \ell$. Thus, we see that either $\pi(U_i)=U_i$ for all $1 \leq i \leq \ell$ or $\pi(U_i)=U_{\ell+1-i}$ for all $1 \leq i \leq \ell$. Thus, since $U_1 \not\cong U_{\ell}$ we see that $\pi(v_i)=v_i$ and $\pi(U_j)=U_j$ for all $\pi\in\mathsf{Aut}(J)$. Thus, we indeed have $J \in \mathcal J$ and thus from Item~(3) in Lemma~\ref{lem-prelim-mathcal-J} we have $|\mathsf{Aut}(J)|=\prod_{1 \leq i \leq \ell} |\mathsf{Aut}_{v_i}(U_i)|$. In conclusion, define $\{ w_i \}=\mathsf L(U_i) \setminus \{ v_i \}$, we have $\big( ([v_1],[U_1]), \ldots, ([v_{\ell}],[U_{\ell}])$ and $\big( ([v_1'],[U_1']), \ldots, ([v_{\ell}'],[U_{\ell}'])$ are mapped to the same element if and only if 
\begin{align*}
    ([v_i'],[U_i'])=([v_i],[U_i]) \mbox{ or } ([v_i'],[U_i'])=([v_{\ell+1-i}],[U_{\ell+1-i}]) \,.
\end{align*}
Thus, 
\begin{align*}
    \beta_{\mathcal J} &\overset{\eqref{eq-def-beta-mathcal-H}}{=} \sum_{ [J] \in \mathcal J } \frac{1}{|\mathsf{Aut}(J)|} \geq \frac{1}{2} \sum_{ ([v_1],[U_1]), \ldots, ([v_{\ell}],[U_{\ell}]) \in \widetilde{\mathcal U}: [U_1] \neq [U_{\ell}] } \prod_{1 \leq i \leq \ell} \frac{ 1 }{ |\mathsf{Aut}_{v_i}(U_i)| } \\
    &\geq \frac{1}{2} \sum_{ ([v_1],[U_1]), \ldots, ([v_{\ell}],[U_{\ell}]) \in \widetilde{\mathcal U} } \prod_{1 \leq i \leq \ell} \frac{ 1 }{ |\mathsf{Aut}_{v_i}(U_i)| } - \sum_{ ([v_1],[U_1]), \ldots, ([v_{\ell-1}],[U_{\ell-1}]) \in \widetilde{\mathcal U} } \prod_{1 \leq i \leq \ell-1} \frac{ 1 }{ |\mathsf{Aut}_{v_i}(U_i)| } \\
    &\overset{\eqref{eq-def-beta-diamond-U}}{=} \frac{1}{2} \big( \beta^{\diamond}_{\mathcal U} \big)^{\ell-1} \big( \beta^{\diamond}_{\mathcal U}-2 \big) \,.
\end{align*}
Combining \eqref{eq-upper-bound-lem-2.9} and \eqref{eq-lower-bound-lem-2.9} yields the first line of inequality. In addition, from \eqref{eq-def-beta-mathcal-H} and \eqref{eq-def-beta-diamond-U} we have $\beta^{\diamond}_{\mathcal U} \geq \beta_{\mathcal U}$. So the second inequality follows from Lemma~\ref{lem-control-Aut-component-hypergraph}.

\section{Postponed proofs in Section~\ref{sec:stat-analysis}}{\label{sec:supp-proofs-sec-3}}

The goal of this section is to provide the proof of Lemma~\ref{lem-very-technical-sec-3-transfer}. We will prove \eqref{eq-technical-part-I}--\eqref{eq-technical-part-IV} separately. Throughout this section, we will always assume the decomposition \eqref{eq-decomposition-S} and \eqref{eq-decomposition-K} holds, and we will always regard $K=U_1' \cup \ldots \cup U_{\ell}'$ as fixed. In addition, using Lemma~\ref{lem-control-Aut-necklace-chain} and \eqref{eq-condition-weak-recovery}, we see that (note that $\beta^{\diamond}_{\mathcal U} \geq \beta_{\mathcal U}$)
\begin{equation}{\label{eq-esay-bound-beta-mathcal-J}}
    \beta_{\mathcal J} \geq \frac{1}{3} \big( \beta^{\diamond}_{\mathcal U} \big)^{\ell}
\end{equation}

\subsection{Proof of (\ref{eq-technical-part-I})}{\label{subsec:proof-lem-3.8-part-1}}

Note that $(S,K) \in \mathsf{Part}_I$ implies that $\mathsf{IND}(S,K)=\{ 1,\ldots,\ell \}$ and thus $S=K$. Thus,
\begin{align*}
    & \sum_{ S:(S,K) \in \mathsf{Part}_I } \big( \lambda^{-2} n^{\frac{p}{2}} \big)^{ |E(S) \cap E(K)| }  \\
    =\ & \big( \lambda^{-2} n^{\frac{p}{2}} \big)^{ |E(K)| } = \big( \lambda^{-2} n^{\frac{p}{2}} \big)^{ 2m\ell } = \lambda^{-4m\ell} n^{pm\ell}  \,. 
\end{align*}

\subsection{Proof of (\ref{eq-technical-part-II})}{\label{subsec:proof-lem-3.8-part-2}}

Note that $(S,K) \in \mathsf{Part}_{II}$ implies that there exists $0 \leq a,b,a+b\leq \ell-1$ such that $\mathsf{SEQ}(S,K)=\{ 1,\ldots,a,\ell-b+1,\ldots,\ell \}$. Thus, from \eqref{eq-def-overline-SEQ(S,K)} we see that $\overline{\mathsf{SEQ}}(S,K)=\{ 1,\ldots,a,\ell-b+1,\ldots,\ell \}$. So we have
\begin{align*}
    E(S) \cap E(K) &= E(U_1' \cup \ldots \cup U_a' \cup U_{\ell+1-b}' \cup \ldots \cup U_{\ell}') \\
    &= E(U_1 \cup \ldots \cup U_a \cup U_{\ell+1-b} \cup \ldots \cup U_{\ell}) \,.
\end{align*}
Thus, we have $|E(S) \cap E(K)|=2m(a+b)$. In addition, since $E(S) \cap E(K) \neq \emptyset$ we see that $a+b \geq 1$. We now control the enumeration of $S$ for all fixed $K$. Note that $S=U_1 \cup \ldots U_{\ell}$ and $U_{\mathtt k}=U_{\mathtt k}'$ is fixed for all $1 \leq \mathtt k \leq a$ and $\ell+1-b \leq \mathtt k \leq \ell$. Thus, it suffices to bound the possible enumeration of $U_{a+1}, \ldots, U_{\ell-b}$. Note that given $U_{1}, \ldots, U_{\mathtt k}$ where $a \leq \mathtt k \leq \ell-b-1$, we have $U_{\mathtt k+1}$ satisfies $[U_{\mathtt k+1}] \in \mathcal U$ with $v_{\mathtt k+1} \in \mathsf L(U_{\mathtt k+1})$ for a fixed vertex $v_{\mathtt k+1}$ (as $v_{\mathtt k+1}=\mathsf L(U_{\mathtt k}) \setminus \{ v_{\mathtt k}\}$ has been decided); thus, using Item~(3) in Lemma~\ref{lem-prelim-result-mathcal-U} the possible choices of $U_{\mathtt k+1}$ is bounded by $[1+o(1)] n^{mp} \beta_{\mathcal U}^{\diamond}$. In addition, given $U_{1}, \ldots, U_{\ell-b-1}, U_{\ell-b+1}, \ldots, U_{\ell}$, then $U_{\ell-b}$ satisfies $[U_{\ell-b}] \in \mathcal U$ with $v_{\ell-b}, v_{\ell-b+1} \in \mathsf L(U_{\ell-b})$ for fixed vertices $v_{\ell-b}, v_{\ell-b+1}$; thus, using Item~(3) in Lemma~\ref{lem-prelim-result-mathcal-U} the possible choices of $U_{k+1}$ is bounded by $[1+o(1)] n^{mp-1} \beta_{\mathcal U}^{\diamond}$. In conclusion, for all fixed $K$ we have  
\begin{align*}
    &\#\Big\{ S: E(S \cap K) = E(U_1' \cup \ldots \cup U_a' \cup U_{\ell+1-b}' \cup \ldots \cup U_{\ell}') \Big\} \\
    \leq\ & \prod_{ a+1 \leq \mathtt k \leq \ell-b } \Big( \sum_{ \substack{ [U] \in \mathcal U \\ [v] \in \mathsf L_{\diamond}(U) } } \frac{ (1+o(1)) n^{mp-\mathbf 1_{\mathtt k=\ell-b}} }{ |\mathsf{Aut}_v(U)| } \Big) = n^{ (\ell-a-b)mp-1 } \cdot \big( \beta^{\diamond}_{\mathcal U} \big)^{ \ell-a-b } \,,
\end{align*}
where the equality follows from \eqref{eq-def-beta-diamond-U}. Thus, we see that
\begin{align*}
    \sum_{ S:(S,K) \in \mathsf{Part}_{II} } \big( \lambda^{-2} n^{\frac{p}{2}} \big)^{ |E(S) \cap E(K)| } &\leq \sum_{1\leq a+b \leq \ell-1} n^{ (\ell-a-b)mp-1 } \cdot \big( \beta^{\diamond}_{\mathcal U} \big)^{ \ell-a-b } \cdot \big( \lambda^{-2} n^{\frac{p}{2}} \big)^{ 2m(a+b) } \\
    & \overset{\text{Lemma~\ref{lem-control-Aut-necklace-chain}}}{\leq} n^{pm\ell-1} \beta_{\mathcal J} \sum_{1\leq a+b \leq \ell-1} \lambda^{-4m(a+b)} \big( \beta^{\diamond}_{\mathcal U} \big)^{ -(a+b) } \\
    & \leq \frac{ \delta^3 }{ (1-2\delta^3)(1-\delta^3)^2 } n^{pm\ell-1} \beta_{\mathcal J} \leq \delta^2 n^{pm\ell-1} \beta_{\mathcal J} \,,
\end{align*}
where the third inequality follows from $\beta^{\diamond}_{\mathcal U} \geq \delta^{-3}$ from \eqref{eq-condition-weak-recovery} and
\begin{equation*}
    \sum_{1\leq a+b \leq \ell-1} \lambda^{-4m(a+b)} \big( \beta^{\diamond}_{\mathcal U} \big)^{ -(a+b) } \overset{\text{Lemma~\ref{lem-control-Aut-component-hypergraph},}\eqref{eq-condition-weak-recovery}}{\leq} \sum_{1\leq a+b \leq \ell-1} \delta^{3(a+b)} \leq \frac{ \delta^3 }{ (1-\delta^3)^2 } \,. 
\end{equation*}

\subsection{Proof of (\ref{eq-technical-part-III})}{\label{subsec:proof-lem-3.8-part-3}}

Note that $(S,K) \in \mathsf{Part}_{III}$ implies that there exists $0 \leq a,b,a+b\leq \ell-1$ such that $\mathsf{IND}(S,K)=\{ 1,\ldots,a, \ell-b+1,\dots,\ell \}$, so $U_{\mathtt k}=U'_{\mathtt k}$ for $\mathtt k \in [1,a] \cup [\ell-b+1,\ell]$ and
\begin{align*}
    & U_1' \cup \ldots \cup U_a' \cup U_{\ell+1-b}' \cup \ldots \cup U_{\ell}' \\
    =\ & U_1 \cup \ldots \cup U_a \cup U_{\ell+1-b} \cup \ldots \cup U_{\ell} \subset S \cap K \,.
\end{align*}
In addition, recall \eqref{eq-def-gamma-t}. Note that $(S,K) \in \mathsf{Part}_{III}$ implies that we have $\gamma=\gamma_1>0$, which means that
\begin{align}
    \Gamma &= \Big\{ 1 \leq \mathtt k \leq \ell: E(U_{\mathtt k}') \cap E(U_{[a+1,\ell-b]}) \neq \emptyset \Big\} \nonumber \\
    &= \Big\{ a+1 \leq \mathtt k \leq \ell-b: E(U_k') \cap E(U_{ [a+1,\ell-b] }) \neq \emptyset \Big\} \neq \emptyset \,,  \label{eq-def-Gamma}
\end{align}
where the second equality follows from $U_{\mathtt k}=U_{\mathtt k}'$ for $1 \leq \mathtt k \leq a$ and $\ell-b+1 \leq \mathtt k \leq \ell$, and thus $E(U_{\mathtt k}') \cap E(U_{[a+1,\ell-b]}) = E(U_{\mathtt k}) \cap E(U_{[a+1,\ell-b]})=\emptyset$. We first argue the following property on $(S,K) \in \mathsf{Part}_{III}$.

\begin{lemma}{\label{lem-key-property-Part-III}}
    for all $(S,K) \in \mathsf{Part}_{III}$ we have
    \begin{align}
        2|V(S) \cap V(K)| \geq p|E(S) \cap E(K)| + 2\gamma +2 + \mathbf 1_{ \{ \gamma=1,a+b \leq \ell-2 \} } \,. \label{eq-key-property-Part-III}
    \end{align}
\end{lemma}
\begin{proof}
    Since $\{ i,j \} \in V(S) \cap V(K)$ and $\mathsf{deg}_S(v)=2$ for all $v \in V(S) \setminus \{ i,j \}$, we see that
    \begin{align*}
        2 = 2|V(K)|-p|E(K)| = \sum_{ v \in V(S) \cap V(K) } (2-\mathsf{deg}_K(v)) \,.
    \end{align*}
    Also, we have
    \begin{align*}
        2|V(S) \cap V(K)| - p|E(S) \cap E(K)| = \sum_{ v \in V(S) \cap V(K) } (2-\mathsf{deg}_{S \cap K}(v)) \,.
    \end{align*}
    Thus, to show \eqref{eq-key-property-Part-III} it suffices to show that
    \begin{align}
        \sum_{ v \in V(S) \cap V(K) } \Big( \mathsf{deg}_K(v) - \mathsf{deg}_{S \cap K}(v) \Big) \geq 2\gamma + \mathbf 1_{ \{ \gamma=1,a+b \leq \ell-2 \} } \,.  \label{eq-key-property-Part-III-transfer}
    \end{align}
    To this end, note that $\gamma=|\Gamma|$ where $\Gamma$ is defined in \eqref{eq-def-Gamma}. For each $\mathtt k \in \Gamma$, we have $E(U_{\mathtt k}') \cap E(U_{a+1,\ell-b})$ and since $\mathtt k \not\in \mathsf{IND}(S,K)$ we have $E(U_{\mathtt k}') \not \subset E(U_{a+1,\ell-b})$. Thus, using Item~(2) Lemma~\ref{lem-prelim-result-mathcal-U} we see that
    \begin{align*}
        \big| \mathsf L(U_{\mathtt k}' \cap U_{[a+1,\ell-b]}) \setminus \mathsf L(U_{\mathtt k}') \big| \geq 2 \,.
    \end{align*}
    Also, it is easy to check that $\mathsf L(U_{\mathtt k}' \cap U_{[a+1,\ell-b]}) \subset \mathsf L(S \cap K)$, thus 
    \begin{align*}
        \Big\{ \mathsf L( U_{\mathtt k}' \cap U_{[a+1,\ell-b]} ) \setminus \mathsf L(U_{\mathtt k}'): \mathtt k \in \Gamma \Big\}
    \end{align*}
    are disjoint and belongs to $\{ v \in V(S) \cap V(K): \mathsf{deg}_K(v) > \mathsf{deg}_{S \cap K}(v) \}$. Finally, it is easy to check that if $a+1 \not \in \Gamma$, then $\mathsf{deg}_{S \cap K}(v_{a+1})< \mathsf{deg}_{K}(v_{a+1})$ and if $\ell-b \not\in \Gamma$, then $\mathsf{deg}_{S \cap K}(v_{\ell-b+1})< \mathsf{deg}_{K}(v_{\ell-b+1})$, thus (note that $v_{a+1},v_{\ell-b+1} \not\in \mathsf L( U_{\mathtt k}' \cap U_{[a+1,\ell-b]}) \setminus \mathsf L(U_{\mathtt k}')$ for all $\mathtt k \in \Gamma$)
    \begin{align*}
        \#\Big\{ v \in V(S) \cap V(K): \mathsf{deg}_K(v) > \mathsf{deg}_{S \cap K}(v) \Big\} &\geq 2|\Gamma| + \mathbf 1_{ \{ a+1 \not \in \Gamma \} } + \mathbf 1_{ \{ \ell-b \not\in \Gamma \} } \\
        &\geq 2\gamma + \mathbf 1_{ \{ \gamma=1,a+b \leq \ell-2 \} } \,,
    \end{align*}
    leading to \eqref{eq-key-property-Part-III-transfer} and thus yields \eqref{eq-key-property-Part-III}.
\end{proof}

We now return to the proof of \eqref{eq-technical-part-III}. Note that
\begin{align*}
    &|E(S) \cap E(K)|= 2m(a+b) + |E(U_{[a+1,\ell-b]}) \cap E(K)| \,, \\
    &|V(S) \cap V(K)|=mp(a+b)+2 + |V(U_{[a+1,\ell-b]}) \cap V(K) \setminus \mathsf L(U_{[a+1,\ell-b]})| \,.
\end{align*}
In addition, we have
\begin{align*}
    |E(U_{[a+1,\ell-b]}) \cap E(K)| \leq 2m \cdot \#\big\{ \mathtt k: E(U_{\mathtt k}') \cap E(U_{[a+1,\ell-b]} \neq \emptyset \big\} \leq 2m\gamma
\end{align*}
For all 
\begin{equation}{\label{eq-constraint-mathtt-E-V-part-3}}
    2m\gamma \geq \mathtt E \geq 0 \mbox{ and } \mathtt V \geq \frac{p\mathtt E + 2(\gamma-1)+\mathbf 1_{ \{ \gamma=1,a+b \leq \ell-2 \} } }{2} \,,
\end{equation}
define $\mathsf{Part}_{III}(a,b,\gamma;\mathtt V,\mathtt E)$ to be the set of $S$ such that 
\begin{itemize}
    \item $(S,K) \in \mathsf{Part}_{III}$;
    \item $\mathsf{SEQ}(S,K)=(0,a+1,\ell-b,\ell+1)$;
    \item $|E(U_{[a+1,\ell-b]}) \cap E(K)|=\mathtt E$;
    \item $|V(U_{[a+1,\ell-b]}) \cap V(K) \setminus \mathsf L(U_{[a+1,\ell-b]})| = \mathtt V$.
\end{itemize}
For all $S \in \mathsf{Part}_{III}(a,b,\gamma;\mathtt V,\mathtt E)$ we have
\begin{align*}
    \big( \lambda^{-2} n^{\frac{p}{2}} \big)^{ |E(S) \cap E(K)| } = \big( \lambda^{-2} n^{\frac{p}{2}} \big)^{ 2m(a+b)+\mathtt E } \leq \lambda^{-4m(a+b)} \lambda^{-2\mathtt E} n^{ \frac{p}{2}(2m(a+b)+\mathtt E) }
\end{align*}
Now we bound the enumeration of $\mathsf{Part}_{III}(a,b,\gamma;\mathtt V,\mathtt E)$. Recall \eqref{eq-decomposition-S} and the fact that $U_{\mathtt k}=U'_{\mathtt k}$ for all $\mathtt k \in [1,a] \cup [\ell-b+1,\ell]$. Thus, it suffices to bound the possible enumeration of $U_{[a+1,\ell-b]}$ such that $|U_{[a+1,\ell-b]} \cap V(K)|=\mathtt V+2$. We now argue that given $v_{a+1}$ and $v_{\ell-b+1}$
\begin{equation}{\label{eq-enu-U-[a,b]}}
    \#\Big\{ U_{[a+1,\ell-b]}: \big| U_{[a+1,\ell-b]} \cap V(K) \big| =\mathtt V+2 \Big\} \leq \big( m^2 p^2 \ell^2 \big)^{\mathtt V} \big( \beta^{\diamond}_{\mathcal U} \big)^{ \ell-a-b } n^{ mp(\ell-a-b)-1-\mathtt V } \,.
\end{equation}
To this end, define
\begin{align*}
    \mathtt X_{\mathtt k} = \big| (U_{\mathtt k} \cap V(K)) \setminus \{ v_{\mathtt k} \} \big|-\mathbf 1_{\mathtt k=\ell-b} \,.
\end{align*}
We have $\sum_{a+1 \leq \mathtt k \leq \ell-b} \mathtt X_{\mathtt k}=\mathtt V$; given $\mathtt X_{\mathtt k}$, the enumeration of $\mathtt W_{\mathtt k}=(V(U_{\mathtt k}) \cap V(K)) \setminus \{ v_{\mathtt k} \}$ is bounded by $\binom{|V(K)|}{\mathtt X_{\mathtt k}}$; finally, using Item~(4) in Lemma~\ref{lem-prelim-result-mathcal-U} given $\mathtt W_{\mathtt k}$, the possible enumeration of $[U_{\mathtt k}] \in \mathcal U$ is bounded by $\beta^{\diamond}_{\mathcal U} n^{ mp-\mathtt X_{\mathtt k}-\mathbf 1_{\mathtt k=\ell-b} } (mp)^{\mathtt X_{\mathtt k}}$. Thus, the left hand side of \eqref{eq-enu-U-[a,b]} is bounded by
\begin{align*}
    & \sum_{ \mathtt X_{a+1}+\ldots+\mathtt X_{\ell-b}=\mathtt V } \prod_{ a+1 \leq \mathtt k \leq \ell-b } \Bigg( \binom{|V(K)|}{\mathtt X_{\mathtt k}} \beta^{\diamond}_{\mathcal U} n^{ mp-\mathtt X_{\mathtt k}-\mathbf 1_{\mathtt k=\ell-b} } (mp)^{\mathtt X_{\mathtt k}} \Bigg) \\
    \leq\ & \sum_{ \mathtt X_{a+1}+\ldots+\mathtt X_{\ell-b}=\mathtt V } \prod_{ a+1 \leq \mathtt k \leq \ell-b } \big( m^2 p^2 \ell \big)^{\mathtt V} \big( \beta^{\diamond}_{\mathcal U} \big)^{ \ell-a-b } n^{ mp(\ell-a-b)-1-\mathtt V } \\
    \leq\ & \big( m^2 p^2 \ell^2 \big)^{\mathtt V} \big( \beta^{\diamond}_{\mathcal U} \big)^{ \ell-a-b } n^{ mp(\ell-a-b)-1-\mathtt V } \,,
\end{align*}
where the last inequality follows from $\#\{ \mathtt X_{a+1}+\ldots+\mathtt X_{\ell-b}=\mathtt V \} \leq \ell^{\mathtt V}$. Thus, we have
\begin{align*}
    & \sum_{ S:(S,K) \in \mathsf{Part}_{III} } \big( \lambda^{-2} n^{\frac{p}{2}} \big)^{ |E(S) \cap E(K)| } \\
    =\ & \sum_{ \substack{ a+b \leq \ell-1 \\ \gamma \geq 1 } } \sum_{ (\mathtt V,\mathtt E) \text{ with } \eqref{eq-constraint-mathtt-E-V-part-3} } \lambda^{-4m(a+b)} \lambda^{-2\mathtt E} n^{ \frac{p}{2}(2m(a+b)+\mathtt E) } \cdot \big( m^2 p^2 \ell^2 \big)^{\mathtt V} \big( \beta^{\diamond}_{\mathcal U} \big)^{ \ell-a-b } n^{ mp(\ell-a-b)-1-\mathtt V } \\
    \overset{\eqref{eq-esay-bound-beta-mathcal-J}}{\leq}\ & 3\beta_{\mathcal J} \sum_{\substack{ a+b \leq \ell-1 \\ \gamma \geq 1 }} \sum_{ (\mathtt V,\mathtt E) \text{ with } \eqref{eq-constraint-mathtt-E-V-part-3} } \big( \lambda^{4m} \beta^{\diamond}_{\mathcal U} \big)^{-(a+b)} n^{ mp\ell+ \frac{p\mathtt E}{2}-1-\mathtt V } \big(m^2 p^2 \ell^2 \big)^{\mathtt V}  \lambda^{-2\mathtt E}  \\
    \leq\ & [3+o(1)] \beta_{\mathcal J} \sum_{\substack{ a+b \leq \ell-1 \\ \gamma \geq 1 }} \sum_{ 0 \leq \mathtt E \leq 2m\gamma } \lambda^{-2\mathtt E} \cdot \big( \lambda^{4m} \beta^{\diamond}_{\mathcal U} \big)^{-(a+b)} \big( m^2 p^2 \ell^2 \big)^{p\mathtt E+1} n^{ mp\ell-1-(\gamma-1)-\frac{1}{2}\cdot\mathbf 1_{\gamma=1,a+b\leq \ell-2} } \\
    \leq\ & 4 \beta_{\mathcal J} \sum_{\substack{ a+b \leq \ell-1 \\ \gamma \geq 1 }} \lambda^{-4m\gamma} (m^2 p^2 \ell^2)^{2mp\gamma+1} n^{mp\ell-1-(\gamma-1)-\frac{1}{2}\cdot\mathbf 1_{\gamma=1,a+b\leq \ell-2} } \\
    =\ & 4 \beta_{\mathcal J} \sum_{\substack{ a+b \leq \ell-1 }} \big( \lambda^{4m} \beta^{\diamond}_{\mathcal U} \big)^{-(a+b)} \lambda^{-4m} (m^2 p^2 \ell^2)^{2mp+1} n^{mp\ell-1-\frac{1}{2} \cdot \mathbf 1_{a+b\leq \ell-2} } \\
    & + 4 \beta_{\mathcal J} \sum_{\substack{ a+b \leq \ell-1 \\ \gamma \geq 2 }} \big( \lambda^{4m} \beta^{\diamond}_{\mathcal U} \big)^{-(a+b)} \lambda^{-4m\gamma} (m^2 p^2 \ell^2)^{2mp\gamma+1} n^{mp\ell-1-(\gamma-1) } \leq n^{-\frac{1}{2}+o(1)} \beta_{\mathcal J} n^{mp\ell-1} \,. 
\end{align*}

\subsection{Proof of (\ref{eq-technical-part-IV})}{\label{subsec:proof-lem-3.8-part-4}}

Recall \eqref{eq-def-SEQ(S,K)} and \eqref{eq-def-overline-SEQ(S,K)}. Note that $(S,K) \in \mathsf{Part}_{IV}$ implies that there exists $\mathtt T \geq 3$ such that
\begin{align*}
    & \mathsf{IND}(S,K) = \big\{ 1,\ldots,\ell \big\} \setminus \big( [\mathtt q_1,\mathtt s_2] \cup \ldots \cup [\mathtt q_{\mathtt T-1}, \mathtt s_{\mathtt T}] \big) \,, \\
    & \mathsf{SEQ}(S,K) := \big( 0, \mathtt q_1, \mathtt s_2, \mathtt q_2, \ldots, \mathtt s_{\mathtt T}, \ell+1 \big) \,; \\
    & \overline{\mathsf{SEQ}}(S,K) := \big( 0, \mathtt q_1', \mathtt s_2', \mathtt q_2', \ldots, \mathtt s_{\mathtt T}', \ell+1 \big) \,.
\end{align*}
In addition, recall that \eqref{eq-def-gamma-t}, we have
\begin{align*}
    \gamma_{\mathtt t} = \#\Gamma_{\mathtt t} \mbox{ where } \Gamma_{\mathtt t} &= \Big\{ 1 \leq \mathtt k \leq \ell: E(U_{\mathtt k}') \cap E(U_{[\mathtt q_{\mathtt t},\mathtt s_{\mathtt t+1}]}) \neq \emptyset \Big\} \\
    & = \Big\{ \mathtt k \in [\mathtt q_1',\mathtt s_2'] \cup \ldots \cup [\mathtt q_{\mathtt T-1}',\mathtt s_{\mathtt T}']: E(U_{\mathtt k}') \cap E(U_{[\mathtt q_{\mathtt t},\mathtt s_{\mathtt T+1}]}) \neq\emptyset \Big\} \,,
\end{align*}
where the second equality follows from 
\begin{align*}
    E(U'_{\mathtt k}) \cap E(U_{[\mathtt q_{\mathtt t},\mathtt s_{\mathtt t+1}]}) = E( U_{\phi^{-1}(\mathtt k)} ) \cap E(U_{[\mathtt q_{\mathtt t},\mathtt s_{\mathtt t+1}]}) = \emptyset
\end{align*}
for all $\mathtt k \not\in[\mathtt q_1',\mathtt s_2'] \cup \ldots \cup [\mathtt q_{\mathtt T-1}',\mathtt s_{\mathtt T}']$. We first show the following lemma:

\begin{lemma}{\label{lem-key-property-Part-IV}}
    For all $(S,K) \in \mathsf{Part}_{IV}$ we have
    \begin{equation}{\label{eq-key-property-Part-IV}}
        2|V(S) \cap V(K)|-p|E(S) \cap E(K)| \geq 2 + 2\sum_{1 \leq \mathtt t \leq \mathtt T-1} \gamma_{\mathtt t} + 2\Big( \mathtt T-1-\sum_{1 \leq \mathtt t \leq \mathtt T-1} \gamma_{\mathtt t} \Big)_+ \,.
    \end{equation}
\end{lemma}
\begin{proof}
    Similar as in Lemma~\ref{lem-key-property-Part-III}, it suffices to show that
    \begin{align}{\label{eq-key-property-Part-IV-transfer}}
        \sum_{ v \in V(S) \cap V(K) } \Big( \mathsf{deg}_K(v) - \mathsf{deg}_{S \cap K}(v) \Big) \geq 2\sum_{1 \leq \mathtt t \leq \mathtt T-1} \gamma_{\mathtt t} + 2\Big( \mathtt T-1-\sum_{1 \leq \mathtt t \leq \mathtt T-1} \gamma_{\mathtt t} \Big)_+ \,.  
    \end{align}
    For each $\mathtt k \in \Gamma_{\mathtt t}$, we have $E(U_{\mathtt k}') \cap E(U_{[\mathtt q_{\mathtt t},\mathtt s_{\mathtt t+1}]}) \neq \emptyset$ and $E(U_{\mathtt k}') \cap E(U_{[\mathtt q_{\mathtt t},\mathtt s_{\mathtt t+1}]}) \neq E(U_{\mathtt k}')$, thus using Item~(2) in Lemma~\ref{lem-prelim-result-mathcal-U} we get that
    \begin{align*}
        \Big| \mathsf{L}( U_{\mathtt k}' \cap U_{[\mathtt q_{\mathtt t}, \mathtt s_{\mathtt t+1}]} ) \setminus \mathsf{L}(U_{\mathtt k}') \Big| \geq 2 \,.
    \end{align*}
    Thus, using the fact that $\{ V(U_{\mathtt k}') \setminus \mathsf{L}(U_{\mathtt k}'): 1 \leq \mathtt k \leq \ell \}$ are disjoint and $\{ V(U_{[\mathtt q_{\mathtt t},\mathtt s_{\mathtt t+1}]}): 1 \leq \mathtt t \leq \mathtt T-1 \}$ are also disjoint, we have
    \begin{align*}
        & \#\Big\{ v \in \cup_{1 \leq \mathtt k \leq \ell} V(U_{\mathtt t}') \setminus \mathsf{L}(U_{\mathtt k}'): v \in V(S) \cap V(K), \mathsf{deg}_{S \cap K}(v) < \mathsf{deg}_K(v) \Big\} \\
        \geq\ & \sum_{ 1 \leq \mathtt t \leq \mathtt T-1 } \sum_{ \mathtt k \in \Gamma_{\mathtt t} } \Big| \mathsf{L}( U_{\mathtt k}' \cap U_{[\mathtt q_{\mathtt t},\mathtt s_{\mathtt t+1}]} ) \setminus \mathsf{L}(U_{\mathtt k}') \Big| = 2\sum_{1 \leq \mathtt t \leq \mathtt T-1} \gamma_{\mathtt t} \,.
    \end{align*}
    In addition, note that 
    \begin{align*}
        & \mathsf{deg}_{S \cap K}(v_{\mathtt q_{\mathtt t}'}') < \mathsf{deg}_{K}(v_{\mathtt q_{\mathtt t}'}') \mbox{ if } v_{\mathtt q_{\mathtt t}'}' \not\in \cup_{1 \leq \mathtt t \leq \mathtt T-1} \Gamma_{\mathtt t} \,;  \\
        & \mathsf{deg}_{S \cap K}(v_{\mathtt s_{\mathtt t+1}'+1}') < \mathsf{deg}_{K}(v_{\mathtt s_{\mathtt t+1}'+1}') \mbox{ if } \mathtt s_{\mathtt t+1}' \not\in \cup_{1 \leq \mathtt t \leq \mathtt T-1} \Gamma_{\mathtt t} \,.
    \end{align*}
    So we have (below we write $\Gamma=\cup_{1 \leq \mathtt t \leq \mathtt T-1} \Gamma_{\mathtt t}$)
    \begin{align*}
        \#\Big\{ 1 \leq \mathtt k \leq \ell+1: \mathsf{deg}_{S \cap K}(v_{\mathtt k}') \geq \mathsf{deg}_{K}(v_{\mathtt k}') \Big\} &\geq 2(K-1) - \sum_{ 1 \leq \mathtt t \leq \mathtt T-1 } \mathbf 1_{ \{ \mathtt q_{\mathtt t}' \not\in \Gamma \} } - \sum_{ 1 \leq \mathtt t \leq \mathtt T-1 } \mathbf 1_{ \{ \mathtt s_{\mathtt t+1}' \not\in \Gamma \} } \\
        &\geq \Big( 2(\mathtt T-1) - |\Gamma|/2 \Big)_+ = 2\Big( \mathtt T-1-\sum_{1 \leq \mathtt t \leq \mathtt T-1} \gamma_{\mathtt t} \Big)_+ \,.
    \end{align*}
    In conclusion, we have \eqref{eq-key-property-Part-IV-transfer} holds.
\end{proof}

Now we return to the proof of \eqref{eq-technical-part-IV}. For 
\begin{align*}
    &\mathtt T \geq 3, \ \overrightarrow{\gamma}=(\gamma_{\mathtt t}:1 \leq \mathtt t \leq \mathtt T-1) \,, \\
    &\mathtt{SEQ}=( \mathtt s_1=0,\mathtt q_1, \ldots, \mathtt s_{\mathtt T}, \mathtt q_{\mathtt T}=\ell+1 ), \ \overline{\mathtt{SEQ}}=( \mathtt s_1'=0,\mathtt q_1', \ldots, \mathtt s_{\mathtt T}', \mathtt q_{\mathtt T}'=\ell+1 )  \\
    &\overrightarrow{\mathtt E} = (\mathtt E_{\mathtt t}:1\leq \mathtt t \leq \mathtt T-1), \ \overrightarrow{\mathtt V} = (\mathtt V_{\mathtt t}:1\leq \mathtt t \leq \mathtt T-1) \,, 
\end{align*}
define $\mathsf{Part}_{IV}(\mathtt T, \mathtt{SEQ}, \overline{\mathtt{SEQ}}, \overrightarrow{\gamma}; \overrightarrow{E}, \overrightarrow{V})$ to be the set of $S \subset \mathsf K_n$ such that
\begin{itemize}
    \item $(S,K) \in \mathsf{Part}_{IV}$ and $E(S) \cap E(K) \neq \emptyset$;
    \item $\mathsf{SEQ}(S,K)=\mathtt{SEQ}$ and $\overline{\mathsf{SEQ}}(S,K)=\overline{\mathtt{SEQ}}$;
    \item $|E(U_{[\mathtt q_{\mathtt t},\mathtt s_{\mathtt t+1}]}) \cap E(K)|=\mathtt E_{\mathtt t}$ and $ |V(U_{[\mathtt q_{\mathtt t}, \mathtt s_{\mathtt t+1}]}) \cap V(K)|=\mathtt V_{\mathtt t}+2$ for all $1 \leq \mathtt t \leq \mathtt T-1$;
\end{itemize}
Then for all $S \in \mathsf{Part}_{IV}(\mathtt T,\mathtt{SEQ}, \overline{\mathtt{SEQ}}, \overrightarrow{\gamma}; \overrightarrow{E}, \overrightarrow{V})$ we have
\begin{align*}
    &|E(S \cap K)|= \sum_{1 \leq \mathtt t \leq \mathtt T} 2m(\mathtt q_{\mathtt t}-\mathtt s_{\mathtt t}-1) + \sum_{1 \leq \mathtt t \leq \mathtt T-1} \mathtt E_{\mathtt t} \,, \\
    &|V(S \cap K)| = \sum_{1 \leq \mathtt t \leq \mathtt T} \Big( mp(\mathtt q_{\mathtt t}-\mathtt s_{\mathtt t}-1)+1 \Big) + \sum_{1 \leq \mathtt t \leq \mathtt T-1} \mathtt V_{\mathtt t} \,.
\end{align*}
Using Lemma~\ref{lem-key-property-Part-IV}, we then have
\begin{equation}{\label{eq-constraint-Part-IV}}
    \begin{aligned}
        & \sum_{1 \leq \mathtt t \leq \mathtt T-1} \mathtt E_{\mathtt t} \leq mp \sum_{1 \leq \mathtt t \leq \mathtt T-1} \#\big\{ 1 \leq \mathtt k \leq \ell: E(U_{\mathtt k}') \cap E(U_{[\mathtt q_{\mathtt t},\mathtt s_{\mathtt t+1}]}) \neq \emptyset \big\} \leq mp\sum_{1 \leq \mathtt t \leq \mathtt T-1} \gamma_{\mathtt t} \,. \\
        & \sum_{1 \leq \mathtt t \leq \mathtt T-1} \mathtt V_{\mathtt t} - \frac{p}{2} \sum_{1 \leq j \leq \mathtt T-1} \mathtt E_{\mathtt t} \geq -\Big( \mathtt T-1-\sum_{1 \leq \mathtt t \leq \mathtt T-1} \gamma_{\mathtt t} \Big) + \Big( \mathtt T-1-\sum_{1 \leq \mathtt t \leq \mathtt T-1} \gamma_{\mathtt t} \Big)_+ \,.
    \end{aligned}
\end{equation}
For all $S \in \mathsf{Part}_{IV}( K, \overrightarrow{\gamma}; \overrightarrow{E}, \overrightarrow{V})$ we have
\begin{align*}
    \big( \lambda^{-2} n^{\frac{p}{2}} \big)^{ |E(S) \cap E(k)| } = \big( \lambda^{-2} n^{\frac{p}{2}} \big)^{ \sum_{1 \leq \mathtt w \leq \mathtt T} 2m(\mathtt q_{\mathtt w}-\mathtt s_{\mathtt w}-1) + \sum_{1 \leq \mathtt t \leq \mathtt T-1} \mathtt E_{\mathtt t} } \,.
\end{align*}
Now we bound the enumerations of $\mathsf{Part}_{IV}( K,\mathtt{SEQ}, \overline{\mathtt{SEQ}}, \overrightarrow{\gamma}; \overrightarrow{E}, \overrightarrow{V})$. Recall \eqref{eq-decomposition-S}. Recall that there exists a bijection 
\begin{align*}
    \phi:[\ell] \setminus \big( [\mathtt q_1,\mathtt s_2] \cup \ldots \cup [\mathtt q_{\mathtt T-1},\mathtt s_{\mathtt T}] \big) \to [\ell] \setminus \big( [\mathtt q_1',\mathtt s_2'] \cup \ldots \cup [\mathtt q_{\mathtt T-1}',\mathtt s_{\mathtt T}'] \big)
\end{align*}
such that $U_{\mathtt k}=U_{\phi(\mathtt k)}$. Thus, to decide $\{ U_{\mathtt k}: \mathtt k \in [\ell] \setminus ( [\mathtt q_1,\mathtt s_2] \cup \ldots \cup [\mathtt q_{\mathtt T-1},\mathtt s_{\mathtt T}] ) \}$ it suffices to decide $\phi$. Note that by \eqref{eq-property-phi}, we see that $\phi$ is determined by
\begin{align*}
    \phi(\mathtt q_1-1), \phi(\mathtt s_2+1), \ldots , \phi(\mathtt q_{\mathtt T-1}-1), \phi(\mathtt s_{\mathtt T}+1) \,.
\end{align*}
Thus, the possible choices of $\phi$ is bounded by $\ell^{2(\mathtt T-1)}$. Now it suffices to bound the possible realizations of $\{ U_{[\mathtt q_{\mathtt t},\mathtt s_{\mathtt t+1}]} \}$. Using \eqref{eq-enu-U-[a,b]}, we get that the enumeration of each $U_{[\mathtt q_{\mathtt t},\mathtt s_{\mathtt t+1}]}$ is bounded by
\begin{align*}
    \big( \beta^{\diamond}_{\mathcal U} \big)^{\mathtt s_{\mathtt t+1}-\mathtt q_{\mathtt t}+1} (m^2 p^2 \ell^2)^{\mathtt V_{\mathtt t}} n^{mp(\mathtt s_{\mathtt t+1}-\mathtt q_{\mathtt t})-1-\mathtt V_{\mathtt t}} \,.
\end{align*}
Thus, we have (below we write $|\overrightarrow{\gamma}|=\sum_{1 \leq \mathtt t \leq \mathtt T-1} \gamma_{\mathtt t}, |\overrightarrow{\mathtt V}|=\sum_{1 \leq \mathtt t \leq \mathtt T-1} \mathtt V_{\mathtt t}$ and $|\overrightarrow{\mathtt E}|=\sum_{1 \leq \mathtt t \leq \mathtt T-1} \mathtt E_{\mathtt t}$)
\begin{align*}
    & \sum_{ S:(S,K) \in \mathsf{Part}_{IV} } \big( \lambda^{-2} n^{\frac{p}{2}} \big)^{ |E(S) \cap E(K)| } \\
    =\ & \sum_{\mathtt T \geq 3} \sum_{ \mathtt{SEQ},\overline{\mathtt{SEQ}} } \sum_{\overrightarrow{\gamma}} \sum_{ \substack{  \overrightarrow{\mathtt E},\overrightarrow{\mathtt V} \text{ with } \eqref{eq-constraint-Part-IV} } } \big( \lambda^{-2} n^{\frac{p}{2}} \big)^{ \sum_{1 \leq \mathtt t \leq \mathtt T} 2m(\mathtt q_{\mathtt t}-\mathtt s_{\mathtt t}-1) + \sum_{1 \leq \mathtt t \leq \mathtt T-1} \mathtt E_{\mathtt t} }  \\
    & \cdot \ell^{2(\mathtt T-1)} \big( \beta^{\diamond}_{\mathcal U} \big)^{ \sum_{1 \leq \mathtt t \leq \mathtt T-1} (\mathtt s_{\mathtt t+1}-\mathtt q_{\mathtt t}+1)} (m^2 p^2 \ell^2)^{\sum_{1 \leq \mathtt t \leq \mathtt T-1}\mathtt V_{\mathtt t}} n^{ \sum_{1 \leq \mathtt t \leq \mathtt T-1}( mp(\mathtt s_{\mathtt t+1}-\mathtt q_{\mathtt t})-1-\mathtt V_{\mathtt t})} \\
    \overset{\eqref{eq-esay-bound-beta-mathcal-J}}{\leq}\ & 3\beta_{\mathcal J} \sum_{\mathtt T \geq 3} \sum_{ \mathtt{SEQ},\overline{\mathtt{SEQ}} } \sum_{\overrightarrow{\gamma}} \sum_{ \substack{ \overrightarrow{\mathtt E},\overrightarrow{\mathtt V} \text{ with } \eqref{eq-constraint-Part-IV} } } \ell^{2(\mathtt T-1)} (m^2 p^2 \ell^2)^{|\overrightarrow{\mathtt V}|} n^{mp\ell-(\mathtt T-1)-\frac{1}{2}|\overrightarrow{\mathtt V}|+\frac{p}{2}|\overrightarrow{\mathtt E}|} \\
    \leq\ & [3+o(1)] \cdot \beta_{\mathcal J} \sum_{\mathtt T \geq 3} \sum_{\overrightarrow{\gamma}} \sum_{ \mathtt{SEQ}, \overline{\mathtt{SEQ}} } \sum_{ \substack{ \overrightarrow{\mathtt E} \text{ with } \eqref{eq-constraint-Part-IV} } } \ell^{2(\mathtt T-1)} (m^2 p^2 \ell^2)^{p|\overrightarrow{\mathtt E}|} n^{mp\ell-(\mathtt T-1)} n^{ -(|\overrightarrow{\gamma}|-\mathtt T+1)_+ } \\
    \leq\ & [3+o(1)] \cdot \beta_{\mathcal J} \sum_{\mathtt T \geq 3} \sum_{\overrightarrow{\gamma}} \ell^{6(\mathtt T-1)} (m^2 p^2 \ell^2)^{mp^2|\overrightarrow{\gamma}|} n^{mp\ell-(\mathtt T-1)} n^{ -(|\overrightarrow{\gamma}|-\mathtt T+1)_+ } \\
    \leq\ & [3+o(1)] \cdot \beta_{\mathcal J} \sum_{\mathtt T \geq 3} \ell^{6(\mathtt T-1)} (m^2 p^2 \ell^2)^{mp^2\mathtt T} n^{mp\ell-(\mathtt T-1)} = n^{-1+o(1)} n^{mp\ell-1} \beta_{\mathcal J} \,,
\end{align*}
where in the third inequality we use $\#\{ \mathtt{SEQ}, \overline{\mathtt{SEQ}} \} \leq \ell^{4(\mathtt T-1)}$.

\section{Postponed proofs in Section~\ref{sec:hypergraph-color-coding}}{\label{sec:supp-proofs-sec-4}}

\subsection{Proof of Proposition~\ref{prop-approximate-detection-statistics}}{\label{subsec:proof-prop-4.1}}

Recall \eqref{eq-def-f-mathcal-H} and \eqref{eq-def-widetilde-f-H}. We then have
\begin{align*}
    &\Big( \widetilde{f}_{\mathcal H} - f_{\mathcal H} \Big)^2 = \Bigg( \frac{1}{\sqrt{ \beta_{\mathcal H} n^{pm\ell} }} \sum_{ [H] \in \mathcal H } \sum_{ S \cong H } f_{S}(\bm Y) \cdot \Big( \frac{1}{t} \sum_{i=1}^{t} \frac{1}{r} \chi_{\tau_i}(V(S))-1 \Big) \Bigg)^2 \\
    =\ & \frac{ 1 }{ \beta_{\mathcal H} n^{pm\ell} } \sum_{ [H],[I] \in \mathcal H } \sum_{ \substack{ S \cong H \\ K \cong I } } f_{S}(\bm Y) f_{K}(\bm Y) \cdot \Big( \frac{1}{t} \sum_{\mathtt k=1}^{t} \frac{1}{r} \chi_{\tau_{\mathtt k}}(V(S))-1 \Big) \Big( \frac{1}{t} \sum_{\mathtt h=1}^{t} \frac{1}{r} \chi_{\tau_{\mathtt h}}(V(K))-1 \Big) \,.
\end{align*}
Note that by conditioning on $\bm Y$, by the randomness over $\{ \tau_{\mathtt k}:1 \leq \mathtt k \leq t \}$ we have
\begin{align*}
    \mathbb E\Big[ \Big( \frac{1}{t} \sum_{\mathtt k=1}^{t} \frac{1}{r} \chi_{\tau_{\mathtt k}}(V(S))-1 \Big) \Big( \frac{1}{t} \sum_{\mathtt h=1}^{t} \frac{1}{r} \chi_{\tau_{\mathtt h}}(V(K))-1 \Big) \Big]=0 \mbox{ if } V(S) \cap V(K)=\emptyset
\end{align*}
and 
\begin{align*}
    &\mathbb E\Big[ \Big( \frac{1}{t} \sum_{\mathtt k=1}^{t} \frac{1}{r} \chi_{\tau_{\mathtt k}}(V(S))-1 \Big) \Big( \frac{1}{t} \sum_{\mathtt h=1}^{t} \frac{1}{r} \chi_{\tau_{\mathtt h}}(V(K))-1 \Big) \Big] \\
    \leq\ & \mathbb E\Big[  \frac{1}{t^2} \sum_{\mathtt k=1}^{t} \frac{1}{r^2} \chi_{\tau_{\mathtt k}}(V(S)) \chi_{\tau_{\mathtt k}}(V(K)) \Big] \leq \frac{1}{tr} = 1 \,.
\end{align*}
Combining the fact that $\mathbb E_{\Pb}[f_{S}(\bm Y) f_{K}(\bm Y)], \mathbb E_{\Qb}[f_{S}(\bm Y) f_{K}(\bm Y)] \geq 0$, we get that
\begin{align}{\label{eq-approx-error-Qb}}
    \mathbb E_{\Qb}\Big[ ( \widetilde{f}_{\mathcal H} - f_{\mathcal H})^2 \Big] \leq \frac{ 1 }{ \beta_{\mathcal H} n^{pm\ell} } \sum_{ [H],[I] \in \mathcal H } \sum_{ \substack{ S \cong H, K \cong I \\ V(S) \cap V(K) \neq \emptyset } } \mathbb E_{\Qb}[f_{S}(\bm Y) f_{K}(\bm Y)] \,.
\end{align}
and 
\begin{align}{\label{eq-approx-error-Pb}}
    \mathbb E_{\Pb}\Big[ ( \widetilde{f}_{\mathcal H} - f_{\mathcal H})^2 \Big] \leq \frac{ 1 }{ \beta_{\mathcal H} n^{pm\ell} } \sum_{ [H],[I] \in \mathcal H } \sum_{ \substack{ S \cong H, K \cong I \\ V(S) \cap V(K) \neq \emptyset } } \mathbb E_{\Pb}[f_{S}(\bm Y) f_{K}(\bm Y)] \,.
\end{align}
We first bound \eqref{eq-approx-error-Qb}. By \eqref{eq-standard-orthogonal} we have
\begin{align*}
    \eqref{eq-approx-error-Qb} &= \frac{ 1 }{ \beta_{\mathcal H} n^{pm\ell} } \sum_{ [H],[I] \in \mathcal H } \sum_{ \substack{ S \cong H, K \cong I \\ V(S) \cap V(K) \neq \emptyset } } \mathbf 1_{ \{ S=K \} } \\
    &= \frac{ 1 }{ \beta_{\mathcal H} n^{pm\ell} } \sum_{ [H] \in \mathcal H } \#\big\{ S \subset \mathsf K_n: S \cong H \big\} \overset{\text{Lemma~\ref{lem-prelim-graphs},(1)}}{=} \frac{ 1+o(1) }{ \beta_{\mathcal H} n^{pm\ell} } \sum_{ [H] \in \mathcal H } \frac{n^{pm\ell}}{|\mathsf{Aut}(H)|} \\
    & \overset{\eqref{eq-def-beta-mathcal-H}}{=} 1+o(1) \overset{\text{Lemma~\ref{lem-mean-var-f-H-part-1}}}{=} o(1) \cdot \mathbb E_{\Pb}[f]^2 \,.
\end{align*}
Now we bound \eqref{eq-approx-error-Pb}. Using Lemma~\ref{lem-est-cov-f-S-f-K}, we get that
\begin{align*}
    \mathbb E_{\Pb}[f_{S}(\bm Y) f_{K}(\bm Y)] &= \big( \lambda n^{-\frac{p}{4}} \big)^{|E(S) \triangle E(K)|} \big( 1+\lambda^2 n^{-\frac{p}{2}} \big)^{ |E(S) \cap E(K)| } \\
    &= [1+o(1)] \lambda^{4m\ell} n^{-pm\ell} \cdot \big( \lambda^{-1} n^{\frac{p}{4}} \big)^{2|E(S) \cap E(K)|} \,.
\end{align*}
Combined with $\mathbb E_{\Pb}[f_{\mathcal H}]=[1+o(1)] \lambda^{2m\ell} \sqrt{\beta_{\mathcal H}}$ in Lemma~\ref{lem-mean-var-f-H-part-1}, we get that
\begin{align}
    \frac{ \mathbb E_{\Pb}[( \widetilde{f}_{\mathcal H} - f_{\mathcal H})^2] }{ \mathbb E_{\Pb}[f_{\mathcal H}]^2 } &\leq \frac{ 1+o(1) }{ \beta_{\mathcal H}^2 n^{2pm\ell} } \sum_{ [H],[I] \in \mathcal H } \sum_{ \substack{ S \cong H, K \cong I \\ V(S) \cap V(K) \neq \emptyset } } \big( \lambda^{-1} n^{\frac{p}{4}} \big)^{2|E(S) \cap E(K)|} \nonumber \\
    &= \frac{ 1+o(1) }{ \beta_{\mathcal H}^2 n^{2pm\ell} } \sum_{1 \leq h \leq pm\ell} \sum_{ [H],[I] \in \mathcal H } \sum_{ \substack{ S \cong H, K \cong I \\ |V(S) \cap V(K)|=h } } \big( \lambda^{-1} n^{\frac{p}{4}} \big)^{2|E(S) \cap E(K)|} \,. \label{eq-approx-error-Pb-relax-1}
\end{align}
Similar as in \eqref{eq-var-Pb-f-H-relax-3}, we see that
\begin{align}
    \eqref{eq-approx-error-Pb-relax-1} \leq o(1) + \frac{ 1+o(1) }{ n^{2pm\ell} \beta_{\mathcal H}^2 } \sum_{ \substack{ S \subset \mathsf{K}_n \\ [S] \in \mathcal H } } \sum_{ 1 \leq h \leq 2pm\ell } \sum_{ 2h>pk } \big( \lambda^{-1} n^{\frac{p}{4}} \big)^{2k} \mathrm{ENUM}(k,h) \,,  \label{eq-approx-error-Pb-relax-2}
\end{align}
where $\operatorname{ENUM}(k,h)$ was defined in \eqref{eq-def-ENUM}. Again, from Item~(4) in Lemma~\ref{lem-prelim-graphs}, we have
\begin{align*}
    \mathrm{ENUM}(k,h) \leq \binom{pm\ell}{h} \sum_{[I] \in \mathcal H} \frac{ n^{pm\ell-h} (pm\ell)! }{ |\mathsf{Aut}(I)| } \overset{\eqref{eq-def-beta-mathcal-H}}{=} \beta_{\mathcal H} n^{pm\ell-h} (pm\ell)! \binom{pm\ell}{h} \leq \beta_{\mathcal H} n^{pm\ell-h} (pm\ell)^{pm\ell} \,.
\end{align*}
Plugging this estimation into \eqref{eq-approx-error-Pb-relax-2} we get that
\begin{align*}
    \eqref{eq-approx-error-Pb-relax-2} &\leq o(1) + \frac{ 1+o(1) }{ n^{2pm\ell} \beta_{\mathcal H}^2 } \sum_{ \substack{ S \subset \mathsf{K}_n \\ [S] \in \mathcal H } } \sum_{ 1 \leq h \leq 2pm\ell } \sum_{ 2h>pk } \big( \lambda^{-1} n^{\frac{p}{4}} \big)^{2k} \beta_{\mathcal H} n^{pm\ell-h} (pm\ell)^{pm\ell} \\
    &\leq o(1) + \frac{ 1+o(1) }{ n^{pm\ell} \beta_{\mathcal H} } \sum_{ \substack{ S \subset \mathsf{K}_n \\ [S] \in \mathcal H } } \sum_{ 1 \leq h \leq 2pm\ell } \sum_{ 2h>pk } \big( \lambda^{-1} n^{\frac{p}{4}} \big)^{2k} (pm\ell)^{pm\ell}  \\
    &\overset{\eqref{eq-condition-strong-detection}}{\leq} o(1) + \frac{ 1+o(1) }{ n^{pm\ell-o(1)} \beta_{\mathcal H} } \sum_{ \substack{ S \subset \mathsf{K}_n \\ [S] \in \mathcal H } } \sum_{ 1 \leq h \leq 2pm\ell } \sum_{ 2h>pk } n^{-\frac{1}{2}+o(1)} = o(1) \,,
\end{align*}
leading to \eqref{eq-approx-error-Pb}.

\subsection{Proof of Lemma~\ref{lem-color-coding}}{\label{subsec:proof-lem-4.2}}

We first bound the total time complexity of Algorithm~\ref{alg:dynamic-programming}. The total number of subsets $[mp\ell]$ with cardinality $mp(k+1)+1$ is $\binom{mp\ell}{mp(k+1)+1}$. Fixing a color set $C$ with $|C|=mp(k+1)+1$ and $c \in C$, the total number of $C_1,C_2$ with $C_1 \cup C_2=C$ and $C_1 \cap C_2 = \{ c \}$ is bounded by $2^{mp(k+1)+1}$. Thus according to Algorithm~\ref{alg:dynamic-programming}, the total time complexity of computing $Y_{\ell-k}(x,y;c_1,c_2;U_{\ell-k} \cup \ldots \cup U_{\ell})$ is bounded by
\begin{align*}
    \sum_{1 \leq k \leq \ell-1} \binom{mp\ell}{mp(k+1)+1} \cdot 2^{mp(k+1)+1} \cdot n^2 \cdot n^{mp+1} \overset{\eqref{eq-condition-strong-detection}}{=} n^{mp+3+o(1)} \,.
\end{align*}
Finally, computing $g_H(\bm Y,\tau)$ according to $Y_{2}(y,x;c_2,c_1;U_2 \cup \ldots \cup U_{\ell})$ takes time
\begin{align*}
    2^{mp\ell} \cdot 2^{mp\ell} \cdot n^2 \cdot n^{mp+1} \overset{\eqref{eq-condition-strong-detection}}{=} n^{mp+1+o(1)} \,.
\end{align*}
Thus, the total running time of Algorithm~\ref{alg:dynamic-programming} is $O(n^{mp+3+o(1)})$.

Next, we prove the correctness of Algorithm~\ref{alg:dynamic-programming}. For any $1 \leq k \leq \ell-2$, define
\begin{align*}
    & Z_{\ell-k}(x,y;c_1,c_2,C;U_{\ell-k} \cup \ldots \cup U_{\ell}) \\
    =\ & \sum_{ \substack{ \psi: \{ v_{\ell-k}, \ldots v_{\ell}, v_1 \} \to [n] \\ \psi(v_{\ell-k})=x, \psi(v_1)=y }  } \sum_{ \substack{ W_{\ell-k}, \ldots W_{\ell} \subset \mathsf K_n: W_i \cong U_i \\ W_i \cap W_{i+1} = \psi(v_{i+1}) \\ x \in \mathsf L(W_{\ell-k}), y \in \mathsf L(W_{\ell}) } } \mathbf 1_{ \{ \tau(W_{\ell-k} \cup \ldots \cup W_{\ell})=C, \tau(x)=c_1,\tau(y)=c_2 \} } \prod_{\ell-k \leq i \leq \ell} \prod_{e \in E(W_i)} \bm Y_e \,.
\end{align*}
It is clear that
\begin{align*}
    Z_{\ell}(x,y;c_1,c_2,C;U_{\ell}) = \Lambda(x,y;c_1,c_2,C;U_{\ell})
\end{align*}
and 
\begin{align*}
    &Z_{\ell-k}(x,y;c_1,c_2,C;U_{\ell-k} \cup \ldots \cup U_{\ell}) \\
    =\ & \sum_{c \in C} \sum_{ \substack{ (C_1,C_2): C_1 \cup C_2 = C \\ c_1 \in C_1, c_2 \in C_2 \\ C_1 \cap C_2 = \{ c \} } } \sum_{ \substack{ z \in [n] \\ z \neq x,y } } \Lambda( x,z;c_1,c;C_1,[U_{\ell-k}] ) Z_{\ell-k+1}(y,z;c,c_2,C_2;U_{\ell-k+1} \cup \ldots \cup U_{\ell}) \,.
\end{align*}
Thus, we have $Z_{\ell-k+1}(y,z;c,c_2,C_2;U_{\ell-k+1} \cup \ldots \cup U_{\ell})=Y_{\ell-k+1}(y,z;c,c_2,C_2;U_{\ell-k+1} \cup \ldots \cup U_{\ell})$ for all $1 \leq k \leq \ell-2$. This yields that
\begin{align*}
    & \sum_{ \substack{ c_1,c_2 \in [mp\ell] \\ c_1 \neq c_2 } } \sum_{ \substack{ (C_1,C_2): C_1 \cup C_2 = [mp\ell] \\ C_1 \cap C_2 = \{ c_1,c_2 \} } } \sum_{ \substack{ x,y \in [n] \\ x \neq y } } \Lambda( x,y;c_1,c_2;C_1,[U_{1}] ) Y_{2}(y,x;c_2,c_1,C_2;U_{2} \cup \ldots \cup U_{\ell}) \\
    =\ & \sum_{ \substack{ \psi: \{ v_{1}, \ldots v_{\ell} \} \to [n] }  } \sum_{ \substack{ W_{1}, \ldots W_{\ell} \subset \mathsf K_n: W_i \cong U_i \\ W_i \cap W_{i+1} = \psi(v_{i+1}) } } \chi_\tau( V(W_{1} \cup \ldots \cup W_{\ell}) ) \prod_{1 \leq i \leq \ell} \prod_{e \in E(W_i)} \bm Y_e = \frac{1}{\mathfrak a(H)} g_H(\bm Y,\tau) \,.
\end{align*}

\subsection{Proof of Proposition~\ref{prop-approximate-recovery-statistics}}{\label{subsec:proof-prop-4.4}}

Recall \eqref{eq-def-Phi-i,j-mathcal-J} and \eqref{eq-def-widetilde-Phi-H}. We then have
\begin{align*}
    & \Big( \widetilde{\Phi}^{\mathcal J}_{i,j} - \Phi^{\mathcal J}_{i,j} \Big)^2 = \Bigg( \frac{ 1 }{ \lambda^{2m\ell} \beta_{\mathcal J} n^{\frac{pm\ell}{2}-1} } \sum_{ \substack{ S \subset \mathsf K_n: [S] \in \mathcal J \\ \mathsf{L}(S)=\{ i,j \} } } f_{S}(\bm Y) \cdot \Big( \frac{1}{t} \sum_{\mathtt k=1}^{t} \frac{1}{\varkappa} \chi_{\xi_{\mathtt k}}(V(S))-1 \Big)  \Bigg)^2 \\
    =\ & \sum_{ \substack{ S,K \subset \mathsf K_n: [S],[K] \in \mathcal J \\ \mathsf{L}(S)=\mathsf{L}(K)=\{ i,j \} } } \frac{ f_{S}(\bm Y) f_{K}(\bm Y) }{ \lambda^{4m\ell} \beta_{\mathcal J}^2 n^{2pm\ell-2} }  \Big( \frac{1}{t} \sum_{\mathtt k=1}^{t} \frac{1}{\varkappa} \chi_{\xi_{\mathtt k}}(V(S))-1 \Big) \Big( \frac{1}{t} \sum_{\mathtt k=1}^{t} \frac{1}{\varkappa} \chi_{\xi_{\mathtt k}}(V(K))-1 \Big) \,.
\end{align*}
We divide the above summation into three parts, the first part is
\begin{equation}{\label{eq-recovery-approx-error-part-1}}
    \sum_{ \substack{ S,K \subset \mathsf K_n: [S],[K] \in \mathcal J \\ \mathsf{L}(S)=\mathsf{L}(K)=\{ i,j \} \\ E(S) \cap E(K) \neq \emptyset } } \frac{ f_{S}(\bm Y) f_{K}(\bm Y) }{ \lambda^{4m\ell} \beta_{\mathcal J}^2 n^{2pm\ell-2} }  \Big( \frac{1}{t} \sum_{\mathtt k=1}^{t} \frac{1}{\varkappa} \chi_{\xi_{\mathtt k}}(V(S))-1 \Big) \Big( \frac{1}{t} \sum_{\mathtt k=1}^{t} \frac{1}{\varkappa} \chi_{\xi_{\mathtt k}}(V(K))-1 \Big) \,,
\end{equation}
the second part is
\begin{equation}{\label{eq-recovery-approx-error-part-2}}
    \sum_{ \substack{ S,K \subset \mathsf K_n: [S],[K] \in \mathcal J \\ \mathsf{L}(S)=\mathsf{L}(K)=\{ i,j \} \\ E(S) \cap E(K) = \emptyset \\ |V(S) \cap V(K)| \leq (mp\ell)^{0.1} } } \frac{ f_{S}(\bm Y) f_{K}(\bm Y) }{ \lambda^{4m\ell} \beta_{\mathcal J}^2 n^{2pm\ell-2} }  \Big( \frac{1}{t} \sum_{\mathtt k=1}^{t} \frac{1}{\varkappa} \chi_{\xi_{\mathtt k}}(V(S))-1 \Big) \Big( \frac{1}{t} \sum_{\mathtt k=1}^{t} \frac{1}{\varkappa} \chi_{\xi_{\mathtt k}}(V(K))-1 \Big) \,,
\end{equation}
and the third part is
\begin{equation}{\label{eq-recovery-approx-error-part-3}}
    \sum_{ \substack{ S,K \subset \mathsf K_n: [S],[K] \in \mathcal J \\ \mathsf{L}(S)=\mathsf{L}(K)=\{ i,j \} \\ E(S) \cap E(K) = \emptyset \\ |V(S) \cap V(K)| > (mp\ell)^{0.1} } } \frac{ f_{S}(\bm Y) f_{K}(\bm Y) }{ \lambda^{4m\ell} \beta_{\mathcal J}^2 n^{2pm\ell-2} }  \Big( \frac{1}{t} \sum_{\mathtt k=1}^{t} \frac{1}{\varkappa} \chi_{\xi_{\mathtt k}}(V(S))-1 \Big) \Big( \frac{1}{t} \sum_{\mathtt k=1}^{t} \frac{1}{\varkappa} \chi_{\xi_{\mathtt k}}(V(K))-1 \Big) \,.
\end{equation}
Note that by conditioning on $\bm Y$, by the randomness over $\{ \xi_{\mathtt k}:1 \leq \mathtt k \leq t \}$ we have 
\begin{align*}
    &\mathbb E\Big[ \Big( \frac{1}{t} \sum_{\mathtt k=1}^{t} \frac{1}{\varkappa} \chi_{\xi_{\mathtt k}}(V(S))-1 \Big) \Big( \frac{1}{t} \sum_{\mathtt k=1}^{\varkappa} \frac{1}{\varkappa} \chi_{\xi_{\mathtt k}}(V(K))-1 \Big) \Big] \\
    \leq\ & \mathbb E\Big[ \frac{1}{t^2} \sum_{\mathtt k=1}^{t} \frac{1}{\varkappa^2} \chi_{\xi_{\mathtt k}}(V(S)) \chi_{\xi_{\mathtt k}}(V(K)) \Big] = \frac{1}{t\varkappa} \leq 1 \,.
\end{align*}
Thus, we have (using $E_{\Pb}[ f_{S}(\bm Y) f_K(\bm Y) ] \geq 0$)
\begin{align}
    \mathbb E_{\Pb}\big[ \eqref{eq-recovery-approx-error-part-1} \big] \leq \frac{ 1 }{ \lambda^{4m\ell} \beta_{\mathcal J}^2 n^{2pm\ell-2} } \sum_{ \substack{ S,K \subset \mathsf K_n: [S],[K] \in \mathcal J \\ \mathsf{L}(S)=\mathsf{L}(K)=\{ i,j \} \\ E(S) \cap E(K) = \emptyset } } \mathbb E_{\Pb}[ f_{S}(\bm Y) f_K(\bm Y) ] \leq \delta^2 \,, \label{eq-recovery-approx-error-part-1-bound}
\end{align}
where the second inequality follows from Lemma~\ref{lem-very-technical-sec-3}. Now we bound \eqref{eq-recovery-approx-error-part-2}. Note that when $|V(S) \cap V(K)| = A \leq (mp\ell)^{0.1}$ 
\begin{align*}
    &\mathbb E\Big[ \Big( \frac{1}{t} \sum_{\mathtt k=1}^{t} \frac{1}{\varkappa} \chi_{\mathtt k}(V(S))-1 \Big) \Big( \frac{1}{t} \sum_{\mathtt k=1}^{t} \frac{1}{\varkappa} \chi_{\xi_{\mathtt k}}(V(K))-1 \Big) \Big] \\
    =\ & \mathbb E\Big[ \frac{1}{t^2} \sum_{\mathtt k=1}^{t} \frac{1}{\varkappa^2} \chi_{\xi_{\mathtt k}}(V(S)) \chi_{\xi_{\mathtt k}}(V(K)) \Big] -1 = \frac{1}{t^2\varkappa} \sum_{\mathtt k=1}^{t} \mathbb E\Big[ \chi_{\xi_{\mathtt k}}(V(S)) \mid \chi_{\xi_{\mathtt k}}(V(K))=1 \Big] -1  \\
    =\ & \frac{1}{t\varkappa} \cdot \frac{ (mp\ell-A)! }{ (mp\ell)^{mp\ell-A} } -1 \overset{\eqref{eq-condition-weak-recovery}}{=} [1+o(1)] \cdot [1+o(1)]-1 = o(1) \,.
\end{align*}
Thus, we have
\begin{align}
    \mathbb E_{\Pb}\big[ \eqref{eq-recovery-approx-error-part-2} \big] &\leq o(1) \cdot \frac{ 1 }{ \lambda^{4m\ell} \beta_{\mathcal J}^2 n^{2pm\ell-2} } \sum_{ \substack{ S,K \subset \mathsf K_n: [S],[K] \in \mathcal J \\ \mathsf{L}(S)=\mathsf{L}(K)=\{ i,j \} \\ E(S) \cap E(K) = \emptyset \\ |V(S) \cap V(K)| \leq (mp\ell)^{0.1} } } \mathbb E_{\Pb}[ f_{S}(\bm Y) f_{K}(\bm Y) ] \nonumber \\
    &\overset{\text{Lemma~\ref{lem-conditional-exp-given-endpoints}}}{=} o(1) \cdot \frac{ 1 }{ \lambda^{4m\ell} \beta_{\mathcal J}^2 n^{pm\ell-2} } \sum_{ \substack{ S,K \subset \mathsf K_n: [S],[K] \in \mathcal J \\ \mathsf{L}(S)=\mathsf{L}(K)=\{ i,j \} \\ E(S) \cap E(K) = \emptyset \\ |V(S) \cap V(K)| \leq (mp\ell)^{0.1} } } \lambda^{4m\ell} n^{-pm\ell} \nonumber \\
    &\leq o(1) \cdot \frac{1}{ \beta_{\mathcal J}^2 n^{2pm\ell-2} } \cdot \#\big\{ S,K \subset \mathsf K_n: [S],[K] \in \mathcal J ,\mathsf{L}(S)=\mathsf{L}(K)=\{ i,j \} \big\} = o(1) \,. \label{eq-recovery-approx-error-part-2-bound}
\end{align}
Finally we bound \eqref{eq-recovery-approx-error-part-3}. Again we have
\begin{align*}
    \mathbb E_{\Pb}\big[ \eqref{eq-recovery-approx-error-part-3} \big] &\leq \frac{ 1 }{ \lambda^{4m\ell} \beta_{\mathcal J}^2 n^{pm\ell-2} } \sum_{ \substack{ S,K \subset \mathsf K_n: [S],[K] \in \mathcal J \\ \mathsf{L}(S)=\mathsf{L}(K)=\{ i,j \} \\ E(S) \cap E(K) = \emptyset \\ |V(S) \cap V(K)| > (mp\ell)^{0.1} } }  \mathbb E_{\Pb}[f_{S}(\bm Y) f_{K}(\bm Y)] \nonumber \\
    &\overset{\text{Lemma~\ref{lem-conditional-exp-given-endpoints}}}{=} \frac{ 1 }{ \lambda^{4m\ell} \beta_{\mathcal J}^2 n^{pm\ell-2} } \sum_{ \substack{ S,K \subset \mathsf K_n: [S],[K] \in \mathcal J \\ \mathsf{L}(S)=\mathsf{L}(K)=\{ i,j \} \\ E(S) \cap E(K) = \emptyset \\ |V(S) \cap V(K)| > (mp\ell)^{0.1} } } \lambda^{4m\ell} n^{-pm\ell} \nonumber \\
    &\leq \frac{ 1 }{ \beta_{\mathcal J}^2 n^{2pm\ell-2} } \cdot \#\big\{ (S,K): S,K \subset \mathsf K_n: [S],[K] \in \mathcal J, |V(S) \cap V(K)| > (mp\ell)^{0.1} \big\} \,.
\end{align*}
Note that for all fixed $K \subset \mathsf K_n, [K] \in \mathcal J$, we can bound the enumeration of $S$ such that $[S] \in \mathcal J,|S\cap K|>(mp\ell)^{0.1}$ as follows: using Item~(2) of Lemma~\ref{lem-prelim-graphs}, for each $(mp\ell)^{0.1} \leq A \leq mp\ell+1$, the possible choices of $S$ such that $[S] \in \mathcal H$ and $|V(S) \cap V(K)|=A$ is bounded by
\begin{align*}
    & \sum_{ [J] \in \mathcal J } \#\big\{ S \subset \mathsf K_n: S \cong J, |V(S) \cap V(K)|=A \big\} \leq \sum_{ [J] \in \mathcal J } \binom{|V(K)|}{A} \frac{ n^{|V(J)|-A} |V(J)|^{A} }{ |\mathsf{Aut}(J)| } \\
    =\ & \sum_{ [J] \in \mathcal J } \binom{mp+1}{A} \frac{ n^{mp+1-A} (mp+1)^{A} }{ |\mathsf{Aut}(J)| } \overset{\eqref{eq-def-beta-mathcal-H}}{=} \beta_{\mathcal J} \cdot \binom{mp\ell+1}{A} n^{mp\ell+1-A} (mp\ell+1)^{A}  \,. 
\end{align*}
Thus we have
\begin{align*}
    & \#\Big\{ S \subset \mathsf K_n: [S] \in \mathcal J, |V(S) \cap V(K)| \geq (mp\ell)^{0.1} \Big\} \\
    \leq\ & \beta_{\mathcal J} \sum_{ (mp\ell)^{0.1} \leq A \leq mp\ell } \binom{mp\ell+1}{A} n^{mp+1-A} (mp\ell+1)^{A} \leq \beta_{\mathcal J} n^{mp\ell+2-(mp\ell)^{0.1}} \,.
\end{align*}
Thus, we have that
\begin{align}
    \mathbb E_{\Pb}\big[ \eqref{eq-recovery-approx-error-part-3} \big] &\leq \frac{1}{\beta_{\mathcal J} n^{ mp\ell-4+(mp\ell)^{0.1} } } \#\big\{ K \subset \mathsf K_n: [K] \in \mathcal J, \mathsf L(K) \big\} = o(1) \,.  \label{eq-recovery-approx-error-part-3-bound}
\end{align}
Combining \eqref{eq-recovery-approx-error-part-1-bound}, \eqref{eq-recovery-approx-error-part-2-bound} and \eqref{eq-recovery-approx-error-part-3-bound} yields that
\begin{align*}
    \mathbb E_{\Pb}\Big[ \big( \widetilde{\Phi}^{\mathcal J}_{i,j} - \Phi^{\mathcal J}_{i,j} \big)^2 \Big] \leq \delta^2 +o(1) \,.
\end{align*}

\subsection{Proof of Lemma~\ref{lem-color-coding-recovery}}{\label{subsec:proof-lem-4.5}}

We first bound the total time complexity of Algorithm~\ref{alg:dynamic-programming-recovery}. The total number of subsets $[mp\ell]$ with cardinality $mp(k+1)+1$ is $\binom{mp\ell}{mp(k+1)+1}$. Fixing a color set $C$ with $|C|=mp(k+1)+1$ and $c \in C$, the total number of $C_1,C_2$ with $C_1 \cup C_2=C$ and $C_1 \cap C_2 = \{ c \}$ is bounded by $2^{mp(k+1)+1}$. Thus according to Algorithm~\ref{alg:dynamic-programming-recovery}, the total time complexity of computing $Y_{\ell-k}(x,y;c_1,c_2;U_{\ell-k} \cup \ldots \cup U_{\ell})$ is bounded by
\begin{align*}
    \sum_{1 \leq k \leq \ell-1} \binom{mp\ell}{mp(k+1)+1} \cdot 2^{mp(k+1)+1} \cdot n^2 \cdot n^{mp+1} \overset{\eqref{eq-condition-weak-recovery}}{=} n^{2mp+3+o(1)} \,.
\end{align*}
Thus, the total running time of Algorithm~\ref{alg:dynamic-programming} is $O(n^{2mp+3+o(1)})$. The correctness of Algorithm~\ref{alg:dynamic-programming-recovery} is almost identical to the correctness of Algorithm~\ref{alg:dynamic-programming}, so we omit further details here for simplicity.

\bibliographystyle{alpha}

\end{document}